\documentclass[reqno,10pt,letterpaper]{amsart}
\usepackage{amsmath,amssymb,amsthm,amscd,graphicx,mathrsfs,url,dsfont,enumitem}
\usepackage[usenames,dvipsnames]{color}
\usepackage[colorlinks=true,linkcolor=Red,citecolor=Green]{hyperref}
\usepackage{dsfont}
\usepackage[super]{nth}
\usepackage[open, openlevel=2, depth=3, atend]{bookmark}
\hypersetup{pdfstartview=XYZ}
\usepackage{epigraph}
\usepackage{mathtools}
\usepackage{enumitem}
\usepackage[mathscr]{eucal}

\def\?[#1]{\textbf{[#1]}\marginpar{\Large{\textbf{??}}}}

\renewcommand{\tilde}{\widetilde}          
\DeclareMathSymbol{\leqslant}{\mathalpha}{AMSa}{"36} 
\DeclareMathSymbol{\geqslant}{\mathalpha}{AMSa}{"3E} 
\DeclareMathSymbol{\eset}{\mathalpha}{AMSb}{"3F}     
\renewcommand{\leq}{\;\leqslant\;}                   
\renewcommand{\geq}{\;\geqslant\;}                   
\renewcommand{\d}{\mathrm{d}}             

\numberwithin{equation}{section}

\newtheorem{theorem}{Theorem}[section]
\newtheorem{lemma}[theorem]{Lemma}
\newtheorem{proposition}[theorem]{Proposition}
\newtheorem{corollary}[theorem]{Corollary}

\theoremstyle{remark}
\newtheorem{remark}{Remark}

\theoremstyle{definition}
\newtheorem{definition}{Definition}[section]

\newcommand{\C}{\mathbb{C}}
\newcommand{\D}{\mathbb{D}}
\newcommand{\R}{\mathbb{R}}
\newcommand{\Z}{\mathbb{Z}}

\newcommand{\N}{\mathbb{N}}

\newcommand{\E}{\mathbb{E}}
\renewcommand{\P}{\mathbb{P}}

\newcommand{\T}{\mathbb{T}}
\renewcommand{\Im}{\mathrm{Im}}
\renewcommand{\Re}{\mathrm{Re}}

\newcommand{\la}{\lambda}

\newcommand{\cC}{\mathcal{C}}

\newcommand{\cN}{\mathcal{N}}

\newcommand{\ind}{\mathds{1}}

\newcommand{\cD}{\mathcal{D}}

\newcommand{\supp}{\mathrm{supp}}

\renewcommand{\hat}{\widehat}

\newcommand{\del}{\partial}

\newcommand{\bv}{\mathbf{v}}

\newcommand{\norm}[1]{\left\Vert #1\right\Vert}

\newcommand{\cQ}{\mathcal{Q}}

\newcommand{\bL}{\mathbf{L}}

\newcommand{\bH}{\mathbf{H}}

\newcommand{\eps}{\epsilon}
\newcommand{\mc}{\mathcal}
\newcommand{\cjg}{\langle}
\newcommand{\cjd}{\rangle}
\newcommand{\pl}{\partial}
\renewcommand{\varepsilon}{\epsilon}
\renewcommand{\log}{\ln}

\newcommand{\dd}{\text{\rm d}}
\newcommand{\bbar}{\overline}

\begin{document}

\title{The Virasoro structure and the scattering matrix for Liouville conformal field theory}

\author{Guillaume Baverez}
\address{Humboldt-Universit\"at zu Berlin, Institut f\"ur Mathematik, Rudower Chaussee 25, 12489 Berlin, Germany}
\email{guillaume.baverez@hu-berlin.de}

\author{Colin Guillarmou}
\address{Universit\'e Paris-Saclay, CNRS,  Laboratoire de math\'ematiques d'Orsay, 91405, Orsay, France.}
\email{colin.guillarmou@universite-paris-saclay.fr}

\author{Antti Kupiainen}
\address{University of Helsinki, Department of Mathematics and Statistics}
\email{antti.kupiainen@helsinki.fi}

\author{R\'emi Rhodes}
\address{Aix Marseille Univ, CNRS, I2M, Marseille, France\\
Institut Universitaire de France (IUF)}
\email{remi.rhodes@univ-amu.fr}

\author{Vincent Vargas}
\address{Universit\'e de Gen\`eve, Section de math\'ematiques, UNI DUFOUR, 24 rue du G\'en\'eral Dufour,
CP 64 1211 Geneva 4, Switzerland}
\email{Vincent.Vargas@unige.ch}

\begin{abstract} 
In this work, we construct a representation of the Virasoro algebra in the canonical Hilbert space associated to Liouville conformal field theory. 
The study of the Virasoro operators is performed through the introduction of a new family of Markovian dynamics associated to holomorphic vector fields defined in the disk. As an output, we show that the Hamiltonian of Liouville conformal field theory can be diagonalized 
through the action of the Virasoro algebra. This enables to show that the scattering matrix of the theory is diagonal and that the 
family of the so-called primary fields (which are eigenvectors of the Hamiltonian) admits an analytic extension to the whole complex plane, as conjectured in the physics literature. 
\end{abstract}

\maketitle

\section{Introduction and main results}

Liouville conformal field theory (LCFT) is a family of conformal field theories which arises in a wide variety of contexts ranging form random planar maps to 4d gauge theory. It was introduced by Polyakov in 1981 \cite{Pol} in his attempt to construct string theory; in this paper, LCFT appears under the form of a 2d version of a Feynman path integral. Recently, in a series of works, a rigorous probabilistic construction of the path integral was provided using the theory of the Gaussian Free Field and the theory of Gaussian multiplicative chaos: see \cite{DKRV} for the sphere, \cite{DRV} for the torus and \cite{GRV} for higher genus surfaces. In this theory, the main objects are the correlation functions of fields $V_\alpha$, denoted $\cjg V_{\alpha_1}(z_1)\dots V_{\alpha_m}(z_m)\cjd_{(\Sigma,g)}$, associated to $m$ marked points $z_1,\dots,z_m$ on a closed Riemannian surface $(\Sigma,g)$. 
A general formalism, called \emph{conformal bootstrap} in physics by Belavin-Polyakov-Zamolodchikov \cite{BPZ}, was developed in order to find explicit formulas for these correlation functions in terms of the $3$-point function on the sphere $\mathbb{S}^2$ and the so-called \emph{conformal blocks}, which are holomorphic functions of the points $z_j$ and the moduli of the Riemann surface $(\Sigma,g)$. The  conformal bootstrap method relies on the representation theory of an infinite dimensional Lie algebra of operators,  called the \emph{Virasoro algebra}, encoding the conformal symmetries of the system. It is formally generated by elements $({\bf L}_n)_{n\in \Z}$ together with  a central element denoted $c_L$, called \emph{the central charge}, with commutation relations   
\[ [\mathbf{L}_n,\mathbf{L}_m]=(n-m)\mathbf{L}_{n+m}+\frac{c_L}{12}(n^3-n)\delta_{n,-m}, \quad [\mathbf{L}_n, c_L]=0.   \]

In the context of the probabilistic  construction of LCFT,   the conformal bootstrap has  recently been established by the last 4 authors in the papers \cite{dozz,GKRV,GKRV1}. To implement this programme, 
a Hilbert space $\mathcal{H}=L^2(H^{-s}(\mathbb{T}),\mu_0)$ has been introduced where $H^{-s}(\T)$ is the Sobolev space of order $-s<0$ on the unit circle $\T=\{z\in \C\,|\, |z|=1\}$ 
and $\mu_0$ is the distribution of the Gaussian Free Field on the extended complex plane $\hat{\mathbb{C}}$ restricted to $\T$ times the Lebesgue measure for the zero Fourier mode on $\T$; the Hilbert space $\mc{H}$ is therefore a space of fields $\varphi$ on the circle. 

The dilation $s_{e^{-t}}: z\mapsto e^{-t}z$ in $\C$ for $t\geq 0$ is a conformal transformation that can be obtained as the flow of the holomorphic vector field $-z\pl_z$. By composing the Gaussian Free Field and Gaussian multiplicative chaos on the unit disk $\D\subset \C$ by $s_{e^{-t}}$, 
a Markovian dynamic has been constructed in \cite{GKRV} which produces a contraction semi-group $e^{-t{\bf H}}:\mc{H}\to \mc{H}$, for some operator  $\bH$  called the Hamiltonian  of LCFT. The diagonalisation in $\mc{H}$ of this Hamiltonian has then been performed in \cite{GKRV}, providing a Plancherel type formula which, when applied to  correlation functions, leads   to the explicit formulas involving the conformal blocks in  \cite{GKRV,GKRV1}. The  operator product expansion used in the physics literature can then be interpreted in terms of the 
 eigenbasis of $\bH$. The family of eigenfunctions is denoted 
 $(\Psi_{Q+iP,\nu,\tilde{\nu}})_{P,\nu,\tilde{\nu}}$ 
 where $P\in \R_+$ and $\nu,\tilde{\nu}$ are Young diagrams: the spectrum of $\bH$ is absolutely continuous and equal to $[Q^2/2,\infty)$ (with $Q>2$).
This family carries an important algebraic structure, namely that  the Young diagrams encode the action of the Virasoro generators on the highest weight vector $\Psi_{Q+iP}:=\Psi_{Q+iP,0,0}$:
\begin{equation}\label{descendantsLn}
\Psi_{Q+iP,\nu,\tilde{\nu}} = {\bf L}_{-\nu_k}\dots {\bf L}_{-\nu_1}\tilde{\bf L}_{-\tilde{\nu}_{k'}}\dots\tilde{\bf L}_{-\tilde{\nu}_1}\Psi_{Q+iP}
\end{equation}
where $\nu=(\nu_1,\dots,\nu_k)\in \N^k$ with $\nu_j\geq \nu_{j+1}$ for all $j$, $\tilde{\nu}=(\tilde{\nu}_1,\dots,\tilde{\nu}_{k'})\in \N^{k'}$ with $\tilde{\nu}_j\geq \tilde{\nu}_{j+1}$ for all $j$,  with ${\bf L}_n,\tilde{\bf L}_m$ being two commuting representations in $\mc{H}$ of the Virasoro algebra with central charge $c_L=1+6Q^2$, 
such that $[{\bf L}_n,\tilde{\bf L}_m]=0$. This can be compared to the way   the harmonic oscillator in $\R^n$ is diagonalised by the action of  the Heisenberg algebra on the constant function (the ground state), except that here one has a continuous family of highest weight states indexed by the  parameter $P>0$ in addition.

Yet the construction in \cite{GKRV} of these descendant states $\Psi_{Q+iP,\nu,\tilde{\nu}}$  does not make  the action  of the Virasoro algebra transparent. Indeed, in the Liouville CFT, the construction of the Virasoro generators ${\bf L}_n$, $\tilde{\bf L}_m$ is technically subtle due to singularities     coming from the potential  (i.e. Gaussian multiplicative chaos). Instead,  \cite{GKRV} bypassed the construction of the Virasoro generators using   scattering theory to construct the descendant states 
$\Psi_{Q+iP,\nu,\tilde{\nu}}$ and heavy algebraic computations (based on Ward identities) to derive algebraic relations between those eigenstates needed to prove the conformal bootstrap. This leads to a picture that is not completely satisfactory, not only  from the conceptual angle but also from the technical angle, because of  certain restrictions on the analytic extension of $ \Psi_{\alpha,\nu,\tilde{\nu}}$ from the line 
$\alpha \in Q+i\R$ to the complex plane. 

The physics literature links the Virasoro generators $({\bf L}_n)_{n\in \Z}$ (or $(\tilde{\bf L}_n)_{n\in \Z}$) to the dynamics obtained as  flows of holomorphic vector fields of the form $v(z)\partial_z$,  generalizing this way the picture drawn in the case of dilations (see for instance \cite[Ch. 9]{polch} or \cite{Gaw}). Yet, as stressed in \cite{Gaw}, a rigorous interpretation of these operators from this angle is far from being straightforward due to the intricate structure of their domains. In this work and inspired by these ideas, we give a probabilistic construction of the operators ${\bf L}_n$ using a family of Markovian dynamics generated by some flows of holomorphic vector fields, as suggested in the first author's doctoral dissertation \cite[Section 1.5.2]{these}. The crux of the matter is that the Markovian nature leads to perfectly well defined operators when seen as generators of contraction semigroups: this property is valid under some conditions of the holomorphic vector field, in which case the vector field will be called Markovian, and 
the case of non Markovian vector field is then treated via polarization type arguments. This leads to an explicit construction of the whole Virasoro algebra as unbounded operators ${\bf L}_n$  acting on $\mc{H}$, 
and we show that the descendant states 
$\Psi_{Q+iP,\nu,\tilde{\nu}}$ constructed by scattering theory in \cite{GKRV} are related to the highest weight state $\Psi_{Q+iP}$ by the relation 
\eqref{descendantsLn}.  This point of view is new and offers a different perspective from the algebraic approach usually developed in physics. Let us stress that the results of this paper rely heavily on the scattering theory part of \cite{GKRV} and therefore this paper does not simplify the arguments of \cite{GKRV}: it may  rather be seen as a conceptualization of \cite{GKRV}. Furthermore it develops new tools that will be needed to the future study of conformal blocks. Let us now explain this construction and its applications. 

\subsection{Hilbert space of Liouville CFT}

The Hilbert space $\mathcal{H}$ can be constructed explicitly as follows. Let $\T$ denote the standard unit circle in the complex plane and $\Z^{\ast}= \Z \setminus \lbrace 0 \rbrace $ be the set of non zero integers. We consider the space $\R \times \Omega_\T$ where $\Omega_\T$ is the set of real sequences $(x_n)_{n\geq 1}$ and $(y_n)_{n\geq 1}$. We introduce $(\varphi_{n})_{n \in \Z^{\ast}}$ (with $\varphi_{-n}=\bbar\varphi_{n}$) such that $\varphi_n:=\frac{1}{2\sqrt{n}}(x_n+iy_n)$. The Hilbert space $\mathcal{H}$ is then  the space $L^2(\R \times \Omega_\T)$ equipped with the measure $\d c \otimes \P_\T$ (and the standard Borel sigma algebra) where $\d c$ denotes the Lebesgue measure and $\P_\T$ is the Gaussian measure
 \begin{align}\label{Pdefin}
 \P_\T:=\bigotimes_{n\geq 1}\frac{1}{2\pi}e^{-\frac{1}{2}(x_n^2+y_n^2)}\dd x_n\dd y_n.
\end{align}
The inner product on $L^2(\R \times \Omega_\T)$ is denoted $\langle \cdot , \cdot\rangle_2$ and the norm $ \| \cdot \|_2$. Expectation with respect to $\P_\T$ will be denoted $\E[\cdot]$. Under $ \P_\T$ the random variable  
\begin{equation}\label{GFFcircle}
\varphi(\theta)=\sum_{n\not=0}\varphi_ne^{in\theta} 
\end{equation}
is the Gaussian Free Field (GFF) on the circle with covariance  
\begin{equation}
\E[\varphi({\theta})\varphi({\theta'})]=-\log |e^{i\theta}-e^{i\theta'}|.
\end{equation}
In fact, we will see $\varphi$ as a variable defined on a larger probability space $\Omega$ where the underlying probability measure will be denoted $\P$ (and expectation $\E[\cdot]$): see subsection \ref{subsectionGMC}. Also we will often identify the sequence $(c,(\varphi_{n})_{n \in \Z^{\star}})$ with the corresponding Fourier series $c+ \sum_{n\not=0}\varphi_ne^{in\theta} $, which is almost surely an element of $H^{-s}(\T):=\{\sum a_ne^{in\theta}\,|\,\sum_n (1+|n|)^{-2s}|a_n|^2<\infty\}$
 for $s>0$. We denote by $\norm{\cdot}_{H^{-s}(\T)}$ and $\langle\cdot,\cdot\rangle_{H^{-s}(\T)}$ the norm and inner-product defined by the sum in the definition of $H^{-s}(\T)$. The space $L^2(\R \times \Omega_\T)$ is then equivalent to $L^2(H^{-s}(\T),\mu_0)$ where $\mu_0=\varphi_*(\dd c\otimes \P_\T)$ ($\varphi_*$ denotes pushforward). 

\subsection{Semi-groups of Liouville CFT and representations of Virasoro algebra}
The Virasoro algebra ${\rm Vir}(c_L)$ is by definition a central extension of the Witt algebra, whose elements are represented by vector fields 
$-z^{n+1}\pl_z$ for $z\in \C\setminus \{0\}$ if $n\in \Z$. The central element, being denoted $c_L$ in our setting, will simply be a constant $c_L:=1+6Q^2$ where $Q\in(2,\infty)$. To represent ${\rm Vir}(c_L)$ into $\mc{H}$, we first consider a certain family of holomorphic vector fields 
$\bv$, called \emph{Markovian} and defined in the closed unit disk $\D:=\{z\in \C\,|\, |z|\leq 1\}$, of the form $\bv=v(z)\pl_z$ where $v(z)=-\sum_{n=-1}^\infty v_nz^{n+1}$, satisfying the property 
${\rm Re}(\bar{z}v(z))<0$ for $ z \in \T$. Each such Markovian vector field generates a flow of holomorphic transformations $f_t: \D\to \D$ solving 
$\pl_t f_t(z)=v(f_t(z))$ with initial condition $f_0(z)=z$, with a unique fixed point in $\D$ and such that $f_{t'}(\D)\subset f_t(\D)$ if $t'\geq t$: as $t\to +\infty$, $f_t$ contracts $\D$ to the unique zero of $v$ in $\D$, which is the attractor of the flow $f_t$. Up to composing $v$ with a M\"obius transformation, we shall choose $v$ so that $v_{-1}=v(0)=0$ and $v_0=-v'(0)=\omega>0$. 

 We consider $X= P\varphi+X_\D$ where $P\varphi$ is the harmonic extension of $\varphi$ on $\D$ and $X_\D$ an independent Dirichlet Gaussian Free Field. One can then define a semigroup $P_t$ on $L^2(\R \times \Omega_\T)$ as follows: for $F=F(c,\varphi)$ 
 depending only on finitely many variables $(\varphi_n)_{|n|\leq N}$ and decaying faster than any exponentials when $|c|\to \infty$ (this set of functions will be denoted $\mc{C}_{\rm exp}$ and is defined rigorously in Section \ref{freeham}), let 
\begin{equation} \label{FeynmanKacintro}
  P_t F(c,\varphi)   :=    |f'_t(0)|^{\frac{Q^2}{2}}    \E_\varphi\Big[  F\Big (  c+ \big(X \circ f_t  + Q \log  \frac {|f_t'|} {|f_t|}\big)\Big|_{\T}    \Big )     e^{- \mu e^{\gamma c} \int_{\D\setminus f_t(\D)}     \frac{e^{\gamma X(x)}} {|x| ^{\gamma Q} } \dd x  } \Big]  
 \end{equation}
where for a function $u$, $u|_{\T}  $ denotes restriction of the function to the unit circle $\T$ , $\E_\varphi[ \cdot ]$  denotes the conditional expectation with respect to $\varphi$, $\mu>0, \gamma\in (0,2)$ are some parameters and $Q=2/\gamma+\gamma/2$.  Since the Gaussian field $X$ is not defined pointwise, the  term $ \int_{\D\setminus f_t(\D)}    \frac{ e^{\gamma X(x)}} {|x| ^{\gamma Q} } \dd x $ is defined via a renormalisation procedure and yields a non trivial quantity, i.e. non zero, for $\gamma \in (0,2)$ (see \eqref{GMCsphere})\footnote{The condition $\gamma \in (0,2)$ in this paper comes from this non triviality result on Gaussian multiplicative chaos. It is a restriction of the probabilistic approach; indeed  LCFT in the physics literature is studied for all $\gamma \in \C$.  }; 
it is called the Gaussian multiplicative chaos measure of $\D\setminus f_t(\D)$. One can check that $P_t$ indeed defines a Markovian semigroup.

In the simple case where $v(z)=-z$, the flow has expression $f_t(z)=e^{-t}z$, and the semigroup $P_t$ coincides with the semigroup $e^{-t\bH}$, where ${\bf H}$ is the self-adjoint Hamiltonian defined and studied 
in \cite{GKRV}, with expression
\[\bH=-\frac{1}{2}\pl_c^2+\frac{Q^2}{2}+{\bf P}+\mu e^{\gamma c}V(\varphi)\]
where ${\bf P}$ and $V$ are unbounded non-negative operators on $L^2(\Omega_\T)$ defined by
\[{\bf P}:=2\sum_{n\geq 1} ( {\bf A}_n^*{\bf A}_n+\tilde{\bf A}_n^*\tilde{\bf A}_n), \quad V(\varphi) =\int_0^{2\pi}e^{\gamma \varphi(\theta)}\dd \theta \]
and ${\bf A}_n:=\frac{i}{2}\sqrt{n}(\pl_{x_n}-i\pl_{y_n})$, $\tilde{\bf A}_n:=\frac{i}{2}\sqrt{n}(\pl_{x_n}+i\pl_{y_n})$. 
Here, ${\bf P}$ has discrete spectrum equal to $\N$ and the operator $V$ has to be defined using Gaussian multiplicative chaos theory with a renormalisation procedure (see \eqref{GMCcircle}): it is a non-negative unbounded  operator, which becomes a multiplication by an $L^{p}(\Omega_\T)$ function for $p<2/\gamma^2$ when $\gamma<\sqrt{2}$. The operator ${\bf H}$ defines a quadratic form $\mc{Q}(F,F):=\cjg \bH F,F\cjd_2$ and we denote $\mc{D}(\mc{Q})$ its domain, and $\mc{D}'(\mc{Q})$ its dual.
We will also use the notation ${\bf A}_0=\tilde{\bf A}_0:=\frac{i}{2}(\pl_c+Q)$, ${\bf A}_{-n}:={\bf A}_n^*$ and $\tilde{\bf A}_{-n}:=\tilde{\bf A}_n^*$ if $n>0$, where the adjoint is taken with respect to the  scalar product on $\mc{H}$.

Our first main theorem is: 
\begin{theorem}\label{theoremfreefieldintro}
 Let $\bv=v(z)\pl_z$ be a Markovian vector field with $v(z)=-\sum_{n=0}^\infty v_{n}z^{n+1}$ and $v'(0)=-\omega$ for $\omega>0$ such that $v$ admits a holomorphic extension in a neighborhood of $\D$. If $\omega> 0$ is large enough\footnote{One could in fact work with a general complex $\omega$ with a large real part and still get the same results but we will stick to real $\omega$ for simplicity.},  depending on the sequence $(v_n)_{n\geq 0}$, then the operator $P_t$ is a contraction semi-group on $L^2(\R\times \Omega_\T)$ whose generator $\bH_\bv$ has the form 
\[\bH_\bv= \omega {\bf H}+ \sum_{n\geq 1}v_n\, \mathbf{L}_n  + \sum_{n\geq 1}\bbar{v_n} \, \widetilde{\mathbf{L}}_n,\]
where $\bH_\bv, {\bf L}_n, \tilde{\bf L}_n$ are bounded as linear maps $\mc{D}(\mc{Q})\to \mc{D}'(\mc{Q})$ but unbounded on $L^2(\R\times \Omega_\T)$. Moreover, they are given by the formula 
\begin{equation}\label{virasoro}
\begin{gathered}
\mathbf{L}_n:=-i(n+1)Q\mathbf{A}_n+\sum_{m\in\Z}:\mathbf{A}_{n-m}\mathbf{A}_m: +  \frac{\mu}{2} e^{\gamma c} \int_0^{2 \pi}  e^{in \theta} e^{\gamma \varphi(\theta)}  \dd \theta,\\
\widetilde{\mathbf{L}}_n:=-i(n+1)Q\widetilde{\mathbf{A}}_n+\sum_{m\in\Z}:\widetilde{\mathbf{A}}_{n-m}\widetilde{\mathbf{A}}_m: 
+  \frac{\mu}{2} e^{\gamma c} \int_0^{2 \pi}  e^{-i n \theta} e^{\gamma \varphi(\theta)}  \dd \theta 
\end{gathered}
\end{equation}
where the normal order product $:{\bf A}_n{\bf A}_m:$ is defined as ${\bf A}_n{\bf A}_m$ if $m>0$ or ${\bf A}_m{\bf A}_n$ if $n>0$.
 \end{theorem}

We recover this way the formulas announced by J. Teschner in \cite[Section 10]{Tesc1}.  Also, we notice that if $\bv_n=-z^{n+1}\pl_z$ and $\bv_0=-z\pl_z$, then we can recover ${\bf L}_n$ by the expression 
\[{\bf L}_n=\frac{1}{2}({\bf H}_{\omega \bv_0+\bv_n}-i{\bf H}_{\omega\bv_0+i\bv_n})-\frac{1}{2}\omega(1-i){\bf H}\]
where $\omega>0$ is chosen large enough so that $\omega \bv_0+\bv_n$ is Markovian. We define for $n>0$ 
\[{\bf L}_{-n}:= {\bf L}_n^* , \quad \tilde{\bf L}_{-n}:=\tilde{\bf L}_n^* \]
where the adjoints are taken with respect to the Hermitian product on $\mc{H}$ while acting on a dense set of regular functions:  
$\cjg{\bf L}_{-n}F,F'\cjd_2=\cjg F,{\bf L}_{n}F'\cjd_2$ for all $F,F'\in \mc{D}(\mc{Q})$.
We will explain below that $({\bf L}_n)_{n\in \Z}$ and $(\tilde{\bf L}_n)_{n\in \Z}$ are two commuting representations of ${\rm Vir}(c_L)$ into $\mc{H}$. A first look at these operators prevents us from composing them, even when acting on very regular functions $F$, which makes it difficult to define 
the commutators $[{\bf L}_n,{\bf L}_m]$ on a dense set. The technical difficulty in dealing with these operators was already stressed in \cite{Gaw}. We will however construct infinite dimensional vector spaces, called Verma modules, 
contained in $e^{N|c|}L^2(\R\times \Omega_\T)$ for $N>0$ which are preserved by all ${\bf L}_n$.

\subsection{Highest weight states,  descendant states and scattering coefficients} 
 In \cite{GKRV}, we constructed a family of eigenfunctions of $\bH$ for $\alpha<Q$
 \begin{equation}
\label{Psialphadef} 
 \Psi_\alpha(c,\varphi):=e^{(\alpha-Q)c} \E_\varphi\Big[ \exp\Big(-\mu e^{\gamma c}\int_{\D} |x|^{-\gamma\alpha }e^{\gamma X(x)}\dd x\Big)\Big]
\in e^{-\beta c_-}\mc{D}(\mc{Q})
\end{equation}
where $\beta>|{\rm Re}(\alpha)-Q|$ and $c_-=\min(c,0)$, satisfying
\[ \big(\bH -2 \Delta_{\alpha}\big)\Psi_{\alpha}=0, \quad \Psi_\alpha(c,\varphi)=e^{(\alpha-Q)c}+\mc{O}(e^{(\alpha-Q+\eps)c}) \textrm{ as }c\to -\infty,\]
where the conformal weight is given by 
\begin{equation}\label{defconformalweight}
\Delta_{\alpha}= \frac{\alpha}{2}(Q-\frac{\alpha}{2}).
\end{equation}
Moreover, we proved in \cite{GKRV} that 
\[ \alpha \mapsto \Psi_{\alpha}\]
admits an analytic continuation to the region $\{{\rm Re}(\alpha)\leq Q\}\setminus \cup_{j\geq 1}\{Q\pm i\sqrt{2j}\}$, it is continuous at the points 
$Q\pm i\sqrt{2j}$ with possible square root singularities. We also constructed using scattering theory a whole family 
\[\Psi_{\alpha,\nu,\tilde{\nu}}\] 
of eigenfunctions of $\bH$ with eigenvalues $\alpha(Q-\frac{\alpha}{2})+|\nu|+|\tilde{\nu}|$
for each Young diagram $\nu=(\nu_1,\dots,\nu_k)$, $\tilde{\nu}=(\tilde{\nu}_1,\dots,\tilde{\nu}_{k'})$, with $|\nu|=\sum_j\nu_j$ and $|\tilde{\nu}|=\sum_j\tilde{\nu}_j$, we showed that they are analytic in an open set $W_{\ell}\subset \{{\rm Re}(\alpha)\leq Q\}\setminus \cup_{j\geq 0}\{Q\pm i\sqrt{2j}\}$ 
containing $(Q+i\R)\setminus \cup_{j\geq 0}\{Q\pm i\sqrt{2j}\}$, where $\ell=|\nu|+|\tilde{\nu}|$. Moreover the following Plancherel type formula holds: for $u,u'\in \mc{H}$
\begin{equation}\label{diagonalisation}
\cjg u,u'\cjd_{2}= \frac{1}{2\pi}\sum_{\substack{\nu,\nu',\tilde{\nu},\tilde{\nu}'\in \mc{T}\\
|\nu|=|\nu'|, |\tilde{\nu}|=|\tilde{\nu}'|}}\int_0^\infty \cjg u,\Psi_{Q+iP,\nu,\tilde{\nu}}\cjd_2 \cjg \Psi_{Q+iP,\nu',\tilde{\nu}'},u'\cjd_2 F_{Q+iP}^{-1}(\nu,\nu')F_{Q+iP}^{-1}(\tilde{\nu},\tilde{\nu}')\dd P \end{equation}
 where $\mc{T}$ denotes the set of Young diagrams, and $(F_{Q+iP}^{-1}(\nu,\nu'))_{\nu,\nu'\in \mc{T}_n}$ 
 are positive definite matrices for each $n>0$ if $\mc{T}_n=\{\nu \in \mc{T}\,|\, |\nu|=n\}$ is the set of Young diagrams of size $n$. In our convention $\mc{T}$ contains $\{0\}$ and $\Psi_{Q+iP,0,0}=\Psi_{Q+iP}$.
 
In this paper, we prove that the $\Psi_{\alpha,\nu,\tilde\nu}$ can be obtained from $\Psi_\alpha$ by 
applying the direct sum of two copies of the Virasoro algebra ${\rm Vir}(c_L)$. 
\begin{theorem}\label{descendantsandLn} The following properties hold:\\
 1) The function $\Psi_{\alpha}$ defined by \eqref{Psialphadef} admits an analytic extension to $\alpha \in \C$, with values in $e^{-\beta c_-}\mc{D}(\mc{Q})$ for 
any $\beta>|{\rm Re}(\alpha)-Q|$, and for each $n\in \Z$, ${\bf L}_{n}\Psi_\alpha$ is well-defined as an element in $e^{-\beta c_-}\mc{D}(\mc{Q})$  for
$\beta>|{\rm Re}(\alpha)-Q|$, it is analytic in $\alpha$ and equal to $0$ when $n>0$.\\
2) The functions $\Psi_{\alpha,\nu,\tilde{\nu}}$ appearing in \eqref{diagonalisation} admit an analytic extension to 
$\alpha\in \C$ as elements of  $e^{-\beta c_-}\mc{D}(\mc{Q})$ for $\beta>|{\rm Re}(\alpha)-Q|$, 
 and for each $n\in \Z$, ${\bf L}_{n}\Psi_{\alpha,\nu,\tilde{\nu}}$ is well-defined as an element in $e^{-\beta c_-}\mc{D}(\mc{Q})$  for
$\beta>|{\rm Re}(\alpha)-Q|$, analytic in $\alpha$, and is equal to $0$ when $n>|\nu|+|\tilde{\nu}|$.\\
3) These functions are related by the formula
\[ \Psi_{\alpha,\nu,\tilde{\nu}} = {\bf L}_{-\nu_k}\dots {\bf L}_{-\nu_1}\tilde{\bf L}_{-\tilde{\nu}_{k'}}\dots\tilde{\bf L}_{-\tilde{\nu}_1}\Psi_{\alpha}\]
and for each $n\in \Z$, and $\alpha\in \C \setminus (Q\pm (\frac{2}{\gamma}\N_0+\frac{\gamma}{2}\N_0))$
\[{\bf L}_{n} \Psi_{\alpha,\nu,\tilde{\nu}}\in {\rm span}\{ \Psi_{\alpha,\nu',\tilde{\nu}}\,|\, \nu' \in \mc{T},\, |\nu|-n=|\nu'|
\}.\]
4) They satisfy the functional equation for all $\alpha\in \C \setminus (Q\pm (\frac{2}{\gamma}\N_0+\frac{\gamma}{2}\N_0))$ and all $\nu,\tilde{\nu}\in \mc{T}$
\[ \Psi_{2Q-\alpha,\nu,\tilde{\nu}}=R(2Q-\alpha)\Psi_{\alpha,\nu,\tilde{\nu}}\]
where $R(\alpha)$ is the scattering coefficient given by 
\[R(\alpha)=-\Big(\pi \mu \frac{\Gamma(\frac{\gamma^2}{4})}{\Gamma(1-\frac{\gamma^2}{4})}\Big)^{2\frac{(Q-\alpha)}{\gamma}}\frac{\Gamma(-\frac{\gamma(Q-\alpha)}{2})\Gamma(-\frac{2(Q-\alpha)}{\gamma})}{\Gamma(\frac{\gamma(Q-\alpha)}{2})\Gamma(\frac{2(Q-\alpha)}{\gamma})}.\]
5) For $P>0$, $\Psi_{Q+iP,\nu,\tilde{\nu}}$ satisfy the asymptotic expansion as $c\to -\infty$
\[ \Psi_{Q+iP,\nu,\tilde{\nu}}(c,\varphi)=e^{iPc}\mc{Q}_{Q+iP,\nu,\tilde{\nu}}(\varphi)+R(Q+iP)e^{-iPc}\mc{Q}_{Q-iP,\nu,\tilde{\nu}}(\varphi) +G_{Q+iP}(c,\varphi)\]
where $G_{Q+iP}(c,\varphi)\in \mc{D}(\mc{Q})\subset L^2$ and $\mc{Q}_{Q\pm iP,\nu,\tilde{\nu}}\in L^2(\Omega_\T)$ are particular eigenfunctions of ${\bf P}$ with eigenvalues $|\nu|+|\tilde{\nu}|$ (see Section \ref{Vermafreefield} for their definition).
\end{theorem}

\begin{remark}
The same statement as 1), 2), 3) holds with $\tilde{\bf L}_n$ in place of ${\bf L}_n$ except that  in 3) one has
\[\tilde{\bf L}_{n} \Psi_{\alpha,\nu,\tilde{\nu}}\in {\rm span}\{ \Psi_{\alpha,\nu,\tilde{\nu}'}\,|\, \tilde{\nu}' \in \mc{T},\, |\tilde{\nu}|-n=|\tilde{\nu}'|
\}.\]

\end{remark}
We notice that Statement 5) means that the scattering matrix is essentially (up to identifying $\mc{Q}_{Q+iP,\nu,\tilde{\nu}}$ with 
$\mc{Q}_{Q-iP,\nu,\tilde{\nu}}$) a constant $R(Q+iP)$ times the identity map, a feature that is important in physics and that is related to the integrability of LCFT (see \cite{Tesc1} for instance). Indeed, the diagonal form of the scattering matrix  is reminiscent of similar results in integrable field theories such as the Sine-Gordon model (see \cite{FSK}). The construction of the scattering matrix for the Liouville CFT was done in \cite{GKRV} using scattering theory: 
for fixed frequency $p>0$ corresponding to the equation $({\bf H}-2\Delta_{Q+ip})$, this is a linear operator acting on 
a finite dimensional subspace $L^2(\Omega_{\mathbb{T}})$ encoding the coefficients in the asymptotics at $c\to -\infty$ of the generalized eigenstates 
\[ \{\Psi_{Q+iP,\nu,\tilde{\nu}} \,|\, P\geq 0, \nu,\tilde{\nu}\in\mc{T}, 2\Delta_{Q+iP}+|\nu|+|\tilde{\nu}|= 2\Delta_{Q+ip}\}.\]
To prove that this operator is diagonal in the basis associated to $\Psi_{Q+iP,\nu,\tilde{\nu}}$, the role of the Virasoro algebras $({\bf L}_n)_n ,(\tilde{{\bf L}}_n)$ is essential: 
the idea is that this algebra generates all the generalized eigenfunctions $\Psi_{\alpha,\nu,\tilde{\nu}}$ from the primary field $\Psi_{\alpha}$, implying some relations on the scattering matrix, namely it says that the asymptotics at $c\to-\infty$ of $\Psi_{Q+iP,\nu,\tilde{\nu}}$ are purely characterized 
by those of the primary field $\Psi_{Q+iP}$.  

This result allows us in particular to define the vector space $\mc{W}_\alpha:={\rm span}\{\Psi_{\alpha,\nu,\tilde{\nu}}\,|\, \nu,\tilde{\nu}\in \mc{T}\}$ and to show that the operators ${\bf L}_n,\tilde{{\bf L}}_m$ preserve $\mc{W}_\alpha$, with the commutation relations  
\[ [\mathbf{L}_n,\mathbf{L}_m]=(n-m)\mathbf{L}_{n+m}+\frac{c_L}{12}(n^3-n)\delta_{n,-m}, \quad [\tilde{\bf L}_n,\tilde{\bf L}_m]=(n-m)\tilde{\bf L}_{n+m}+\frac{c_L}{12}(n^3-n)\delta_{n,-m}\]
and $[{\bf L}_n,\tilde{\bf L}_m]=0$.
The space $\mc{W}_\alpha$ is a Verma module for a highest weight representation 
of the direct sum ${\rm Vir}(c_L)\oplus {\rm Vir}(c_L)$ of two Virasoro algebras with central charge $c_L=1+6Q^2$, it can be splitted as a tensor product of two Verma modules $\mc{W}_{\alpha}=\mc{V}_\alpha\otimes \bbar{\mc{V}}_\alpha$, each one associated to the representations ${\bf L}_n$ and $\tilde{\bf L}_n$ of ${\rm Vir}(c_L)$. We refer to Section \ref{VermaLiouville} for more details on these aspects related to representation theory.

\subsection{Future applications}
This approach will be instrumental in several subsequent works in the context of the conformal bootstrap for LCFT. First it will serve as a key tool in a program aiming at establishing the conformal bootstrap for open surfaces: the analyticity of the eigenstates over the full complex plane as well as their asymptotics is needed to analyze the contribution of the one point bulk correlator, the FZZ structure constant introduced in \cite{ARS}. This issue already appears in \cite{wu} in the case of the annulus and will be further developed in a forthcoming work treating the case of general open surfaces. 

Moreover, the conformal blocks appearing in the conformal bootstrap formulae depend on a choice of pant decomposition of the underlying Riemann surface. As stated in  \cite{GKRV1}, the conformal blocks depend on the splitting curves used to obtain the pant decomposition. Therefore they are not yet fully understood as analytic functions on Teichm\"uller space as they depend on the choice of local coordinates (the splitting curves). In a forthcoming work, we will use the Markovian dynamics introduced in this manuscript to show that the conformal blocks only depend in a quite simple way on local deformations of the splitting curves. This will show that conformal blocks, if viewed as sections of some line bundle, only depend on the homotopy classes of the splitting curves, and therefore are well-defined on Teichm\"uller space. As an intermediate step, we will identify the annulus amplitudes with the kernels of the semigroups introduced in the present work. This is in agreement with Segal's axioms \cite{Gaw} and generalizes \cite[Section 6]{GKRV1}.

 On the other hand, conformal blocks are not expected to be single-valued on moduli space (contrarily to correlation functions), but it is conjectured that there exists a representation of the mapping class group describing the variation of the blocks under a change of pants decomposition. This is related to a conjecture of Verlinde identifying the space of conformal blocks of Liouville theory with the Hilbert space of quantum Teichm\"uller theory as isomorphic representations of the mapping class group (see \cite{Teschner04} for an overview). In geometric terms, the Virasoro generators introduced in this article should allow to define a (projectively) flat connection on the bundle of conformal blocks, whose monodromy is described by the above-mentioned representation; this is work in progress. Our interpretation of the Virasoro operators as generators of diffusion processes is reminiscent to a type of connection on the moduli space introduced by Hitchin \cite{Hitchin}, for which parallel transport is given by the heat kernel. Finally, the full analycity of the descendants  $\Psi_{\alpha,\nu,\tilde{\nu}}$ will be used to bridge the gap between the probabilistic construction of LCFT and the Vertex Operator Algebra approach.

\subsubsection{Organization of the paper}
The paper is organized as follows. In section 2, we introduce the main notations and the framework of the paper; in this section, we also construct the semigroup associated to the free field case $\mu=0$ and we study its stationary measure $\mu_h$. In section 3, we construct the semigroup associated to the general case, i.e. with the Liouville potential, via a Feynman-Kac type formula. Also, we use the dynamics to show that the basis $\Psi_{\alpha, \nu,\tilde \nu}$ can be constructed via the action of the Virasoro algebra: this is the content of Proposition \ref{intertwining_for_descendants}. In section 4, we use these results to study the scattering coefficients and show the analytic extension and functional equations for the eigenstates $\Psi_{\alpha,\nu,\tilde{\nu}}$.

{\bf Acknowledgements:} 
We thank the referees for their careful reading and comments in order to improve the paper.
Colin Guillarmou acknowledges the support of European Research Council (ERC) Consolidator grant 725967 and Antti Kupiainen the support of the ERC Advanced Grant 741487. R\'emi Rhodes is
partially supported by the Institut Universitaire de France (IUF). R\'emi Rhodes and Vincent Vargas acknowledge
the support of the French National Research Agency (ANR)  ANR-21-CE40-003.

\section{Segal-Sugawara representation and the Gaussian Free Field}\label{sec:free_field}

In this Section, we introduce the notion of Markovian holomorphic vector fields on the unit disk and show that they generate Markovian processes and 
semi-groups on a Hilbert space by composing the Gaussian Free Field by the flow of these vector fields. The generators of these 
semi-groups induce a unitary representation of the Virasoro algebra with central charge $1+6Q^2$ in a Hilbert space $\mc{H}$. 
This approach provides a probabilistic construction of the so-called \emph{Segal-Sugawara representation} of Virasoro algebra.

\subsection{Free Hamiltonian and Virasoro generators for the free field}\label{freeham}

Let us introduce some material taken from \cite[Section 4]{GKRV}. 
As mentionned in the Introduction, we consider the space $\Omega_\T=(\R^2)^\N$ with the Gaussian probability measure 
 \[
 \P_\T:=\bigotimes_{n\geq 1}\frac{1}{2\pi}e^{-\frac{1}{2}(x_n^2+y_n^2)}\dd x_n\dd y_n.
\]
The space $\mathcal{H}$ is then defined to be the space $L^2(\R \times \Omega_\T)$ equipped with the measure $\mu_0:=\dd c \times \P_\T$ (and the standard Borel sigma algebra) where $\dd c$ denotes the Lebesgue measure. Consider the real valued random variable on $\T$
\[
\varphi(\theta)=\sum_{n\not=0}\varphi_ne^{in\theta} , \quad \varphi_n=\frac{x_n+iy_n}{2\sqrt{n}} \textrm{ for }n>0 .
\]
Recall that we will see $\varphi$ as a variable defined on a larger probability space $\Omega$ where the underlying probability measure will be denoted $\P$ (and expectation $\E[\cdot]$): see subsection \ref{subsectionGMC}. Note that $\varphi\in H^{-s}(\T)$ for all $s>0$ almost surely. The space $H^{-s}(\T)$ can be equipped with the pushforward measure $\mu_0$ (still denoted $\mu_0$ for simplicity) by the map 
$(c,(\varphi_n)_n)\in \R\times \Omega_\T\mapsto c+\varphi\in H^{-s}(\T)$ so that $\mc{H}\simeq L^2(H^{-s}(\T),\mu_0)$. 

We set $\partial_n:= \partial_{\varphi_n}= \sqrt{n} (\partial_{x_n}-i \partial_{y_n})$ (and $\partial_{-n}:= \partial_{\varphi_{-n}}= \sqrt{n} (\partial_{x_n}+i \partial_{y_n})$) and introduce $\mathcal{S}$ the set of smooth functions of the form $F(x_1,y_1, \cdots, x_N,y_N)$ for some $N \geq 1$ where   $F\in C^\infty((\R^2)^N)$ with at most polynomial growth at infinity for $F$ and its derivatives. We consider the  vector space $\mc{C}_{\rm exp}$ 
of smooth functions $F:\R\times \Omega_\T\to \C$ so that there is $N>0$ such that $F(c,\varphi)=F(c,x_1,y_1, \cdots, x_N,y_N)$ for all $c\in\R$, and (with $\N_0:=\N\cup\{0\}$)
\[ \forall k\in \N_0, \forall \alpha,\beta \in \N_0^{N}, \exists L\geq 0, \forall M\geq 0,  \exists C>0, 
\quad  |\pl_c^k \pl_{x}^\alpha\pl_y^\beta F(c,\varphi)|\leq Ce^{-M|c|}\cjg \varphi\cjd_N^{L}
\]
if $\pl_{x}^\alpha=\pl_{x_1}^{\alpha_1}\dots \pl_{x_N}^{\alpha_N}$ and $\cjg \varphi\cjd_N=(1+\sum_{|n|\leq N}|\varphi_n|^2)^{1/2}$ is the Japanese bracket.

The space $\mc{C}_{\rm exp}$ is a dense vector subspace of $\mc{H}$ and 
the free Hamiltonian is the Friedrichs extension of the operator defined on $\mc{C}_{\rm exp}$ by the expression
\begin{equation}\label{defH0}
{\bf H}^0=\frac{1}{2}(-\pl_c^2+Q^2+2{\bf P}) ,\quad \textrm{ with }{\bf P}=\sum_{n=1}^\infty n(\pl_{x_n}^*\pl_{x_n}+\pl_{y_n}^*\pl_{y_n}),
\end{equation}
where $\pl_{x_n}^*=-\partial_{x_n}+x_n$ denotes the adjoint of $\pl_{x_n}$ with respect to $\mu_0$. Recall that the construction of the Friedrichs extension relies on the quadratic   form   denoted $\mc{Q}_0(F,F):=\cjg{\bf H}^0F,F\cjd_2$ for $F\in \mc{C}_{\rm exp}$. By \cite[Proposition 4.3]{GKRV}, 
the quadratic form is closable with closure still denoted $\mc{Q}_0$ and domain denoted $\mc{D}(\mc{Q}_0)$, which becomes a Hilbert space when equipped with the norm $\|F\|_{\mc{Q}_0}:=\sqrt{\mc{Q}_0(F,F)}$. 
This quadratic form produces a self-adjoint extension (Friedrichs extension) for ${\bf H}^0$ on a domain $\mc{D}(\mc{H}^0)\subset \mc{D}(\mc{Q}_0)$. 
Moreover ${\bf H}^0$ extends to a bounded operator ${\bf H}^0:\mc{D}(\mc{Q}_0)\to \mc{D}'(\mc{Q}_0)$ where $\mc{D}'(\mc{Q}_0)$ 
is the dual to $\mc{D}(\mc{Q}_0)$. 
 
Next we recall the construction of the representation of two copies of the Virasoro algebra as operators on the space $\mc{C}_{\rm exp}$. Consider first
 the operators for $n\geq 1$
\begin{equation}\label{defAn}
\mathbf{A}_n= \tfrac{i\sqrt{n}}{2}(\pl_{x_n}-i\pl_{y_n}),\, \,  \mathbf{A}_{-n}={\bf A}_n^*, \qquad \widetilde{\mathbf{A}}_n=  \tfrac{i\sqrt{n}}{2}(\pl_{x_n}+i\pl_{y_n}),\,\,
 \widetilde{\mathbf{A}}_{-n}=\widetilde{\bf A}_n^*,
 \end{equation} 
where the adjoint here is formal and has to be understood in the sense of pairing on $\mc{C}_{\rm exp}$ with respect to the measure $\mu_0$, and let 
\[\mathbf{A}_0=\widetilde{\mathbf{A}}_0=\tfrac{i}{2}(\partial_c+Q).\] 
In fact, considering the notation $\partial_n:= \frac{\partial}{\partial \varphi_n}$ (where $\partial_0:= \partial_c$), we see that, for $n \not = 0$,
\begin{equation}\label{defAnbis}
\mathbf{A}_n= \tfrac{i}{2} (\partial_n+ 2n \varphi_{-n}1_{n <0}) ,\, \, \, \, \widetilde{\mathbf{A}}_n= \tfrac{i}{2} (\partial_{-n}+ 2n \varphi_{n}1_{n <0}).
 \end{equation} 
Recall the following commutation relations for $n,m$
\begin{equation}\label{commutationAn}
 [\mathbf{A}_n,\mathbf{A}_m]= \frac{n}{2} \delta_{n,-m}, \quad [ \widetilde{\mathbf{A}}_n, \widetilde{\mathbf{A}}_m]= \frac{n}{2} \delta_{n,-m}, \quad  [\mathbf{A}_n, \widetilde{\mathbf{A}}_m]=0 .
\end{equation}
Then for $n\in \Z$ define the operators acting on $\mc{C}_{\rm exp}$ 
\[\mathbf{L}_n^0=-i(n+1)Q\mathbf{A}_n+\sum_{m\in\Z}:\!\mathbf{A}_{n-m}\mathbf{A}_m\!:, \quad \widetilde{\mathbf{L}}_n^0:=-i(n+1)Q\widetilde{\mathbf{A}}_n+\sum_{m\in\Z}:\!\widetilde{\mathbf{A}}_{n-m}\widetilde{\mathbf{A}}_m\!: 
\]
and where the normal ordering is defined by $:\!\mathbf{A}_n\mathbf{A}_m\!\!:\,=\mathbf{A}_n\mathbf{A}_m$ if $m>0$ and $\mathbf{A}_m\mathbf{A}_n$ if $n>0$ (i.e. annihilation operators are on the right). In particular, for $n \not = 0$ the only $c$ derivative in ${\bf L}_n^0$ is the term involving ${\bf A}_0$, which is given by 
${\bf A}_n{\bf A}_0+{\bf A}_0{\bf A}_n=2{\bf A}_n{\bf A}_0$.
Notice that $({\bf L}_n^0)^*= {\bf L}_{-n}^0$ in the weak sense, i.e. for each $F_1,F_2\in \mc{C}_{\rm exp}$ the following relation holds  
\[ \cjg {\bf L}_n^0F_1,F_2\cjd_2 =  \cjg F_1, {\bf L}_{-n}^0 F_2\cjd_2.\]
A direct computation allows us to express the Hamiltonian in terms of the $n=0$ Virasoro elements
\[ {\bf H}^0={\bf L}_0^0+ \widetilde{\mathbf{L}}_0^0=\frac{1}{2}(-\pl_c^2+Q^2)+2\sum_{n=1}^\infty ( {\bf A}_n^*{\bf A}_n+\widetilde{\bf A}_n^*\widetilde{\bf A}_n)\]
and one can derive the following commutation relations by using \eqref{commutationAn} (for $c_L=1+6Q^2)$ (see \cite[Prop 2.3]{kac} )
\[ [\mathbf{L}^0_n,\mathbf{L}^0_m]=(n-m)\mathbf{L}^0_{n+m}+\frac{c_L}{12}(n^3-n)\delta_{n,-m}, \quad [\widetilde{\mathbf{L}}^0_n,\widetilde{\mathbf{L}}^0_m]=(n-m)\widetilde{\mathbf{L}}^0_{n+m}+\frac{c_L}{12}(n^3-n)\delta_{n,-m}.\] 
Finally, using expression \eqref{defAnbis} and little bit of algebra yields for $n \geq 0$ (for $n \leq -1$, one has to integrate by parts the relation below)
\begin{equation}\label{newexpressionLno}
\mathbf{L}_n^0=  \tfrac{Q^2}{4}1_{n=0}+  \tfrac{Qn}{2} \partial_n   - \tfrac{1}{4} \sum_{m \in \Z} \partial_{n-m} \partial_{m}+ \sum_{m \geq 1} m \varphi_{n+m} \partial_{n+m}
\end{equation}
and a similar relation for $\widetilde{\mathbf{L}}_n^0$. Notice that $\mc{Q}_0$ can be written in terms of the ${\bf A}_n$'s as 
\[\begin{split} 
\mc{Q}_0(F,F)= \frac{1}{2}(\|\pl_cF\|_2^2+Q^2\|F\|^2_2)+2\sum_{n=1}^\infty (\|{\bf A}_nF\|^2_{2}+\|\widetilde{\bf A}_nF\|^2_{2})= 2\|{\bf A}_0F\|_2^2+2\sum_{n=1}^\infty (\|{\bf A}_nF\|^2_{2}+\|\widetilde{\bf A}_nF\|^2_{2}).
\end{split}\]
We first show the following: 
\begin{lemma}\label{boundonLn}
There exist some constant $C>0$ such that for all $n \geq 1$ and all $F,G\in\mc{C}_{\rm exp}$ we have  
\begin{equation}\label{bilinearboundLn}
 | \cjg {\bf L}^0_n F,G \cjd_2|\leq C(1+n)^{3/2}\|F\|_{\mc{Q}_0}\|G\|_{\mc{Q}_0}   
 \end{equation}
and the operator ${\bf L}_n^0$ extends as a bounded operator ${\bf L}_n^0: \mc{D}(\mc{Q}_0)\to \mc{D}'(\mc{Q}_0)$. The same statement holds for $\widetilde{{\bf L}}_n^0$. 
\end{lemma}
\begin{proof}
First we note that the extension of ${\bf L}_n^0$ follows directly from the estimate \eqref{bilinearboundLn}.
Using Cauchy-Schwarz, we have for $F,G\in\mc{C}_{\rm exp}$ 
\begin{align*}
 | \cjg {\bf L}_n^0 F,G \cjd_2   | & =   \Big| -i (n+1) Q \cjg {\bf A}_n F , G \cjd_2  +  2 \cjg {\bf A}_0 F,  {\bf A}_{-n} G \cjd_2 + \sum_{k=1}^{n-1} \cjg {\bf A}_{n-k}F , {\bf A}_{-k}G \cjd_2 + 2 \sum_{m \geq 1} \cjg {\bf A}_{n+m} F, {\bf A}_m G \cjd_2   \Big|   \\
& \leq  \Big (  (n+1) Q \|{\bf A}_nF\|^2_2+ 2 \|{\bf A}_0F\|^2_2  +  \sum_{k=1}^{n-1}   \| {\bf A}_kF \|^2_2+ 2   \sum_{m \geq 1} \|{\bf A}_{n+m} F \|^2_2 \Big)^{\frac{1}{2}}    \\
& \quad \times \Big (  (n+1) Q \| G \|^2_2  +  2 \| {\bf A}_{-n} G \|^2_2 + \sum_{k=1}^{n-1} \| {\bf A}_{-k}G \|^2_2 + 2 \sum_{m \geq 1} \| {\bf A}_m G \|^2_2\Big)^{\frac{1}{2}}  
\end{align*}
and using the commutation relation $[{\bf A}_k,{\bf A}_{-k}]= \frac{k}{2}$ which implies $\| {\bf A}_{-k} G \|^2= \|{\bf  A}_{k} G \|^2+ \frac{k}{2}  \|  G \|^2 $ we obtain 
\[ \begin{split} 
| \cjg {\bf L}_n^0 F,G \cjd_2   |  & \leq  \Big (  (n+1) Q \|{\bf A}_nF\|^2+ 2 \|{\bf A}_0F\|^2_2  +  \sum_{k=1}^{n-1}   \| {\bf A}_{k}F \|^2_2+ 2   \sum_{m \geq 1} \| {\bf A}_{n+m} F \|^2_2 \Big)^{\frac{1}{2}}     \\
& \quad \times \Big (  ( (n+1) Q+ \frac{n(n-1)}{4}+n) \| G \|^2_2  +  2 \| {\bf A}_n G \|^2_2 + \sum_{k=1}^{n-1} \| {\bf A}_{k}G \|^2_2 + 2 \sum_{m \geq 1} \| {\bf A}_m G \|^2_2\Big)^{\frac{1}{2}} \end{split}\]
The first term is bounded by $C(1+n)^{1/2}\|F\|_{\mc{Q}_0}$ and the second by $C(1+n)\|G\|_{\mc{Q}_0}$ for some uniform $C>0$, which concludes the proof.
\end{proof}

\subsection{Gaussian Free Field and Gaussian multiplicative chaos}\label{subsectionGMC}

On the Riemann sphere $\hat{\C}$, we put the Riemannian metric $g_0=|dz|^2/|z|_+^4$ (with $|z|_+:=\max(1,|z|)$)
 which is invariant by the inversion $z\mapsto 1/z$. We shall use freely the complex variable $z$ in $\C$ or $x\in \R^2$ when we think of it as a real variable.
The Gaussian Free Field $X$ (GFF in short) in the metric $g_0$ is a random variable in the Sobolev space $H^{s}(\mathbb{C})$ for $s<0$ 
with mean $0$ on $\T$ and covariance kernel 
 \begin{equation*}
 \E[X(z)X(z')]= \ln \frac{|z|_+|z'|_+}{|z-z'|}.
 \end{equation*}
If $\varphi$ is the random variable \eqref{GFFcircle} on the unit circle $\mathbb{T}:=\{z\in \C\,|\, z=1\}$, it is easily checked (see \cite{GKRV}) that 
\begin{equation}\label{decomposGFF}
 X= P\varphi+ X_\D+X_{\D^*}
 \end{equation} 
 where $P\varphi$ is the harmonic extension of $\varphi$  and $X_\D,X_{\D^*}$ are two independent GFFs on $\D:=\{z\in \C\,|\, |z|\leq 1\}$  
 and $\D^*:=\{z\in \hat{\C}\,|\, |z|\geq 1\}$  with Dirichlet boundary conditions defined respectively on  the probability spaces $ (\Omega_\D, \Sigma_\D,\P_\D)$ and $ (\Omega_{\D^*}, \Sigma_{\D^*}, \P_{\D^*})$\footnote{With a slight abuse of notations, we will assume that these spaces are canonically embedded in the product space $(\Omega,\Sigma)$ and we will identify them with the respective images of the respective embeddings.}. 
 The covariance of $X_{\D}$ is the Dirichlet Green's function on $\D$, given by 
 \[ \E[X_{\D}(z)X_{\D}(z')]=G_\D(z,z')= \ln \frac{|1-z\bar{z}'|}{|z-z'|}.\]
 The random variable $X$  is defined on a probability space $(\Omega, \Sigma, \P)$ (with expectation $\E[.]$) where $\Omega= \Omega_\T \times \Omega_\D \times \Omega_{\D^*} $, $\Sigma= \Sigma_\T \otimes \Sigma_\D \otimes \Sigma_{\D^*}$ and $\P$ is a product measure $\P=\P_\T\otimes \P_{\D}\otimes \P_{\D^*}$. We will write  $\E_\varphi[\cdot ]$ for conditional expectation with respect to the GFF on the circle $\varphi$.
 
We introduce the Gaussian multiplicative chaos measure (GMC in short), defined by Kahane \cite{cf:Kah},
\begin{equation}\label{GMCsphere}
e^{\gamma X(x)}\dd x:=  \underset{\epsilon \to 0} {\lim}  \; \; e^{\gamma X_\epsilon(x)-\frac{\gamma^2}{2} \E[X_\epsilon(x)^2]}\dd x
\end{equation}
where $\dd x$ denotes the Lebesgue measure of $\C$ and  $X_\epsilon= X \ast \theta_\epsilon$ is the mollification of $X$ with an  approximation   $(\theta_\epsilon)_{\epsilon>0}$ of the Dirac mass $\delta_0$; indeed, one can show that the limit \eqref{GMCsphere} exists in probability in the space of Radon measures on $\hat\C$ and that the limit does not depend on the mollifier $\theta_\eps$: see \cite{RoV, review, Ber}.   The condition $\gamma \in (0,2)$ stems from the fact that the random measure $e^{\gamma X(x)}\dd x$ is different from zero if and only if $\gamma \in (0,2)$. We will also consider the GMC measure on the circle $\T$ associated to $\varphi$ defined via 
\begin{equation}\label{GMCcircle}
e^{\gamma \varphi(\theta)} \dd \theta:=  \underset{\epsilon \to 0} {\lim}  \; \; e^{\gamma \varphi_\epsilon(\theta)-\frac{\gamma^2}{2} \E[\varphi_\epsilon(\theta)^2]} \dd \theta
\end{equation}
where $\varphi_\epsilon$ is a mollification at scale $\epsilon$ of $\varphi$. The measure $e^{\gamma \varphi(\theta)} \dd \theta$ is different from zero if and only if $\gamma \in (0,\sqrt{2})$ and therefore for $\gamma \in [\sqrt{2},2)$, the Fourier coefficients in \eqref{virasoro} will not act as multiplication by a variable and will rather be defined via the Girsanov transform in the context of Dirichlet forms. 

\subsection{Liouville Hamiltonian and its quadratic form} 
In \cite[Section 5]{GKRV}, an Hamiltonian ${\bf H}$ associated to Liouville CFT is defined as an unbounded
 operator acting on the Hilbert space $\mc{H}$, it is formally  given by the expression on $\mc{C}_{\rm exp}$ 
\begin{equation}\label{LiouvilleH}
{\bf H}={\bf H}^0+\mu e^{\gamma c}V 
\end{equation}
where $V$ is a non-negative unbounded operator on $L^2(\Omega_\T,\mathbb{P}_\T)$ associated to a symmetric 
quadratic form $\mc{Q}_V$; when $\gamma<\sqrt{2}$, $V$ is a multiplication operator by the potential  
\[V(\varphi)=\int_0^{2\pi} e^{\gamma \varphi(\theta)}d\theta \in L^{\frac{2}{\gamma^2}-\eps}(\Omega_\T), \quad \forall \eps>0\]
with $e^{\gamma \varphi(\theta)}d\theta$ being the GMC measure on the circle defined in \eqref{GMCcircle}.
The rigorous construction of   ${\bf H}$ uses again the notion of Friedrichs extension. Consider the symmetric quadratic form 
\begin{equation}\label{quadformQ}
\mc{Q}(F,F):=\mc{Q}_0(F,F)+\mu \int_{\R} e^{\gamma c}\mc{Q}_V(F,F)\dd c \geq \mc{Q}_0(F,F).
\end{equation}
This quadratic form admits a closure on a domain $\mc{D}(\mc{Q})\subset \mc{D}(\mc{Q}_0)$ and one can view ${\bf H}$ as a self-adjoint operator using the Friedrichs extension \cite[Proposition 5.5]{GKRV}. The operator ${\bf H}$ has a domain $\mc{D}({\bf H})\subset \mc{D}(\mc{Q})$ but it is also bounded as a map 
\[ {\bf H} : \mc{D}(\mc{Q})\to \mc{D}'(\mc{Q})\]
where $ \mc{D}'(\mc{Q})$ is the dual space to $\mc{D}(\mc{Q})$. The propagator $e^{-t{\bf H}}$ for $t\geq 0$ is a contraction 
semi-group that has the Markov property \cite[Proposition 5.2]{GKRV}, it will appear also below as a special case of 
Markovian dynamic that one can construct using Markovian vector fields.

\subsection{Markovian holomorphic vector fields}
For $\eps>0$ small, let ${\rm BiHol}_\eps(\D)$ be the space of biholomorphic maps $f:(1+\eps)\D^\circ \to f((1+\eps)\D^\circ)\subset \C$ extending smoothly to $(1+\eps)\T$. 
Infinitesimal variations around the identity in ${\rm BiHol}_\eps(\D)$ give rise to 
the space $\mathrm{Vect}_\eps(\D)$ of holomorphic vector fields on $(1+\eps)\D^\circ$ extending smoothly to $(1+\eps)\T$. 
These vector fields can be  written as $\bv=v\del_z$ for some holomorphic function $v$, with power series expansion at $z=0$
\begin{equation}\label{eq:coordinates_vector_fields}
v(z)=-\sum_{n=-1}^\infty v_nz^{n+1}.
\end{equation}
These vector fields encode the infinitesimal motion $f_t(z)=z+tv(z)+o(t)$ on ${\rm BiHol}_\eps(\D)$.

The basis elements 
\[\bv_n=-z^{n+1}\del_z \in\mathrm{Vect}_\eps (\D) ,\quad n\geq-1\] 
span a Lie subalgebra of the complex Witt algebra. Recall that the complex Witt algebra is the complex Lie algebra of polynomial vector fields on $\T$, i.e.vector fields of the form $\sum_{|k|\leq N}a_k e^{ik\theta}\pl_{\theta}$ for $a_k\in \C$ and some $N<\infty$. This is a subalgebra of the Lie algebra $\C(z)\del_z$ of meromorphic vector fields in the Riemann sphere $\hat{\C}$. 
There is an antilinear involution on $\C(z)\del_z$ given by $\bv\mapsto \bv^*:=z^2\bar{v}(1/\bar{z})\del_z$. On the canonical generators, we have $\bv_n^*=\bv_{-n}$ and $(i\bv_n)^*=-i\bv_{-n}$. The vector fields with negative powers encode the deformations near the identity of the space of conformal transformations of the outer disc. The involution $^*$ relates the two types of deformations (i.e inner and outer discs). 

Let us introduce the following subspace of $\mathrm{Vect}_\eps(\D)$:
\begin{equation}\label{eq:def_vect_plus}
\begin{aligned}
&\mathrm{Vect}_\eps^+(\D):=\left\lbrace\bv\in\mathrm{Vect}_\eps(\D)|\,  \forall z\in\T,  \Re\left(\bar{z}v(z)\right)<0\,\right\rbrace.
\end{aligned}
\end{equation}
We will refer to such vector fields as \emph{Markovian}. 
Notice that for each $\bv\in\mathrm{Vect}_\eps (\D)$ with $v(0)=0$, one can find $\omega>0$ large enough such that $\omega \bv_0+\bv$ is Markovian, if $\bv_0=-z\pl_z$. A solution $f_t(z)$ of the ODE $\pl_tf_t(z)=v(f_t(z))$  in some open subset of $(1+\eps)\D^\circ$ with initial condition $f_0(z)=z$ must satisfy
that the curve $t\mapsto ({\rm Re}(f_t(z)),{\rm Im}(f_t(z))\in \R^2$ is the integral curve of the real vector field $V={\rm Re}(v)\pl_x+{\rm Im}(v)\pl_y=2{\rm Re}(\bv)$, 
and the assumption $\Re\left(\bar{z}v(z)\right)<0$ means that $V$ is pointing strictly inside $\D$ at $\T=\pl\D$. In particular this implies that the flow 
$f_t(z)$ of $\bv$ is defined for all $t\geq 0$ in $\D$. By continuity, the vector field $V$ also points inside the disk $(1+\eps')\D$ if $\eps'>0$ is small enough, and therefore $f_t(z)$ is also well-defined for all $t\geq 0$ inside $(1+\eps')\D$.

\begin{lemma}\label{properties_flows}
For  $\bv=v(z)\pl_z \in \mathrm{Vect}_\eps^+(\D)$, let $f_t(z)$ be the solution of the complex ordinary differential equation $\pl_tf_t(z)=v(f_t(z))$ with $f_0(z)=z$, defined for all $t\geq 0$ in $(1+\eps')\D$ for some $0<\eps'<\eps$. 
The family $(f_t)_{t\geq 0}$ is a family of holomorphic maps in ${\rm BiHol}_{\eps'}(\D)$ satisfying $f_{t}(\D)\subset f_{s}(\D)$ for all  
$t\geq s$ and $f_t(\D)$ converges exponentially fast to $a\in \D^\circ$, the unique zero of $v(z)$ in $\D$.
\end{lemma}
\begin{proof}
The functions $(f_t)_{t\geq 0}$ are biholomorphisms in $(1+\eps')\D$ as a well-defined flow of a holomorphic vector field. 
We first show that there exists an $a \in \D^\circ$ such that $v(a)=0$. If this is not the case then $1/v(z)$ is holomorphic in $\D$ and hence by Cauchy's theorem
\begin{equation*}
\int_{0}^{2\pi} \frac{e^{i \theta} \overline{v (e^{i \theta})}}{|v (e^{i \theta})|^2} d \theta = \int_{0}^{2\pi} \frac{e^{i \theta} }{v (e^{i \theta})} d \theta =0
\end{equation*} 
which is in contradiction with the fact that $\Re (e^{i \theta} \overline{v (e^{i \theta})}  )    <0$ for all $\theta \in [0,2\pi]$. Uniqueness of $a$ will then be a consequence of the fact that $f_t(\D)$ converges exponentially fast to $a$ (indeed $f_t(\D)$ can only have one limit).

We first consider a vector field $\bv \in \mathrm{Vect}_\eps^+(\D)$ which satisfies $v(0)=0$. By the maximum principle for harmonic functions applied to $\Re\left( \frac{v(z)}{z}\right)$, there exists $c>0$ such that
\begin{equation}\label{inequa}
\sup_{|z| \leq 1} \Re\left( \frac{v(z)}{z}\right) \leq -c.
\end{equation} 
The flow satisfies $\partial_t f_t(z)= v( f_t(z) )$ and $f_{0}(z)=z$. 
Thanks to \eqref{inequa}, one has the following 
\begin{equation*}
\partial_t  |f_t(z)|^2 =2{\rm Re}\Big (\frac{\pl_tf_t(z)}{f_t(z)}\Big)|f_t(z)|^2=2 {\rm Re}\Big(\frac{v(f_t(z))}{f_t(z)}\Big)|f_t(z)|^2 \leq -2c |f_t(z)|^2
\end{equation*}  
and therefore $|f_t(z)| \leq e^{-ct}|z|$.
Next we consider the case where there exists some $a \in \D^\circ$ such that $v(a)=0$. 
We then introduce the M\"obius map $\psi(z)= \frac{z-a}{1-\bar a z}$ preserving $\D$ and such that $\psi(a)=0$ and consider the vector field $v_\psi= (\psi' \circ \psi^{-1})  v \circ \psi^{-1} $ which satisfies $v_\psi(0)=0$ and 
\begin{equation*}
\sup_{|z| = 1} \Re\left(  \bar{z}  v_\psi(z) \right) = \sup_{|z| = 1} \Re\left(  \bbar{\psi (z)} \psi'(z)  v(z) \right)= \sup_{|z| = 1} 
\Re\left(  \bar{ z}  v(z) \right) \frac{1-|a|^2}{|1-a\bar{z}|^2}< 0
\end{equation*} 
by noticing that $ \overline{\psi (z)} \psi'(z)= \bar{z}\frac{1-|a|^2}{|1-\bar{z}a|^2}$ when $|z| = 1$. The vector field $v_\psi$ has  $ \psi \circ f_t \circ \psi^{-1}$ as flow. This concludes the proof.
\end{proof}

\begin{lemma}\label{limith}
Assume that  $\bv=v(z)\pl_z \in \mathrm{Vect}_\eps^+(\D)$, that $v(0)=0$, $v'(0)=-\omega$ for $\omega>0$. Let $f_t(z)$ be the solution of the complex ordinary differential equation $\pl_tf_t(z)=v(f_t(z))$ with initial condition  $f_0(z)=z$ defined for all $t\geq 0$ in $|z|<1+\eps'$ for some $0<\eps'<\eps$. 
For each $k$, the function $e^{\omega t}f_t$ converges in $C^k(\D)$ norm (i.e. all the partial derivatives of order $j \leq k$ converge uniformly on $\D$) as $t\to \infty$ towards a function $h\in {\rm BiHol}_{\eps'}(\D)$ for some $\eps'>0$.
Finally, $f'_t(0)=e^{-\omega t}$ and $h'(0)=1$.
 \end{lemma}
\begin{proof}
By Lemma \ref{properties_flows}, we know that $v$ has no other zeros than $z=0$ in $(1+\eps')\D$ for some $\eps'<\eps$ and that $|f_t(z)|$ converges exponentially fast to $0$.
Since $v(z)=-\omega z+\mc{O}(|z|^2)$ near $z=0$, for all $\delta>0$ small there is $\eta>0$ small such that for $|z|\leq \eta$  
we have ${\rm Re}(v(z)/z)<-\omega+\delta$. By the argument below \eqref{inequa}, we deduce that $|f_t(z)|\leq e^{-(\omega-\delta)t}|\eta|$ for $|z|\leq \eta$ 
and thus there is $C>0$ such that for all $|z|\leq 1+\eps'$, we have $|f_t(z)|\leq Ce^{-(\omega-\delta)t}$.
 Therefore, the function $g_t(z):=e^{\omega t}f_t(z)$ satisfies, uniformly in $|z|\leq 1+\eps'$, 
\[ \pl_t g_t(z)=e^{\omega t}(\omega f_t(z)+v(f_t(z)))=\mc{O}(e^{\omega t}|f_t(z)|^2)=\mc{O}(e^{-(\omega-2\delta)t}),\]
where we have used $v(z)=-\omega z+\mc{O}(|z|^2)$ for the second equality. Thus, choosing $\delta<\omega/2$, we deduce that there is some function $h: (1+\eps')\D \to \C$ so that, uniformly on $(1+\eps')\D$, the following converge holds
\[ e^{\omega t}f_t(z)=g_t(z)= z+\int_{0}^{t}  \pl_s g_s(z)ds \to_{t\to \infty} h(z). \]
The function $h$ is then holomorphic in  $(1+\eps')\D^\circ$  as uniform limit of holomorphic functions, moreover $g_t\to h$ in all $C^k((1+\eps')\D)$ norms 
(using Cauchy formula, the derivatives are dealt with using the $C^0$ convergence). The fact that $h$ is injective on $(1+\eps)\D^\circ$ follows from Carath\'eodory's kernel theorem, since $g_t$ is injective on $(1+\eps')\D$ and converges uniformly on  $(1+\eps')\D$.
Expanding $f_t(z)=\sum_{n\geq 1}a_n(t)z^n$ at $z=0$, we see that $a_n(t)$ must solve the ODE $\pl_t a_1(t)=-\omega a_1(t)$ with $a_1(0)=1$, 
thus $a_1(t)=e^{-\omega t}$ and this implies that $h'(0)=1$. 
\end{proof}

We conclude the basic properties of the low  with the following lemma, which will be useful in the sequel:

\begin{lemma}\label{limitandh}
Assume that $\bv=v(z)\pl_z \in {\rm Vect}_\eps^+(\D)$, that $v(0)=0$, $v'(0)=-\omega$. Then for all $t\geq 0$, the flow $f_t(z)$ of $\bv$ solves 
$v(f_t(z))=f_t'(z)v(z)$  and  the function $h$ of Lemma \ref{limith} solves the differential equation $\omega h(z)= -h'(z) v(z)$.
\end{lemma}
\begin{proof}
Since $f_t(z)$ is one-to-one, $f_t'(z)$ never vanishes in $\D$. We then have 
\begin{equation*}
\pl_t  \frac{v(f_t(z))}{f_t'(z)}= \frac{ \pl_t f_t(z) v'( f_t(z)  )}{f_t'(z)}- \frac{ ( \pl_t f_t )'(z) v( f_t(z)  )}{(f_t'(z))^2} =  \frac{ v( f_t(z)  ) v'( f_t(z)  )}{f_t'(z)}-\frac{ f_t'(z) v'( f_t(z))   v( f_t(z)  )}{(f_t'(z))^2}=0.
\end{equation*}
This shows that $v(f_t(z))=f_t'(z)v(z)$ and thus as $t\to +\infty$ 
\begin{equation}\label{limitandheq}
e^{\omega t}f'_t(z)v(z)= e^{\omega t}v(f_t(z))=-\omega e^{\omega t}f_t(z)+\mc{O}(e^{-\omega t})
\end{equation}
and by letting $t\to +\infty$, we get $h'(z)v(z)=-\omega h(z)$.
\end{proof}

\subsection{The probabilistic formula for the vector field dynamics}

Let $\bv=v(z)\del_z$ be a Markovian vector field, and we assume that $v(0)=0$ and $v'(0)=-\omega$ for $\omega>0$. This vector field induces a family of small conformal transformations $f_t(z)=z+tv(z)+o(t)$ on $\D$ such that $f_t(\T)\subset\D$ for all $t$ and $f'_t(0)=e^{-\omega t}$. Given an initial condition $c+\varphi \in H^{-s}(\T)$ with $s>0$, we define a stochastic process $(c+B_t,\varphi_t)\in H^{-s}(\T)$ via the following formula
\begin{equation}\label{eq:process}
(c+B_t,\varphi_t):=\left(c+(P\varphi+X_\D)\circ f_t+Q\log \frac{|f_t'|}{|f_t|}\right)\Big|_{\T}
\end{equation}
where $|_{\T}$ denotes restriction to the unit circle and $(c+B_t,\varphi_t)$ is the standard decomposition into the average and the average zero part. That is, we look at the trace of the Gaussian Free Field $X=P\varphi+X_\D+X_{\D^*}$ on the deformed circle $f_t(\T)$ and pull this field back to $\T$ using $f_t$ (notice that the last term $X_{\D^*}|_{f_t(\T)}=0$) . The extra term $Q\log \frac{|f_t'|}{|f_t|}$ is here to conform with the transformation properties of the Liouville field. 
Notice that $X_{\D}\circ f_t=\sum_{n}X_n(t)e^{in\theta}$ where 
\begin{equation}\label{X_n(t)}
X_n(t)=\frac{1}{2\pi}\int_0^{2\pi} X_{\D}(f_t(e^{i\theta}))e^{-in\theta}d\theta
\end{equation} 
\begin{equation}\label{covarianceXn} 
\E[ X_n(t)X_{m}(t')]= \frac{1}{4\pi^2}\int_0^{2\pi}\int_0^{2\pi} e^{-in\theta-im\theta'}\log \frac{| 1-f_t(e^{i\theta})\bbar{f_{t'}(e^{i\theta'})}|}{|f_t(e^{i\theta})-f_{t'}(e^{i\theta'})|}d\theta d\theta'.
\end{equation}
We can now register a few formulas: for  all $x \in \D \setminus f_t(\D)$
\begin{equation*}
\frac{1}{2 \pi}  \int_0^{2 \pi}  G_{\D} (x,f_t(e^{i \theta'}) )  d \theta'= \log \frac{|   1- f_t(0) \overline{x})|}{|f_t(0)-x|}= \log \frac{1}{|x|}.   
\end{equation*}
This formula is just the average principle applied to the harmonic function $z \mapsto G_{\D} (x,f_t(z) )$ for $z \in \D$. Taking $x \to f_t(e^{i \theta})$ yields by continuity for all $\theta$
\begin{equation*}
\frac{1}{2 \pi}  \int_0^{2 \pi}  G_{\D} (f_t(e^{i \theta}),f_t(e^{i \theta'}) )  d \theta'= \log \frac{ 1 }{|f_t(e^{i \theta}) |}.  
\end{equation*}
Now using the mean value property applied to the harmonic function $z \mapsto \log \frac {|f_t(z) |}{|z|} $ for $z \in \D$, we get
\begin{equation*}
\frac{1}{2 \pi} \frac{1}{2 \pi}  \int_0^{2 \pi}  \int_0^{2 \pi}  G_{\D} (f_t(e^{i \theta}),f_t(e^{i \nu}) )  d \theta d \nu= - \frac{1}{2 \pi}   \int_0^{2 \pi}   \log \frac {|f_t(e^{i \theta}) |} { |e^{i \theta}| } d \theta =     -\log |f_t'(0)|.
\end{equation*}
Similar considerations yield for $s \leq t$
\begin{equation*}
\frac{1}{2 \pi} \frac{1}{2 \pi}  \int_0^{2 \pi}  \int_0^{2 \pi}  G_{\D} (f_t(e^{i \theta}),f_s(e^{i \nu}) )  d \theta d \nu= - \frac{1}{2 \pi}   \int_0^{2 \pi}   \log \frac {|f_s(e^{i \theta}) |} { |e^{i \theta}| } d \theta =     -\log |f_s'(0)|.
\end{equation*}
In particular this shows that  $\E[ X_0(s) X_0(t) ]=  \omega \min(  s,t) $  and hence $X_0(t)=B_t$ is a Brownian motion with diffusivity $\omega$. 

\begin{proposition}\label{continuousprocess}
Let $\bv=v(z)\pl_z \in  {\rm Vect}_\eps^+(\D) $ with $v(0)=0$ and $v'(0)=-\omega<0$, and let $f_t$ be the associated flow. Then the field $(B_t,\varphi_t)$ is a continuous process with values in $H^{-s}(\T)$ for $s>0$ and the operator $P_t^0$ associated to this process, defined for all bounded and continuous $F$ by
\begin{equation}\label{definitionPt0}
  P_t^0F(c,\varphi):= |f_t'(0)|^{\frac{Q^2}{2}} \E_\varphi[  F( c+B_t+\varphi_t )]  
\end{equation} 
is a Markov semi-group.
\end{proposition}
\begin{proof}
First, the field $\left(c+P\varphi \circ f_t+Q\log \frac{|f_t'|}{|f_t|}\right)|_{\T}$ is clearly continuous in  $H^{-s}(\T)$. The property of the 
second part of the process is a consequence of the following lemma
\begin{lemma}
Let $s>0$. There exist $\alpha>0$ and $C>0$ such that for all $t,t'\geq 0$
\[
\E[ \| (X_\D \circ f_t)|_{\T} - (X_\D \circ f_{t'})|_{\T} \|_{H^{-s}(\T)}^2  ]  \leq C |t-t'|^\alpha
\]
\end{lemma}
\proof 
Writing $\cjg n\cjd=(1+|n|)$, we have 
\begin{align*}
& \E\Big[   \sum_{n \in \Z} \cjg n\cjd^{-2s}  \Big|  \int_0^{2 \pi }  (X_\D \circ f_t)  (e^{i \theta})  e^{-i n \theta}   d \theta -   \int_0^{2 \pi }  (X_\D \circ f_{t'})  (e^{i \theta})  e^{-i n \theta}   d \theta  \Big| ^2     \Big]   \\
& = \sum_{n \in \Z} \frac{\cjg n\cjd^{-2s}}{2 \pi}  \int_{[0,2\pi]^2}  \left ( G_\D (f_t (e^{i \theta}), f_t (e^{i \theta'}))+G_\D (f_{t'} (e^{i \theta}), f_{t'} (e^{i \theta'})) -2 G_\D (f_t (e^{i \theta}), f_{t'} (e^{i \theta'}))  \right )    e^{in (\theta-\theta')} d\theta d \theta'    \\
&=   \int_{[0,2\pi]^2}   \left ( G_\D (f_t (e^{i \theta}), f_t (e^{i \theta'}))+G_\D (f_{t'} (e^{i \theta}), f_{t'} (e^{i \theta'})) -2 G_\D (f_t (e^{i \theta}), f_{t'} (e^{i \theta'}))  \right )    \left (   \sum_{n \in \Z} \frac{\cjg n\cjd^{-2s}}{2\pi} e^{in (\theta-\theta')}  \right ) d\theta d \theta' \\
& \leq   C  \int_0^{2 \pi } \int_0^{2 \pi }   | G_\D (f_t (e^{i \theta}), f_t (e^{i \theta'}))+G_\D (f_{t'} (e^{i \theta}), f_{t'} (e^{i \theta'})) -2 G_\D (f_t (e^{i \theta}), f_{t'} (e^{i \theta'}))  |   \frac{1}{|\theta-\theta'|^{1-2s}} d\theta d \theta' \\
& \leq  2C  \int_0^{2 \pi } \int_0^{2 \pi }   | G_\D (f_{t'} (e^{i \theta}), f_{t'} (e^{i \theta'})) - G_\D (f_t (e^{i \theta}), f_{t'} (e^{i \theta'}))  |   \frac{1}{|\theta-\theta'|^{1-2s}} d\theta d \theta' \\
& \quad +   2C  \int_0^{2 \pi } \int_0^{2 \pi }   | G_\D (f_t (e^{i \theta}), f_t (e^{i \theta'}))- G_\D (f_t (e^{i \theta}), f_{t'} (e^{i \theta'}))  |   \frac{1}{|\theta-\theta'|^{1-2s}} d\theta d \theta', 
\end{align*}
where we used the bound $  \sum_{n \in \Z} \cjg n\cjd^{-2s} e^{in (\theta-\theta')}  \leq C|\theta-\theta'|^{-1+2s}$. We only deal with the first term since and the second is similar. Now we have
\begin{align*}
 & \int_0^{2 \pi } \int_0^{2 \pi }   \Big| G_\D (f_{t'} (e^{i \theta}), f_{t'} (e^{i \theta'})) - G_\D (f_t (e^{i \theta}), f_{t'}(e^{i \theta'}))  \Big|   \frac{1}{|\theta-\theta'|^{1-2s}} d\theta d \theta'   \\
&= \int_{|t-t'| \leq |\theta-\theta'|^2}  \Big| G_\D (f_{t'} (e^{i \theta}), f_{t'} (e^{i \theta'})) - G_\D (f_t (e^{i \theta}), f_{t'} (e^{i \theta'})) \Big |   \frac{1}{|\theta-\theta'|^{1-2s}} d\theta d \theta'   \\
& \quad +\int_{|t-t'| > |\theta-\theta'|^2}  \Big | G_\D (f_{t'} (e^{i \theta}), f_{t'} (e^{i \theta'})) - G_\D (f_t (e^{i \theta}), f_{t'} (e^{i \theta'}))  \Big|   \frac{1}{|\theta-\theta'|^{1-2s}} d\theta d \theta'.
\end{align*}
But, using the explicit formula $G_{\D}(z,z')=\log \frac{|z-z'|}{|1-\bar{z}z'|}$, the bounds 
\[|f_t (e^{i \theta})-f_{t'} (e^{i \theta'})|\geq C^{-1}(|t-t'|+|\theta-\theta'|),\quad |1-\bbar{f_t (e^{i \theta})}f_{t'} (e^{i \theta'})|\geq c(|t-t'|+|\theta-\theta'|)
\] 
for some $c>0$ locally uniform in $t',t$,  there is $C>0$ locally uniform in $t,t'$ so that
\begin{align*}
& \int_{|\theta-\theta'|^2 < |t-t'|}     \Big| G_\D (f_{t'} (e^{i \theta}), f_{t'} (e^{i \theta'})) - G_\D (f_t (e^{i \theta}), f_{t'} (e^{i \theta'}))  \Big|   \frac{1}{|\theta-\theta'|^{1-2s}} d\theta d \theta'  \\
& \leq C \int_{|\theta-\theta'|^2 < |t-t'|}  |\ln (|\theta-\theta'|)|    \frac{1}{|\theta-\theta'|^{1-2s}} d\theta d \theta'    \leq C  |t-t'|^{s} .
\end{align*}
Similarly, one gets
\begin{align*}
& \int_{|t-t'| \leq |\theta-\theta'|^2}      \Big|  \ln \Big|\frac{1-f_{t'}(e^{i \theta})  \overline{ f_t(e^{i \theta'})  } }{  1-  f_t(e^{i \theta})  \overline{ f_t(e^{i \theta'})  }  }\Big| \Big|   \frac{1}{|\theta-\theta'|^{1-2s}} d\theta d \theta'  \\ 
& =  \int_{|t-t'| \leq |\theta-\theta'|^2}    \Big |  \ln\Big |1+ \frac{    ( f_t(e^{i \theta})-f_{t'}(e^{i \theta})  )    \overline{ f_t(e^{i \theta'})  } }{  1-  f_t(e^{i \theta})  \overline{ f_t(e^{i \theta'})  }  }\Big| \Big |   \frac{1}{|\theta-\theta'|^{1-2s}} d\theta d \theta'    \\
& \leq  \int_{|t-t'| \leq |\theta-\theta'|^2}     \Big|    \frac { ( f_t(e^{i \theta})-f_{t'}(e^{i \theta})  )    \overline{ f_t(e^{i \theta'})  } }{  1-  f_t(e^{i \theta})  \overline{ f_t(e^{i \theta'})  }  } \Big |   \frac{1}{|\theta-\theta'|^{1-2s}} d\theta d \theta'    \\
& \leq  C \int_{|t-t'| \leq |\theta-\theta'|^2}     |t-t'| \frac{1}{|\theta-\theta'|^{2-2s}} d\theta d \theta'    \leq  C  |t-t'|^{s}.
\end{align*}
The term in $\ln |  \frac{f_t(e^{i \theta}) -f_{t'}(e^{i \theta'}) }{f_t(e^{i \theta}) -f_t(e^{i \theta'})}  |$ can be dealt with similarly.\qed

On Gaussian spaces, the $L^p$ norms are all equivalent on polynomials of fixed bounded degree (see \cite[Theorem 3.50]{jan}) and therefore for all $p \geq 1$ there exists some $C>0$ such that for  $t,t'\geq 0$ 
\begin{equation*}
\E[   \| (X_\D \circ f_t)|_{\T} - (X_\D \circ f_{t'})|_{\T} \|_{H^{-s}(\T)}^p  ]  \leq C |t-t'|^{p\frac{\alpha}{2}}
\end{equation*}
hence by choosing $p \alpha >2$ the Kolmogorov continuity theorem ensures that the process $(\varphi_t)_{t \geq 0}$ admits a continuous 
version in $H^{-s}(\T)$ for $s>0$. 

We now turn to the Markov property. Let $s,t>0$; since $f_{t+s}= f_t \circ f_s$ we have the following decomposition
\begin{equation*}
(P\varphi+X_\D)\circ f_{t+s}+Q\log \frac{|f_{t+s}'|}{|f_{t+s}|}= P\varphi \circ f_t \circ f_s+ X_\D \circ f_t \circ f_s+ Q \log \frac{|f_{t}' \circ f_s| |f_s|}{|f_{t}\circ f_s|}+ Q\log \frac{|f_{s}'|}{|f_{s}|}
\end{equation*}
By conformal invariance of $X_\D$ and using the Markov property for the Dirichlet GFF, we get that
\begin{equation*}
 X_\D \circ f_t \circ f_s= (P(X_\D \circ f_t)|_\T)\circ f_s+ \tilde{X}_\D \circ f_s
\end{equation*}
where $P(X_\D \circ f_t)|_\T$ is the Harmonic extension of $(X_\D \circ f_t)|_\T$ and $\tilde{X}_\D$ is an independent Dirichlet GFF. Obviously for all $|z| \leq 1$
\begin{equation*}
(P\varphi \circ f_t)(z)= P ( (P\varphi \circ f_t)|_\T )(z)
\end{equation*}
since both sides are harmonic with same value at $|z|=1$ and therefore we can compose with $f_s$ to get
\begin{equation*}
P\varphi \circ f_t \circ f_s= P ( (P\varphi \circ f_t)|_\T )\circ f_s.
\end{equation*}
Finally, we have for all $|z| \leq 1$
\begin{equation*}
P  (  \log \frac{|f_{t}'|}{|f_{t}|} |_\T )(z) = \log \frac{|f_{t}'(z)| |z|}{|f_{t}(z)|}
\end{equation*}
since both sides are harmonic with same value at $|z|=1$ and therefore we can compose with $f_s$ to get  
\begin{equation*}
P  (  \log \frac{|f_{t}'|}{|f_{t}|} |_\T ) \circ f_s(z)= \log \frac{|f_{t}' \circ f_s | |f_s(z)|}{|f_{t}\circ f_s|}.
\end{equation*}
We can combine the previous identities to get the claimed Markov property.
\end{proof}

In order to describe the invariant measure of the operator $|f_t'(0)|^{-\frac{Q^2}{2}} P_t^0$, we identify the kernel $-\log |h(e^{i \theta})-h(e^{i \theta'})|$ (recall that $h$ appears in Lemma \ref{limith}) with the resolvent of a jump operator across the curve $h(\T)$ and deduce the existence of a Gaussian field with this covariance.
\begin{lemma}\label{firstlemma1}
There exists a Gaussian random variable $X_h\in H^{-s}(\T)$, with vanishing mean and covariance
\[
\E[  X_h(e^{i \theta})  X_h(e^{i \theta'})   ]  = \log \frac{1}{|h(e^{i \theta})-h(e^{i \theta'})  |}.
\]
\end{lemma}
\begin{proof}
It suffices to show that for all real non zero $f\in L^2(\T)$ with $\int_{0}^{2 \pi}f(e^{i \theta})\dd \theta=0$,
\begin{equation}\label{positivecovariance}
\int_0^{2 \pi} \int_0^{2 \pi}   \log \frac{1}{|h(e^{i \theta})-h(e^{i \theta'})  |}  f(e^{i\theta}) f(e^{i\theta'}) \dd \theta \dd \theta' >0
\end{equation}
namely that the kernel is positive definite (see \cite{jan}). Let $K$ be the single layer operator $K:C^\infty(\T)\to C^0(\R^2)$
\[ Kf(x)=\int_{\T} G(x,h(e^{i\theta}))f(e^{i\theta}) \dd \theta \]
with  $G(x,x'):=(2\pi)^{-1} \log \frac{1}{|x-x'|}$ the Green's function on $\R^2$. We have $\Delta Kf=0$ in $\R^2\setminus h(\T)$ and 
$((\pl_{\nu}^+-\pl_{\nu}^-)Kf)(h(e^{i\theta}))=\frac{f(e^{i\theta})}{ |h'(e^{i\theta})|  }$ where $\pl_{\nu}^+$ is the outward  unit normal derivative in $h(\D)$ and $\pl_{\nu}^-$ the outward unit normal derivative in $\R^2\setminus h(\D)$. Thus $Kf$ solves a Neumann problem 
\begin{equation}\label{Neumannprb}
 \Delta Kf=0 \textrm{  in } \R^2 \setminus h(\T) ,\quad (\pl_{\nu}^+-\pl_{\nu}^-)Kf |_{h(\T)} =\frac{f\circ h^{-1}}{|h'|\circ h^{-1}},
 \end{equation}
moreover we see that 
\[ Kf(x)=-(2\pi)^{-1}\log(|x|) \int_{0}^{2 \pi} f(e^{i\theta}) \dd \theta + \mc{O}(1/|x|) \textrm{ as }|x|\to \infty\]
so $Kf(x)\to 0$ as $|x|\to \infty$ if $\int_{0}^{2 \pi}  f(e^{i\theta}) \dd \theta=0$. It is also direct to check that $|\nabla Kf(z)|=\mc{O}(1/|z|^2)$ for large $|z|$ so that we can write using Green's formula and \eqref{Neumannprb}
\[ \int_{\R^2} |\nabla Kf(x)|^2 \dd x=\int_{\R^2\setminus h(\D)}|\nabla Kf(x)|^2 \dd x+\int_{h(\D)}|\nabla Kf(x)|^2\dd x= \int_{0}^{2\pi} f(e^{i\theta})(Kf)(h(e^{i\theta}))\dd \theta \]
which is nothing more than the left hand side of \eqref{positivecovariance}.  
\end{proof}

Let $\P_h$ be the probability measure induced by $X_h+Q \log \frac{\omega}{|v|}$ on $H^{-s}(\T)$ for $s>0$ (recall that $\omega$ and $v$ appear in Proposition \ref{continuousprocess}) 
 and denote 
 \begin{equation}\label{defmuh}
 \mu_h:= \dd c \otimes \P_h
 \end{equation} 
 the product measure on $H^{-s}(\T)$.

If ${\bf v}=v(z)\pl_z$ with $v(z)=-\omega z-\sum_{n=1}^\infty v_nz^{n+1}$, we define the operator $\bH_\bv^0$ via the following formula: 
\begin{equation}\label{eq:def_L}
 \forall F\in\cC_\mathrm{exp},\quad  \bH_\bv^0 F=\omega   \bH^0 F + \sum_{n=1}^{\infty}  \Re (v_n)  (\mathbf{L}_n^0+\widetilde{\mathbf{L}}_n^0 ) F  + i \sum_{n=1}^{\infty}  \Im (v_n) (\mathbf{L}_n^0-\widetilde{\mathbf{L}}_n^0)F.
\end{equation}
One should notice that in the two sums $\sum_{n = 1}^{\infty}$ only a finite number of terms are not equal to $0$ because $F \in \cC_\mathrm{exp}$ and therefore the above quantity is well defined. In particular, $\bH^0:=\bH_{\bv_0}^0$ (recall $\bv_0=-z\pl_z$) is the free field Hamiltonian whose expression is  \eqref{defH0}.

The main result of this section is the following theorem which establishes the link between  $\bH_\bv^0$ and $P_t^0$:

\begin{theorem}\label{theoremfreefield}
 Let $\bv=v(z)\pl_z$ be a Markovian vector field such that $v(0)=0$, $v'(0)=-\omega$ and $v$ admits a holomorphic extension in a neighborhood of $\D$. There exists an absolute constant $K>0$ such that if $\omega> K \sum_{n \geq 1}|v_n|n^2 $ then the operator $\bH_\bv^0$ defined in \eqref{eq:def_L} admits a closed extension such that $e^{-t \bH_\bv^0 }$ is a continuous contraction semigroup on $\mc{H}=L^2(\R\times \Omega_\T)$. The semigroup $|f_t'(0)|^{-\frac{Q^2}{2}} P_t^0$ admits $\mu_h$ of \eqref{defmuh} as invariant measure, this measure is absolutely continuous with respect to $\mu_0$ and $P_t^0$ coincides with $e^{-t \bH_\bv^0 }$ on $L^2(\R\times \Omega_\T)$. 
 \end{theorem}

We will first prove the following intermediate lemma:

\begin{lemma}
The semigroup $|f_t'(0)|^{-\frac{Q^2}{2}}P_t^0$ with $P_t^0$ defined by \eqref{definitionPt0} admits $\mu_h$ as an invariant measure and the measure $\mu_h$ is absolutely continuous with respect to $\mu_0$.
\end{lemma}

\proof
We first want to apply  Example 3.8.15 in \cite{boga} to show convergence in law of $\varphi_t$ towards $X_h+Q \log \frac{\omega}{|v|}$ as $t$ tends to infinity.
Let $s>0$. For all  $g(\theta)= \sum_{k \in \Z^\ast}  z_k e^{ik \theta} \in H^{-s}(\T)$ and $g'(\theta)= \sum_{k \in \Z^\ast} z_k' e^{ik \theta}\in H^{-s}(\T)$, we have
\[
\E\Big[  \cjg   g , (X_\D \circ f_t)|_{\T} \cjd_{H^{-s}(\T)} \cjg   g' ,  (X_\D \circ f_t)|_{\T} \cjd_{H^{-s}(\T)} \Big] =  \frac{1}{(2 \pi)^2}\int_{0}^{2 \pi} \int_0^{2 \pi} \tilde{g}(\theta) \tilde{g}'(\theta')  \log \frac{| 1-f_t(e^{i\theta})\bbar{f_t(e^{i\theta'})}|}{|f_t(e^{i \theta})-f_t(e^{i \theta'})  |}  d \theta d\theta'
\]
where $\tilde{g}(\theta)= \sum_{k \in \Z^\ast} (1+|k|)^{-2s} z_k e^{ik \theta} $ and $\tilde{g}'(\theta')= \sum_{k \in \Z^\ast} (1+|k|)^{-2s} z_k' e^{ik \theta'}$. Since $g,g'$ have average $0$,
\begin{equation*}
\E\Big[  \cjg   g ,  (X_\D \circ f_t)|_{\T} \cjd_{H^{-s}(\T)} \cjg   g' ,  (X_\D \circ f_t)|_{\T} \cjd_{H^{-s}(\T)}\Big ] =  \frac{1}{(2 \pi)^2}\int_{0}^{2 \pi} \int_0^{2 \pi} \tilde{g}(\theta) \tilde{g}'(\theta')  \log \frac{| 1-f_t(e^{i\theta})\bbar{f_t(e^{i\theta'})}|}{|e^{\omega t} f_t(e^{i \theta})-e^{\omega t}f_t(e^{i \theta'})  |}  d \theta d\theta'.
\end{equation*}
We have the following convergence (recall $f_t(z)\to 0$ and $e^{\omega t}f_t(z)\to h(z)$ as $t\to \infty$ uniformly in $z$)
\begin{equation}\label{HSconvergence}
\int_{0}^{2 \pi} \int_0^{2 \pi}  \Big|   \log \frac{| 1-f_t(e^{i\theta})\bbar{f_t(e^{i\theta'})}|}{|e^{\omega t} f_t(e^{i \theta})-e^{\omega t}f_t(e^{i \theta'})  |}  -  \log \frac{1}{|h(e^{i \theta})-h(e^{i \theta'})  |} \Big|^2 \dd \theta \dd \theta'  \underset{t \to \infty}{ \rightarrow}  0 .
\end{equation}
Now, since $  \|  \tilde{g}   \|^2_{L^2(\T)}  \leq  \|  g   \|_{H^{-2s}(\T)}^2\leq   \|  g   \|_{H^{-s}(\T)}^2$, we get from the convergence \eqref{HSconvergence} that 
\begin{align*}
& \sup_{  \|  g   \|_{H^{s}(\T)}  \leq 1   }   \big | \E[   \cjg   g , (X_\D \circ f_t)|_{\T} \cjd_{H^{-s}(\T)}^2  ]    -  \E[   \cjg   g , X_h \cjd_{H^{-s}(\T)}^2  ]  \big|  \\  
&  \leq  \sup_{  \|  \tilde{g}   \|_{L^2(\T)}^2  \leq 1   }  \Big|  \frac{1}{(2 \pi)^2}\int_{0}^{2 \pi} \int_0^{2 \pi} \tilde{g}(\theta) \tilde{g}(\theta') \Big (  \log \frac{| 1-f_t(e^{i\theta})\bbar{f_{t}(e^{i\theta'})}|}{|e^{\omega t} f_t(e^{i \theta})-e^{\omega t}f_t(e^{i \theta'})  |}-  \log \frac{1}{|h(e^{i \theta})-h(e^{i \theta'})  |}  \Big )  d \theta d\theta'   \Big|  \underset{t \to \infty}{ \rightarrow}  0 .
\end{align*}
This establishes (ii) of Example 3.8.15 in \cite{boga}. 
We also have that
\begin{align*}
 \E[     \| (X_\D \circ f_t) |_{\T}   \|_{H^{-s}(\T)}^2   ]    & = \frac{1}{(2 \pi)^2}    \int_{0}^{2 \pi} \int_0^{2 \pi}    G_{\D}   (  f_t(e^{i \theta}),  f_t(e^{i \theta'})   )  \Big (   \sum_{n \in \Z} \frac{\cjg n\cjd^{-2s}}{2\pi} e^{in (\theta-\theta')}  \Big )    \dd \theta   \dd \theta'  \\
& =\frac{1}{(2 \pi)^2}    \int_{0}^{2 \pi} \int_0^{2 \pi}    \log |1- f_t(e^{i \theta}) \overline{f_t(e^{i \theta'})}  |    \Big (   \sum_{n \in \Z} \frac{\cjg n\cjd^{-2s}}{2\pi} e^{in (\theta-\theta')}  \Big )    \dd \theta   \dd \theta'  \\
& + \frac{1}{(2 \pi)^2}    \int_{0}^{2 \pi} \int_0^{2 \pi}     \log \frac{1}{|e^{\omega t} f_t(e^{i \theta})-e^{\omega t}f_t(e^{i \theta'})  |}  \Big (   \sum_{n \in \Z} \frac{\cjg n\cjd^{-2s}}{2\pi} e^{in (\theta-\theta')}  \Big )    \dd \theta   \dd \theta'.
\end{align*}
The first term has norm bounded by $C e^{-2\omega t} \int_{0}^{2 \pi} \int_0^{2 \pi}      \left (   \sum_{n \in \Z} \frac{\cjg n\cjd^{-2s}}{2\pi} e^{in (\theta-\theta')}  \right )    \dd \theta   \dd \theta'  $ hence converges to $0$ as $t$ goes to infinity and the second term converges to $ \frac{1}{(2 \pi)^2}    \int_{0}^{2 \pi} \int_0^{2 \pi}     \log \frac{1}{|h(e^{i \theta})-h(e^{i \theta'})  |}  \left (   \sum_{n \in \Z} \frac{\cjg n\cjd^{-2s}}{2\pi} e^{in (\theta-\theta')}  \right )    \dd \theta   \dd \theta' $. This establishes (iii) of Example 3.8.15 in \cite{boga}.

Finally, using \eqref{limitandheq} and the fact that $P\varphi(0)=0$,  the sequence $(P\varphi \circ f_t+Q \log \frac{|f'_t|}{|f_t|}) |_{\T}$ converges in $H^{-s}(\T)$ towards $Q \log \frac{\omega}{|v|}_{\T}$ as $t$ goes to infinity; this establishes (i) of Example 3.8.15 in \cite{boga}. Hence we have established convergence  of $\varphi_t$ towards $X_h+Q \log \frac{\omega}{|v||_{\T}}$. 

Now we turn to the convergence of the couple $(B_t=X_0(t), \varphi_t)$. We denote by $\varphi_{n,t}$ the Fourier modes of $\varphi_t$ and  $X_n(t)$ the random part defined in \eqref{X_n(t)}.

Consider a continuous function $(u,\varphi) \to F(u,\varphi)$ ($u$ stands for the constant part) such that there exists a function $G(u)$ such that $G(u)=\mc{O}((1+|u|^2)^{-N})$ for all $N>0$ and such that $|F(u,\varphi)| \leq G(u)$. Let us decompose 
\begin{align*}
& \varphi_{n,t}= \tilde{\varphi}_{n,t}+ c_{n,t} B_t, \quad \textrm{ with }\\ 
& c_{n,t}:= \frac{\E[ B_t  X_{n}(t) ]}{\omega t}=\frac{1}{\omega t (2 \pi)^2} \int_{0}^{2 \pi} \int_{0}^{2\pi}  \log  \frac{|1- f_t(e^{i \theta}) \overline{f_t(e^{i \theta'})}  |    }{| e^{\omega t}f_t(e^{i \theta}) -e^{\omega t}f_t(e^{i \theta'})  |}   e^{-in \theta} \dd \theta \dd \theta'=\mc{O}(t^{-1})
\end{align*}
and $\tilde{\varphi}_{n,t}$ is independent of $B_t$. Let $c_t= \sum_{n\in \Z^\ast}c_{n,t}e^{in\theta}$ 
and $\tilde{\varphi}_t =\sum_{n\in \Z^\ast }\tilde{\varphi}_{n,t}e^{in\theta}$.
We have by conditioning on $B_t$ 
\begin{align*}
|  \E[  \sqrt{2 \pi \omega t} F(c+B_t, \varphi_t) - \int_{\R}  F(u, c_t u+ \tilde{\varphi}_t) \dd u ] |  & \leq   \E\Big[   \int_{\R} | e^{-\frac{(u-c)^2}{2 \omega t}}  -1| |F(u, c_t (u-c)+ \tilde{\varphi}_t)| \dd u \Big]   \\
& \leq   \int_{\R} | e^{-\frac{(u-c)^2}{2 \omega t}}  -1| G(u)  \dd u \underset{t \to \infty}{\rightarrow}  0   .
\end{align*}
Now, one can show that $\tilde{\varphi}_t$ converges in law towards $ X_h+Q \log \frac{\omega}{|v|}$ (since $\varphi_t-\tilde{\varphi}_t$ converges in probability towards $0$) and $c_t$ converges as $t$ goes to infinity towards $0$ in $H^s(\T)$.
We have
\begin{equation*}
\E[  \int_{\R}  F(u, c_t(u-c)+ \tilde{\varphi}_t) \dd u  ]=  \int_{\R} \E[   F(u, c_t(u-c)+ \tilde{\varphi}_t)   ]  \dd u  \underset{t \to \infty}{\rightarrow}  \int_{\R} \E[   F(u, X_h-Q \log \frac{|v||_{\T}}{\omega})   ]  \dd u
\end{equation*}
where we have used the weak convergence at $u$ fixed and the dominated convergence theorem.

If we fix $t_0>0$ the continuous function $(u, \varphi) \to P_{t_0}F(u, \varphi)$ satisfies also that its norm is dominated by a positive 
function $\tilde{G}(u)=\mc{O}((1+ |u|^2)^{-N})$ for all $N>0$, hence we get that
\begin{equation*}
 \sqrt{2 \pi \omega t} P_t(P_{t_0}F)(c,\varphi)  \underset{t \to \infty}{\rightarrow}  \int_{\R} \E[   (P_{t_0}F)(u, X_h-Q \log \frac{|v||_{\T}}{\omega})   ]  \dd u.
\end{equation*} 
By using the semigroup property, we deduce the following identity
\begin{equation}\label{invariance}
\int_{\R} \E[   (P_{t_0}F)(u, X_h)   ]  \dd u=  \int_{\R} \E[   F(u, X_h)   ]  \dd u .
\end{equation}
To summarize, we have proved that \eqref{invariance} holds  for all continuous function $(u,\varphi) \to F(u,\varphi)$ such that there exists a function 
$G(u)=\mc{O}((1+|u|^2)^{-N})$ for all $N>0$ and for which $|F(u,\varphi)| \leq G(u)$. We can extend the above identity to all bounded Borelian functions by a density argument.

Finally, we have to show that the measure $\mu_h$ is absolutely continuous with respect to $\mu_0$. In order to do so, we have to show that $\P_h$ is absolutely continuous with respect to $\P_\T$. We have the following lemma:
\begin{lemma}\label{secondlemma1}
There exists a constant $C>0$ and $\rho \in (0,1)$ such that for all $n, p \in \Z^*$
\begin{equation}\label{secondlemma1eq}
 \Big|\int_0^{2 \pi} \int_0^{2 \pi}   \log \frac{|e^{i \theta}- e^{i \theta'}|}{|h(e^{i \theta})-h(e^{i \theta'})  |}  e^{- in \theta} e^{-ip \theta'} \dd \theta \dd \theta' \Big| \leq  C \rho^{|n|+|p|}.
\end{equation}
\end{lemma}
\begin{proof} 
The function $h$ can be extended to a univalent function on $(1+\delta) \D$ for some $\delta >0$. Therefore, the holomorphic function $(z,w) \mapsto \frac{h(z)- h(w)}{z-w}$ is equal to $\exp( \sum_{j,k \geq 0}  a_{j,k}  z^j w^k)$ where  
 $|a_{j,k}| \leq C \frac{1}{(1+\delta/2)^{j+k}}$. Relation \eqref{secondlemma1eq} can then be seen by writing $ \log \frac{|h(e^{i \theta})-h(e^{i \theta'})  |} {|e^{i \theta}- e^{i \theta'}|}= {\rm Re} \left (  \sum_{j,k \geq 0}  a_{j,k}  e^{ij \theta} e^{i k \theta'}   \right )$.
\end{proof}
Let $G_{\rm Id}$ be the operator on $L^2(\T)$ whose integral kernel is given by the covariance $-\log|e^{i\theta}-e^{i\theta'}|$ of $X_{{\rm Id}}$: a direct computation shows that it is equal to the Fourier multiplier $G_{\rm Id}=\pi |D|^{-1}$ where we set $|D|^{-s}e^{in\theta}:= |n|^{-s}e^{in\theta}$ and $|D|^{-s}1:=0$ for $s\in \R$. Let $G_h$ be the operator on $L^2(\T)$ whose integral kernel is given by the covariance of $X_h$. By 
Lemma \ref{secondlemma1}, we have 
\[ G_{h}=G_{\rm Id}+W\]
where $W$ is a smoothing operator, bounded as operators $H^{-N}(\T)\to H^N(\T)$ for all $N>0$. We can then write 
$G_h=G^{1/2}_{\rm Id}({\rm Id}+\tilde{W})G^{1/2}_{\rm Id}$ for some $\tilde{W}$ satisfying the same properties as $W$, and ${\rm Id}+\tilde{W}$ is a positive self-adjoint Fredholm operator on $H_0^{-1-2s}(\T)=\{f\in H^{-1-2s}(\T)\,|\, \cjg f,1\cjd=0\}$ for $s>0$ thus there is $C>0$ such that for all $f\in H_0^{-2s}(\T)$
\[  C^{-1}  \cjg G_{\rm Id}f,f\cjd_{L^2} \leq  \cjg G_hf,f\cjd_{L^2} \leq C \cjg G_{\rm Id}f,f\cjd_{L^2}.\]
Since $\E[ \cjg   g , X_{h}  |_{\T} \cjd_{H^{-s}(\T)}^2 ] = \cjg G_h |D|^{-2s}g,|D|^{-2s}g\cjd_{L^2}= \cjg G_hf,f\cjd_{L^2}$ with $f:=|D|^{-2s}g\in H_0^{-2s}(\T)$, we get
\begin{equation*}
C^{-1}\E[   \cjg   g , X_{\rm Id}  |_{\T} \cjd_{H^{-s}(\T)}^2 ] \leq \E[   \cjg   g , X_{h}  |_{\T} \cjd_{H^{-s}(\T)}^2 ]  \leq   C \E[   \cjg   g, X_{\rm Id}  |_{\T} \cjd_{H^{-s}(\T)}^2 ]
\end{equation*}
and so $X_h$ and $X_{\rm Id}$ have the same Cameron-Martin space. Therefore we can conclude that both fields yield equivalent probability measures by using the discussion at the bottom of p.294 in \cite{boga}.   \qed

\noindent
\emph{Proof of Theorem \ref{theoremfreefield}}. 
Consider the Hilbert space $(\mc{D}(\mc{Q}_0),\mc{Q}_0)$ and the quadratic form 
$\mc{Q}_{0}^{\bv}(F,F):= \cjg  \bH_\bv^0 F , F \cjd_2$ on $\mc{C}_{\rm exp}$. 

Using Lemma \ref{boundonLn}, there is $C>0$ such that for for all $F \in \cC_{\rm exp}$ 
\begin{align*}
 {\rm Re}(\mc{Q}_{0}^{\bv}(F,F)) & = \omega  \cjg  \bH^0 F , F \cjd_2 + {\rm Re}(\sum_{n \geq 1}  \Re (v_n)  \cjg   (\mathbf{L}_n^0+\widetilde{\mathbf{L}}_n^0 ) F  , F \cjd_2+ i \sum_{n \geq 1}  \Im (v_n)  \cjg   (\mathbf{L}_n^0-\widetilde{\mathbf{L}}_n^0 ) F  , F \cjd_2)   \\
& \geq \omega  \mc{Q}_0 (F, F)  - \sum_{n \geq 1} | \Re (v_n)  | | \cjg   (\mathbf{L}_n^0+\widetilde{\mathbf{L}}_n^0 ) F  , F \cjd_2  | - \sum_{n \geq 1}  | \Im (v_n) |  | \cjg   (\mathbf{L}_n^0-\widetilde{\mathbf{L}}_n^0 ) F  , F \cjd_2    \\
& \geq \omega  \mc{Q}_0 (F, F) - 2C\Big(  \sum_{n \geq 1} n^2  ( | \Re (v_n)  |+ |\Im v_n| )\Big )  \mc{Q}_0 (F, F)  \\
& \geq \Big( \omega - 2C \sum_{n \geq 1} n^2 |v_n|\Big)  \mc{Q}_0 (F, F).
\end{align*}
On the other hand one also has by the same argument 
\[  |\mc{Q}_{0}^{\bv}(F,F)| \leq \Big(\omega +2C\sum_{n \geq 1} n^2 |v_n|\Big)\mc{Q}_0(F,F).\]
Choosing  $\omega >2C \left (  \sum_{n \geq 1} n^2 |v_n|  \right )+1$, we see that there is $C_0>1$ such that 
\[  C_0^{-1}\mc{Q}_0(F,F) \leq |\mc{Q}_{0}^{\bv}(F,F)|  \leq C_0\mc{Q}_0(F,F)\]
which implies that $\mc{Q}_{0}^{\bv}$ is a closed quadratic form, that extends to $\mc{D}(\mc{Q}_0)$, 
and also that there is $\theta\in (0,\pi/2)$ so that 
 $|{\rm arg}(\mc{Q}_{0}^{\bv}(F,F)|\leq \theta$ for all $F\in \mc{D}(\mc{Q}_0)$. This means that $\mc{Q}_0^{\bv}$ is strictly m-accretive \cite[Chapter VIII.6]{rs1}. By Theorem VIII.16 and the following Lemma in \cite{rs1}, there is a unique closed operator extending ${\bf H}^0_{\bv}$ defined in a dense domain $\mc{D}({\bf H}^0_{\bv})\subset\mc{D}(\mc{Q}_0)$, with $({\bf H}^0_{\bv}-\lambda)^{-1}$ invertible  if ${\rm Re}(\la)<0$, and resolvent 
 bound $\|({\bf H}^0_{\bv}-\lambda)^{-1}\|_{\mc{H}\to \mc{H}}\leq (-{\rm Re}(\la))^{-1}$. By the Hille-Yosida theorem, ${\bf H}^0_{\bv}$ is the generator of a contraction semi-group denoted $e^{-t{\bf H}^0_{\bv}}$. 
Next, we notice that if $\mc{H}_{\R}$ is the real Hilbert space consisting of the real valued elements $F\in \mc{H}$, we can 
restrict  $\mc{Q}_0^{\bv}$ to  $\mc{D}_{\R}(\mc{Q}_0):= \mc{D}(\mc{Q}_0)\cap \mc{H}_{\R}$, then it is easily seen that 
 $\mc{Q}_{0}^{\bv}(F,F)>0$. In view of the discussion above, it is then a coercive closed form in the sense of \cite[Definition 2.4]{MaRockner}.
 
Now we show that $e^{-t{\bf H}_{\bv}^0}$ and $P_t^0$ coincide on $\cC_{\rm exp}$.
We first consider a function $F$ in $\mc{C}_{\rm exp}$ of the form $f(c, (\varphi_n)_{n \in [-N,N]})$. In this case, we get that for all $t$, with $B_t$ the Brownian motion,
\begin{equation*}
P_t^0F(c,\varphi)=  \E_{\varphi}  [ f(c+B_t, \varphi_{-N,t},\dots, \varphi_{N,t})    ]
\end{equation*}
where $\varphi_{n,t}$ denotes the Fourier modes of $\varphi_t$.
From the previous discussions on $f_t$, we have the series representation
\begin{equation*}
f_t(e^{i \theta})= e^{-\omega t} e^{i \theta}+ \sum_{j \geq 2} \alpha_j(t) e^{ij \theta}
\end{equation*}
for some $\alpha_j(t)=\mc{O}(e^{- \omega t})$ and therefore for $n \geq 1$
\begin{equation}\label{varphi_tn}
\begin{split}
\varphi_{n,t}= &  \frac{1}{2 \pi}  \int_0^{2 \pi}  P\varphi( f_t(e^{i \theta}) ) e^{-i n \theta} d \theta + \frac{1}{2 \pi}  \int_0^{2 \pi}  X_{\D}( f_t(e^{i \theta}) ) e^{- i n \theta} d \theta 
= e^{-\omega n t} \varphi_n+ \sum_{k=1}^{n-1} \tilde{\alpha}_k(t) \varphi_k + X_n(t)
\end{split}
\end{equation}
where   
\begin{equation*}
\tilde{\alpha}_k(t)= \sum_{\substack{j_1, \dots, j_k \geq 1\\ j_1+ \cdots +j_k=n}}  \alpha_{j_1}(t) \cdots \alpha_{j_k}(t). 
\end{equation*}
A similar relation holds for negative $n$ (with $\varphi_{-n}$ in place of $\varphi_n$). We deduce that 
for fixed $t$ the function $P_t^0F$ depends only on $(c,(\varphi_n)_{|n|\leq N})$. From \eqref{varphi_tn} and using that $B_t$ is a Brownian motion, for all 
$k,\beta,M$, there is $C>0$ (depending locally uniformly on $t$ and whose value changes from line to line) so that
\[\begin{split} 
|\pl_c^k \pl_{(x,y)}^\beta P_t^0F(c,\varphi)| \leq &  C\cjg (x,y)\cjd^L\E_{\varphi}[ e^{-M|c+B_t|}\cjg (X_{-N}(t),\dots,X_N(t))\cjd^L]\\
\leq & Ce^{-M|c|}\cjg (x,y)\cjd^L
\E_{\varphi}[ e^{M|B_t|}\cjg (X_{-N}(t),\dots,X_N(t))\cjd^L]\\
\leq & Ce^{-M|c|}\cjg (x,y)\cjd^L
\end{split}\]
thus $P_t^0F\in \mc{C}_{\rm exp}$, i.e. $P_t^0:\mc{C}_{\rm exp}\to \mc{C}_{\rm exp}$.

In this case, we can apply Propositions  \ref{prop:first_order} and \ref{prop:second_order} (recall the expression \eqref{newexpressionLno}) to show that the following holds in $L^2(\R\times \Omega_\T,\mu_0)$
\begin{equation}\label{itsevolutionbaby}
\partial_t P_t^0 F= \partial_{s}|_{s=0}  P_{s}^0 (  P_t^0F ) = -  \bH_\bv^0 P_t^0F, \quad P_{t}^0 F|_{t=0}= F
\end{equation}
This is because for fixed $t$ the function $P_t^0F$ belongs to $\mc{C}_{\rm exp}$. By using uniqueness of the solution in equation \eqref{itsevolutionbaby}, we see that for all $t \geq 0$ the identity $e^{-t{\bf H}_{\bv}^0}F=P_t^0F$ holds everywhere. This identity can be extended to $L^2(\R\times \Omega_\T)$ by a density argument.

Now, we prove the key relation \eqref{itsevolutionbaby}. This is the purpose of the following two Propositions:

\begin{proposition}\label{prop:first_order}
For all $\bv\in\mathrm{Vect}_\eps^+(\D)$, let us define the operator $\nabla_\bv$ on $\cC_{\rm exp}$ by
\[\nabla_\bv F(\varphi):=-\frac{\d}{\d t}_{|t=0} F\Big((P\varphi\circ f_t+Q\log\frac{|f_t'|}{|f_t|})|_{\T}\Big).\]
The above quantity exists in the classical sense for all $\varphi \in H^{s}(\T)$ and the limit which defines the derivative converges in $L^2$ if $F \in \cC_\mathrm{exp}$. We have $\nabla_\bv= \sum_{n \geq 0} v_n \nabla_n $ where
\[\nabla_n=\Re\left(Qn\del_n+2\sum_{m=1}^\infty m\varphi_m\del_{n+m}\right).\]
 The $\C$-linear and $\C$-antilinear parts of $\nabla_n$ are
\begin{align*}
&\nabla_n^{1,0}=\frac{1}{2}Qn\del_n+\sum_{m=1}^\infty m\varphi_m\del_{n+m},\quad \nabla_n^{0,1}=\frac{1}{2}Qn\del_{-n}+\sum_{m=1}^\infty m\varphi_{-m}\del_{-n-m}.
\end{align*}
\end{proposition}
\begin{proof}
Let $\varphi\in\cC^\infty(\T)$ (we can suppose that $\varphi\in\cC^\infty(\T)$ because $F$ depends on a finite number of variables), which we write in Fourier expansion
\[\varphi(\theta)=\sum_{n\in\Z}\varphi_n e^{in\theta},\]
where  $\varphi_{-n}=\bar{\varphi}_n$ for all $n\in\Z$.

Let $\bv=v\partial_z\in\mathrm{Vect}_\eps^+(\D)$ generating the infinitesimal deformation $f_t(z)=z+tv(z)+o(t)$. We have
\begin{align*}
P\varphi\circ f_t(z)+Q\log \frac{|f_t'(z)|}{|f_t(z)|}
&=P\varphi(z+tv(z))+Q\Re\log(1+tv'(z))-Q \log |z| - Q\Re\log\Big(1+t \frac{v(z)}{z}\Big)+o(t) \\
&=P\varphi(z)+2t{\rm Re}(v(z)\pl_z P\varphi(z))+tQ\Re(v'(z))-Q \log |z| -tQ \Re(\frac{v(z)}{z}) +o(t)\\
&=P\varphi(z)-Q\log |z|+t\Re\Big(2v(z)\partial_z\varphi(z)+Qv'(z)  -Q \frac{v(z)}{z}) \Big)+o(t).
\end{align*}

We can consider the case $\bv_n=-z^{n+1}\del_z$ since the general case is just a linear combination of this case. Specialising to $\bv_n=-z^{n+1}\del_z$, we obtain for all $k \in \Z$
\begin{equation*}
\begin{split}
\frac{1}{2 \pi}\int_{0}^{2 \pi} \left (  P\varphi\circ f_t(e^{i \theta})+Q\log \frac{|f_t'(e^{i \theta})|}{|f_t(e^{i \theta})|}  \right )  e^{-ik \theta}  d \theta = & 
\varphi_k -t(k-n) \varphi_{k-n}{\bf 1}_{k \geq n+1}- tn\frac{Q}{2}\delta_{|k|-n}\\
& +t(k+n) \varphi_{k+n}{\bf 1}_{k\leq -n-1}+o(t) 
\end{split}
\end{equation*}
For $F\in\cC_\mathrm{exp}$, we then deduce the following limit in $L^2$
\[ -\frac{\d}{\d t}_{|t=0} F\Big(\big(P\varphi\circ f_t+Q\log\frac{|f_t'|}{|f_t|}\big)\Big|_{\T}\Big)=\frac{nQ}{2}(\pl_nF+\pl_{-n}F)+\sum_{m\geq 1}
(m\varphi_m\pl_{n+m}+m\varphi_{-m}\pl_{-m-n})F.\qedhere\]
\end{proof}

\begin{proposition}\label{prop:second_order}
For all $\bv\in\mathrm{Vect}_+(\D)$, let us define the operator $\Delta_\bv$ on $\cC_\mathrm{exp}$ by
\[\Delta_\bv F(\varphi)=-\frac{\d}{\d t}_{|t=0}\E_\varphi\left[F\left(\varphi+(X_\D\circ f_t)|_{\T}\right)\right].\]
The above quantity exists in the classical sense for all $\varphi \in H^{s}(\T)$ and the limit which defines the derivative converges in $L^2$ if $F \in \cC_\mathrm{exp}$. This definition extends uniquely to all $\bv\in\mathrm{Vect}(\D)$, where $\Delta_n:=\Delta_{\bv_n}$ is given for all $n\geq0$ by
\[\Delta_n=-\frac{1}{2}\Re\left(\sum_{m\in\Z}\del_{n-m}\del_m\right).\]
The $\C$-linear and $\C$-antilinear parts of $\Delta_n$ are
\begin{align*}
&\Delta_n^+=\frac{1}{4}\sum_{m\in\Z}\del_{n-m}\del_m;\qquad\Delta_n^-=\frac{1}{4}\sum_{m\in\Z}\del_{-n-m}\del_m.
\end{align*}
\end{proposition}

The proof of this proposition relies on the following simple lemma about finite dimensional Gaussian random variables.

\begin{lemma}
Let $n\in\N$ and $C_t=t(\dot{C}+K_t)$ be covariance matrices on $\R^n$ such that $K_t\to0$ as $t\to0$. Let $\xi_t\sim\cN(0,\dot{C}+K_t)$. Then, for all $F\in\cC^\infty(\R^n)$ such that $F$ and its derivatives have at most exponential growth, for all $x\in\R^n$, we have
\[\underset{t\to0}{\lim}\,\frac{\E[F(x+\sqrt{t}\xi_t)]-F(x)}{t}=\frac{1}{2}\E[\mathrm{Hess}\,F_x(\xi_0,\xi_0)]\]
where ${\rm Hess}\, F_x$ denotes the Hessian of $F$ at $x$.
\end{lemma}
\begin{proof}
Let $\epsilon\in(0,\frac{1}{2(n+3)})$. We have
\begin{align*}
&|\E[F(x+\sqrt{t}\xi_t)]-F(x)-\frac{t}{2}\E[\mathrm{Hess}\,F_x(\xi_t,\xi_t)]|\\
&=|\E[F(x+\sqrt{t}\xi_t)-F(x)-\sqrt{t}\nabla F_x(\xi_t)-\frac{t}{2}\mathrm{Hess}\,F_x(\xi_t,\xi_t)]|\\
&\leq\E|F(x+\sqrt{t}\xi_t)-F(x)-\sqrt{t}\nabla F_x(\xi_t)-\frac{t}{2}\mathrm{Hess}\,F_x(\xi_t,\xi_t)|\\
&=\int_{\R^n}|F(x+\sqrt{t}u)-F(x)-\sqrt{t}\nabla F_x(u)-\frac{t}{2}\mathrm{Hess}\,F_x(u,u)|\frac{e^{-\frac{1}{2}\langle(\dot{C}+K_t)u,u\rangle}}{(2\pi)^{n/2}\det(\dot{C}+K_t)}\d u\\
&=\int_{|u|\leq t^{-\epsilon}}|F(x+\sqrt{t}u)-F(x)-\sqrt{t}\nabla F_x(u)-\frac{t}{2}\mathrm{Hess}\,F_x(u,u)|\frac{e^{-\frac{1}{2}\langle(\dot{C}+K_t)u,u\rangle}}{(2\pi)^{n/2}\det(\dot{C}+K_t)}\d u\\
&\qquad+\int_{|u|>t^{-\epsilon}}|F(x+\sqrt{t}u)-F(x)-\sqrt{t}\nabla F_x(u)-\frac{t}{2}\mathrm{Hess}\,F_x(u,u)|\frac{e^{-\frac{1}{2}\langle(\dot{C}+K_t)u,u\rangle}}{(2\pi)^{n/2}\det(\dot{C}+K_t)}\d u\\
&\leq M_1 t^{3(\frac{1}{2}-\epsilon)-n\epsilon}+M_2 e^{-at^{-2\epsilon}}=O(t^{\frac{3}{2}-\epsilon(n+3)})=o(t),
\end{align*}
for some constants $M_1,M_2,a>0$. For the first bound, we used that $F$ is $C^3$ around $x$, so the integrand is uniformly $O(t^{3(\frac{1}{2}-\epsilon)})$ in the region $|u|\leq t^{-\epsilon}$. For the second bound, we used the exponential bound on the growth of $F$ and some elementary bounds on the Gaussian density in the region $|u|>t^{-\epsilon}$. 

Finally, we have $|\E[\mathrm{Hess}\,F_x(\xi_t,\xi_t)-\mathrm{Hess}\,F_x(\xi_0,\xi_0)]|=\mc{O}(\norm{K_t})=o(1)$.
\end{proof}

\begin{proof}[Proof of Proposition \ref{prop:second_order}]
Since $F\in\cC_\mathrm{exp}$, we will apply the previous lemma where the covariance matrices are those of the finite-dimensional distributions of $(X_\D\circ f_t)|_\T$. By the lemma, it suffices to compute the derivative at $t=0$ of this covariance matrix.

We write
\[(X_\D\circ f_t)(e^{i \theta})=\sum_{n\in\Z}X_n(t)e^{i n \theta},\]
where $X_n(t)$ is given by \eqref{X_n(t)} with covariance \eqref{covarianceXn}. First, we show that the covariance is differentiable with respect to $t$ and that the differential exists at $t=0$. First, we have for $z\neq z'\in\bbar{\D}$,
\begin{align*}
\del_t\log|f_t(z)-f_t(z')|
=\del_t\Re\log(f_t(z)-f_t(z'))
=\Re\frac{\del_tf_t(z)-\del_tf_t(z')}{f_t(z)-f_t(z')}
=\Re\frac{v(f_t(z))-v(f_t(z'))}{f_t(z)-f_t(z')}.
\end{align*}
This term has the finite limit ${\rm Re}(f_t'(z)v'(f_t(z)))$ as $z'\to z$, and  $\del_t\log|f_t(z)-f_t(z')|$ is uniformly bounded with respect to $t$ (in compact sets) on $\bbar{\D}^2$. Applying the dominated convergence theorem, we have
\begin{align*}
\del_t\int_0^{2\pi}\int_0^{2\pi}\log|f_t(e^{i\theta})-f_t(e^{i\theta'})|e^{-ip\theta}e^{iq\theta'}\frac{\d\theta\d\theta'}{4\pi^2}
&=\int_0^{2\pi}\int_0^{2\pi}\Re\left(\frac{v(f_t(e^{i\theta}))-v(f_t(e^{i\theta'}))}{f_t(e^{i\theta})-f_t(e^{i\theta'})}\right)e^{-ip\theta}e^{iq\theta'}\frac{\d\theta\d\theta'}{4\pi^2}\\
&\underset{t\to0}{\to}\int_0^{2\pi}\int_0^{2\pi}\Re\left(\frac{v(e^{i\theta})-v(e^{i\theta'})}{e^{i\theta}-e^{i\theta'}}\right)e^{-ip\theta}e^{iq\theta'}\frac{\d\theta\d\theta'}{4\pi^2}.
\end{align*}

Now we treat the term $\log|1-f_t(z)\overline{f_t(z')}|$, which is the real part of a holomorphic (resp. antiholomorphic) function in $z$ (resp. $z'$) converging in the disc. The values $\int\int\log(1-f_t(e^{i\theta})\overline{f_t(e^{i\theta'})})e^{-ip\theta}e^{iq\theta'}\frac{\d\theta\d\theta'}{4\pi^2}$ are the coefficients of the power series expansion of this function at $0$. The derivatives with respect to $t$ of these coefficients are the coefficients of the function $\del_t\log(1-f_t(z)\overline{f_t(z')})=-\frac{v(f_t(z))\overline{f_t(z')}+f_t(z)\overline{v(f_t(z'))}}{1-f_t(z)\overline{f_t(z')}}$. As $t\to0$, this function converges uniformly on compact sets of $\D^2$ to $-\frac{v(z)\bar{z}'+z\bar{v}(z')}{1-z\bar{z}'}$, which is holomorphic (resp. antiholomorphic) in $z$ (resp. $z'$). Therefore, the coefficients of the expansion converge individually. Moreover, the limiting function has a continuation to $\bbar{\D}^2$ with simple poles at $z=z'\in\T$. Thus,
\[\del_t|_{t=0}\int_{[0,2\pi]^2}\log|1-f_t(e^{i\theta})\bbar{f_t(e^{i\theta'})}|e^{-ip\theta+iq\theta'}\frac{\d\theta\d\theta'}{4\pi^2}=-\int_{[0,2\pi]^2}\Re\left(\frac{v(e^{i\theta})e^{-i\theta'}+e^{i\theta}\bbar{v(e^{i\theta'})}}{1-e^{i(\theta-\theta')}}\right)e^{-ip\theta}e^{iq\theta'}\frac{\d\theta\d\theta'}{4\pi^2},\]
where the integral is understood in the principal value sense (the formula computes the coefficients of the power series expansion of the integrand).

Putting the two contributions together, we get
\begin{equation}\label{eq:correl_asymp}
\begin{aligned}
h_{p,q}
&:=\del_t|_{t=0}\E\left[X_p(t)X_{-q}(t)\right]\\
&=-\int_0^{2\pi}\int_0^{2\pi}\Re\left(\frac{v(e^{i\theta})-v(e^{i\theta'})}{e^{i\theta}-e^{i\theta'}}+\frac{e^{i\theta}\bbar{v(e^{i\theta'})}}{1-e^{i(\theta-\theta')}}+\frac{e^{-i\theta'}v(e^{i\theta})}{1-e^{i(\theta-\theta')}}\right)e^{-ip\theta}e^{iq\theta'}\frac{\d\theta\d\theta'}{4\pi^2}.
\end{aligned}
\end{equation}

To get explicit formulas for these operators, we need to compute $h_{p,q}$ for $v(z)=-z-z^{n+1}$, $n\geq0$. Since the derivative of the 
 covariance this is linear in $v$, it suffices to compute each term in \eqref{eq:correl_asymp} separately for $v(z)=-z$ and $v(z)=-z^{n+1}$. We thus do it for $v(z)=-z^{n+1}$ for $n\geq 0$. The first term is given by
\begin{equation}\label{eq:first_term}
\begin{aligned}
\int_0^{2\pi}\int_0^{2\pi}\frac{e^{i(n+1)\theta}-e^{i(n+1)\theta'}}{e^{i\theta}-e^{i\theta'}}e^{-ip\theta}e^{iq\theta'}\frac{\d\theta\d\theta'}{4\pi^2}
&=\sum_{k=0}^n\int_0^{2\pi}\int_0^{2\pi}e^{i(k-p)\theta}e^{i(n-k+q)\theta'}\frac{\d\theta\d\theta'}{4\pi^2}\\
&=\left\lbrace\begin{aligned}
&1\text{ if }0\leq p\leq n\text{ and }q=p-n\\
&0\text{ otherwise}.
\end{aligned}\right.
\end{aligned}
\end{equation}
and the conjugate term by
\begin{equation}\label{eq:conjugate}
\begin{aligned}
\int_0^{2\pi}\int_0^{2\pi}\frac{e^{-i(n+1)\theta}-e^{-i(n+1)\theta'}}{e^{-i\theta}-e^{-i\theta'}}e^{-ip\theta}e^{iq\theta'}\frac{\d\theta\d\theta'}{4\pi^2}
&=\left\lbrace\begin{aligned}
&1\text{ if }-n\leq p\leq0\text{ and }q=p+n\\
&0\text{ otherwise}.
\end{aligned}\right.
\end{aligned}
\end{equation}
To compute the other terms, we look at the power series expansion of the integrand and extract the relevant coefficient. For instance, we have
\begin{align*}
\int_0^{2\pi}\int_0^{2\pi}\frac{e^{i\theta}e^{-i(n+1)\theta'}}{1-e^{i(\theta-\theta')}}e^{-ip\theta}e^{iq\theta'}\frac{\d\theta\d\theta'}{4\pi^2}
&=\int_0^{2\pi}\int_0^{2\pi}\sum_{k=0}^{\infty}e^{-i(p-1-k)\theta}e^{i(q-n-1-k)\theta'}\frac{\d\theta\d\theta'}{4\pi^2}\\
&=\left\lbrace\begin{aligned}
&1\text{ if }p>0\text{ and }q=p+n\\
&0\text{ otherwise}.
\end{aligned}\right.
\end{align*}
We proceed similarly for the other terms and we collect all the results below:
\begin{equation}\label{eq:residues}
\begin{aligned}
&\int_0^{2\pi}\int_0^{2\pi}\frac{e^{-i\theta}e^{i(n+1)\theta'}}{1-e^{-i(\theta-\theta')}}e^{-ip\theta}e^{iq\theta'}\frac{\d\theta\d\theta'}{4\pi^2}=\left\lbrace\begin{aligned}
&1\text{ if }p<0\text{ and }q=p-n\\
&0\text{ otherwise}.
\end{aligned}\right.\\
&\int_0^{2\pi}\int_0^{2\pi}\frac{e^{-i\theta'}e^{i(n+1)\theta}}{1-e^{i(\theta-\theta')}}e^{-ip\theta}e^{iq\theta'}\frac{\d\theta\d\theta'}{4\pi^2}=\left\lbrace\begin{aligned}
&1\text{ if }p>n\text{ and }q=p-n\\
&0\text{ otherwise}.
\end{aligned}\right.\\
&\int_0^{2\pi}\int_0^{2\pi}\frac{e^{i\theta'}e^{-i(n+1)\theta}}{1-e^{-i(\theta-\theta')}}e^{-ip\theta}e^{iq\theta'}\frac{\d\theta\d\theta'}{4\pi^2}=\left\lbrace\begin{aligned}
&1\text{ if }p<-n\text{ and }q=p+n\\
&0\text{ otherwise}.
\end{aligned}\right.
\end{aligned}
\end{equation}

 For Markovian $\bv$, the operator $\Delta_\bv$ is the generator of a (complex) Brownian motion with covariance matrix $(h_{p,q})_{p,q\in\Z}$, i.e.
\[\Delta_\bv=\frac{1}{2}\sum_{p,q\in\Z}h_{p,q}\del_p\del_{-q}.\]
For $F\in\cC_\mathrm{exp}$, this reduces to a finite sum, and we deduce the claim about the convergence to $\Delta_\bv F$ in $L^2$.

Now, summing up all six terms appearing in \eqref{eq:first_term}, \eqref{eq:conjugate} and \eqref{eq:residues}, we obtain for $\bv=\bv_n$:
\begin{equation}\label{eq:covariance_wn}
h_{p,q}=\frac{1}{2}(\delta_{q,p-n}+\delta_{q,p+n})
\end{equation}
so that
\begin{align*}
\Delta_n=\frac{1}{4}\sum_{p,q\in\Z}(\delta_{q,p+n}+\delta_{q,p-n})\del_p\del_{-q}
&=\frac{1}{4}\sum_{m\in\Z}\del_{n-m}\del_m+\del_{-n-m}\del_m\\
&=\frac{1}{2}\Re\left(\sum_{m\in\Z}\del_{n-m}\del_m\right).
\end{align*}

Notice that \eqref{eq:correl_asymp} is the real part of a complex linear map in $\bv$; this gives the decomposition $\Delta_\bv=\Delta_\bv^++\Delta_\bv^-$. Retaining only the complex linear (resp. antilinear) terms in $\bv$, we have
\[h_{p,q}^+=\frac{1}{2}\delta_{q,p-n};\qquad h_{p,q}^-=\frac{1}{2}\delta_{q,p+n},\]
so that
\[\Delta_n^+=\frac{1}{4}\sum_{m\in\Z}\del_{n-m}\del_m;\qquad\Delta_n^-=\frac{1}{4}\sum_{m\in\Z}\del_{-n-m}\del_m.\qedhere\]
\end{proof}

\section{Liouville operators}\label{subsec:operators}

In this Section, we construct a representation of the Virasoro algebra for the Liouville theory by adding a potential term in the Markovian 
process introduced above for the Gaussian Free Field, producing a Feynman-Kac for the Liouville semi-groups. The potential is the mass 
of $\D\setminus f_t(\D)$ for the Gaussian Multiplicative Chaos measure if $f_t$ is the flow of the considered Markovian vector field.  
We show that the generators of the Liouville semi-groups induce a representation of the Virasoro algebra, and we compare the construction 
of the Liouville descendant states from \cite{GKRV} (constructed using the stress-energy tensor and an intertwining procedure between the free field theory and the Liouville theory) with the iterated applications of our Virasoro operators to the primary fields. 

\subsection{Definition and Feynman-Kac formula}
	
In this section we define the Liouville operators. These will be densely defined, unbounded operators on the Liouville Hilbert space $L^2(\d c\otimes\P)$. The operators associated with Markovian vector fields are closable operators. They form the natural generalisation of the Liouville Hamiltonian of \cite[Section 5]{GKRV}, and they are the basis for the definition of operators associated with non-Markovian vector fields.

Let $\bv=v(z)\pl_z\in\mathrm{Vect}_\eps^+(\D)$ with $v(z)=-\sum_{n=-1}^\infty v_nz^{n+1}$ such that $v(0)=0$, $v'(0)=-\omega$ and satisfying the conditions of Theorem \ref{theoremfreefield}. Recall that $\omega> K \sum_{n \geq 1}|v_n|n^2 $ where $K>0$ is the constant which appears in Theorem \ref{theoremfreefield}. We consider the positive measure $\d\varrho_\bv(\theta)=-\Re(e^{-i\theta}v(e^{i\theta}))\d\theta$. One can then define the Dirichlet form on $\cC_{\rm exp}$ via the formula
\begin{equation}\label{defform}
\mathcal{Q}_{\bv}(F,G): = \mc{Q}_{\bv }^0(F,G)+ \mu \cjg e^{\gamma c} V_\bv  F , G \cjd_2 ,\quad \textrm{with } V_\bv:= \int_{0}^{2 \pi} e^{\gamma\varphi(\theta)}\d\varrho_\bv(\theta).
 \end{equation}  
 When $\gamma<\sqrt{2}$, this potential is an $L^{p}(\Omega_\T)$ function if $1\leq p<2/\gamma^2$, while when $\gamma\in [\sqrt{2},2)$, we define the potential term in \eqref{defform} by using the Girsanov formula to shift the variables
similarly to \cite[Section 5]{GKRV}. More specifically for $F,G \in \cC_{\rm exp}$ of the form $F(c,\varphi)=F(c,x_1,y_1, \cdots, x_N,y_N)$ and $G(c,\varphi)=G(c,x_1,y_1, \cdots, x_N,y_N)$ , we set the following definition 
\begin{align*}
 & \cjg e^{\gamma c} V_\bv  F , F \cjd_2   := \\
 & \int_{\R} \int_0^{2 \pi} e^{\gamma c } \E\big[ \big(F(c,x_1+ \gamma \cos (\theta),y_1- \gamma \sin (\theta) , \cdots, x_N+ \frac{\gamma}{\sqrt{N}} \cos (N \theta) ,y_N-  \frac{\gamma}{\sqrt{N}} \sin (N \theta)\big)^2  \big] \d\varrho_\bv(\theta) \d c.
\end{align*}
The general case $\cjg e^{\gamma c} V_\bv  F , G \cjd_2.$ with $F,G \in \cC_{\rm exp}$ can be obtained by polarization. Since $\d\varrho_\bv$ has a positive smooth density with respect to $\d\theta$, there is a constant $C>1$ such that $C^{-1}V \leq V_\bv \leq CV$ as bounded operators $\cD(\cQ)\to\cD'(\cQ)$ and therefore we have for all $F \in \cC_{\rm exp}$
 \begin{equation}\label{fundinequality}
 C^{-1} \mathcal{Q}(F,F) \leq \mathcal{Q}_{\bv}(F,F) \leq C \mathcal{Q}(F,F).
 \end{equation} 
The form  $\mathcal{Q}_{\bv}$ is a positive definite bilinear form satisfying the strong sector condition, namely there exists $K>0$ such that for all $F,G \in \cD(\cQ)$, the (modified) Cauchy-Schwarz inequality holds $$\mathcal{Q}_{\bv}(F,G) \leq K \mathcal{Q}_{\bv}(F,F)^{1/2} \mathcal{Q}_{\bv}(G,G)^{1/2}.$$ By inequality \eqref{fundinequality} and using the fact that $\mathcal{Q}$ is closable, we know that $ \mathcal{Q}_{\bv}$ is closable and the domain of $\mathcal{Q}_{\bv}$ coincides with the domain of $\mathcal{Q}$. Hence  $\mathcal{Q}_{\bv}$ is a coercive closed form: there exists a unique continuous contraction semigroup associated to $\mathcal{Q}_{\bv}$, 
denoted $e^{-t \bH_\bv }$ and its generator 
\[{\bf H}_{\bv}: \mc{D}(\mc{Q})\to \mc{D}'(\mc{Q})\] 
is the operator  
\begin{equation}\label{defHv}  
\begin{split}
\bH_\bv=&  \bH_\bv^0+  \mu e^{\gamma c} V_{\bf v}= \omega  \bH^0  + \sum_{n \geq 1}  \Re (v_n)  (\mathbf{L}_n^0+\widetilde{\mathbf{L}}_n^0 )   + i \sum_{n \geq 1}  \Im (v_n) (\mathbf{L}_n^0-\widetilde{\mathbf{L}}_n^0)+  \mu e^{\gamma c} V_{\bf v}\\
=& \omega {\bf H}+ \sum_{n\geq 1}v_n\, \mathbf{L}_n  + \sum_{n\geq 1}\bbar{v_n} \, \widetilde{\mathbf{L}}_n
\end{split}
\end{equation}   
where 
\begin{equation}\label{formulaLn}  
\mathbf{L}_n =  \mathbf{L}_n^0+ \frac{\mu}{2}e^{\gamma c}\int_0^{2\pi} e^{\gamma \varphi(\theta)}e^{in\theta}d\theta,\quad  \widetilde{\mathbf{L}}_n =  
\widetilde{\mathbf{L}}_n^0+ \frac{\mu}{2}e^{\gamma c}\int_0^{2\pi} e^{\gamma \varphi(\theta)}e^{-in\theta}d\theta.
\end{equation}
Notice that ${\bf L}_n$ can be recovered as a sum of operators ${\bf H}_{\bv}$ with $\bv$ Markovian by the formula 
\begin{equation}\label{Ln_en_terme_de_Hv_n} 
{\bf L}_n=\frac{1}{2}({\bf H}_{\omega \bv_0+\bv_n}-i{\bf H}_{\omega\bv_0+i\bv_n})-\frac{1}{2}\omega(1-i){\bf H}.
\end{equation}
The adjoint ${\bf H}_{\bv}^*:\mc{D}(\mc{Q})\to \mc{D}'(\mc{Q})$ of ${\bf H}_{\bv}$ is defined by the identity: for all $F,G\in \mc{D}(\mc{Q})$
\[ \cjg {\bf H}_{\bv}^*F,G\cjd_2=\cjg F, {\bf H}_{\bv}G\cjd_2,\]
and we set for $n\geq 1$
\[ {\bf L}_{-n}:={\bf L}_n^* : \mc{D}(\mc{Q})\to \mc{D}'(\mc{Q}), \quad \tilde{{\bf L}}_{-n}:=\tilde{{\bf L}}_n^*:\mc{D}(\mc{Q})\to \mc{D}'(\mc{Q}).\]
Now, we consider the following operator $P_t$, for $t\geq 0,$ defined for $F$ bounded on $H^{-s}(\T)$ (with $s>0$) by the formula:
\begin{equation} \label{FeynmanKac}
  P_t F(c,\varphi)   :=    |f'_t(0)|^{\frac{Q^2}{2}}    \E_\varphi \Big [  F\Big(   (c+X \circ f_t  + Q \ln | \frac { f_t'} {f_t} | )|_{\T}    \Big)     e^{- \mu e^{\gamma c} \int_{\D_t^c}    e^{\gamma X(x)} \frac{1} {|x| ^{\gamma Q} } \dd x  }  \Big]  
 \end{equation}
where $\D_t^c:=\D\setminus f_t(\D)$, $X=P\varphi+X_{\D}$ is the Gaussian Free Field from \eqref{decomposGFF} (we use that  $X_{\D^*}\circ f_t|_{\T}=0$ and $X_{\D^*}|_{\D_t^c}=0$).
We prove the following important property of $P_t$:  
\begin{theorem}\label{FeynmanKacPt}
The operator $P_t$ is a semigroup, which coincides with $e^{-t \bH_\bv}$ on $L^2(\R \times \Omega_\T)$.
\end{theorem} 
\begin{proof} 
 The map $(s,\theta)\mapsto f_s(e^{i\theta})$ is a diffeomorphism from $(0,t)\times \T$ to $\D_t^c$, and its Jacobian is given by 
 $-\Re (e^{- i \theta} v(f_s (e^{i \theta}))\bbar{f'_s(e^{i \theta})})$. Therefore, for $\epsilon>0$ we get that 
 \begin{equation*}
  \int_{\D_t^c}  \epsilon^{\frac{\gamma^2}{2}}  e^{\gamma X_\epsilon(x)} \frac{1} {|x| ^{\gamma Q} } \dd x= -\int_{0}^t  \int_{0}^{2 \pi} \epsilon^{\frac{\gamma^2}{2}} 
  |f_s(e^{i\theta})|^{-\gamma Q}e^{\gamma   X_{\epsilon} (f_s(e^{i \theta})) }  \Re \Big(e^{- i \theta} v(f_s (e^{i \theta})) \bbar{f'_s(e^{i \theta})}\Big) \dd \theta \dd s
 \end{equation*}
where the cutoff $X_\epsilon$ is defined via circle averages  $X_\epsilon(x)= \frac{1}{2 \pi} \int_0^{2 \pi}  X(x+ \epsilon e^{i \theta})  \dd \theta$. Now we have up to a neglectable term $X_{\epsilon} (f_s(e^{i \theta}))=  (X \circ f_s) _{\epsilon_{s, \theta}} (e^{i \theta})$ where $\epsilon_{s, \theta}=  \frac{\epsilon}{|f_s'(e^{i \theta})|}$. This leads to (up to a neglectable term) 
 \begin{equation*}
  \int_{\D_t^c}  \epsilon^{\frac{\gamma^2}{2}}  e^{\gamma X_\epsilon(x)} \frac{1} {|x| ^{\gamma Q} } \dd x= -\int_{0}^t  \int_{0}^{2 \pi} \epsilon_{s, \theta}^{\frac{\gamma^2}{2}} e^{\gamma  ( (X \circ f_s) _{\epsilon_{s, \theta}} (e^{i \theta}) + Q \log \frac{|f_s'(e^{i \theta}|}{|f_s(e^{i \theta})|}  )}  \Re \Big(e^{- i \theta}   \frac{v(f_s (e^{i \theta}))}{f'_s(e^{i \theta})}  \Big) \dd \theta \dd s.
 \end{equation*}
Since $\log \frac{|f_s'(e^{i \theta}|}{|f_s(e^{i \theta})|}  $ is regular, we get by a change of variables on the cutoff that
\begin{equation}\label{polardecomposition}
 \int_{\D_t^c}    e^{\gamma X(x)} \frac{1} {|x| ^{\gamma Q} } \dd x= -\underset{\epsilon \to 0}{\lim}  \int_{0}^t  \int_{0}^{2 \pi} \epsilon^{\frac{\gamma^2}{2}} e^{\gamma \varphi_{s,\epsilon}(\theta)}  \Re \Big(e^{- i \theta} \frac{v(f_s (e^{i \theta}))}{f'_s(e^{i \theta})}\Big) \dd \theta \dd s
\end{equation}
where we recall that $\varphi_t(\theta)= (X \circ f_t)(e^{i \theta})   + Q \ln | \frac { f_t'(e^{i \theta})} {f_t(e^{i \theta})}  |$ and $\varphi_{t,\epsilon}(\theta)$ is its regularization at scale $\eps$. 
Recall from Lemma \ref{limitandh} that $\frac{v(f_s (e^{i \theta}))}{f'_s(e^{i \theta})}$ does not depend on $s$ and therefore
\begin{equation}\label{polardecompositionctd}
 \int_{\D^c_t}    e^{\gamma X(x)} \frac{1} {|x| ^{\gamma Q} } \dd x= -\underset{\epsilon \to 0}{\lim}  \int_{0}^t  \int_{0}^{2 \pi} \epsilon^{\gamma^2/2} e^{\gamma \varphi_{s,\epsilon}(\theta)}  \Re (e^{- i \theta} v(e^{i \theta})) \dd \theta \dd s .
\end{equation}
The above formula shows that $P_t$ is a semigroup since it is obtained from $P_t^0$ via a Feynman-Kac type formula.
 Thanks to the free case, we know that for $F \in \cC_{\rm exp}$,  $\|P_tF\|_{2}\leq \|P_t^0F\|_{2}\leq  \|F\|_2$ and hence $P_t$ can be extended on $L^2$ into a contraction semigroup $\tilde{T}_t$. By adapting the arguments of \cite[Section 5]{GKRV} one can show that for all $F,G \in \cC_{\rm exp}$ one has
 \begin{equation*}
  \Big\cjg  \frac{(I-\tilde{T}_t)F}{t} , G \Big\cjd_2= \Big\cjg  \frac{(I-P_t)F}{t} , G \Big\cjd_2  \underset{t \to 0}{\rightarrow}  \cjg  \bH_\bv F , G \cjd_2 . 
 \end{equation*} 
The above convergence is based on the the weak convergence of measures $\frac{\ind_{\D_t^c}}{t}\d x\to\varrho_\bv$ on $\bar{\D}$, and the rest of the proof follows \cite[Section 5]{GKRV} since we know that $\frac{(I-P_t^0)F}{t}$ converges in $L^2$ to $\bH_\bv^0 F$ as $t$ goes to $0$.

Let $P_t^N$ be the semigroup  of the form \eqref{FeynmanKac} where one works with the first $N$ harmonics of $\varphi$ and similarly for $\bH_\bv^N, \mathcal{Q}_{\bv}^N$. One can adapt the arguments of \cite[Section 5]{GKRV} to show that $P_t^N=e^{-t\bH_\bv^N}$ (see Proposition 5.3 in \cite{GKRV}). In the sequel, we will consider the semigroups and the quadratic forms on the space of real functions (this no restriction by a linearity argument). Let $\lambda>0$ and $F \in \mathcal C_\mathrm{exp}$ be real valued. We have for all real valued $G \in \mathcal C_\mathrm{exp}$
\begin{equation*}
\lambda \cjg  R_\lambda^N F  , G \cjd_2+ \mathcal{Q}_{\bv}^N(R_\lambda^N F,G)=  \cjg F  ,G \cjd_2
\end{equation*} 
where $R_\lambda^N$ is the resolvent associated to $P_t^N$. Since  $R_\lambda^N F$ and $G$ depends on the first $N$ harmonics for $N$ large, the following holds by definition of $\mathcal{Q}_{\bv}$ if $G$ depends on the first $N$ harmonics
\begin{equation}\label{quadrelation}
\lambda \cjg  R_\lambda^N F  , G \cjd_2+ \mathcal{Q}_{\bv}(R_\lambda^N F, G)=  \cjg F  , G \cjd_2.
\end{equation} 
The above identity for $G= R_\lambda^N F$ shows that $\sup_{N \geq 1} \mathcal{Q}_{\bv}(R_\lambda^N F,R_\lambda^N F ) < \infty$. Moreover, by adapting  \cite[Section 5]{GKRV} (see proof of Proposition 5.5), we know that $R_\lambda^N F$ converges in $L^2(\R \times \Omega_\T)$ to $R_\lambda F$ where $R_\lambda$ is the resolvent associated to $P_t$. Therefore, we can apply Lemma 2.12 in \cite{MaRockner} which yields that $R_\lambda F \in \mathcal{D}(\mathcal{Q}_{\bv} )$ and one can find a subsequence $(R_\lambda^{N_k}F)_{k \geq 1}$ which converges in Cesaro mean to $R_\lambda F$ with respect to $\mathcal{Q}_{\bv}$. Hence, we can take the limit in \eqref{quadrelation} which yields for all $G \in \mathcal C_\mathrm{exp}$
\begin{equation}\label{quadrelationlimit}
\lambda \cjg  R_\lambda F , G \cjd_2+ \mathcal{Q}_{\bv}(R_\lambda F, G)=  \cjg F , G \cjd_2.
\end{equation} 
Since $\mathcal C_\mathrm{exp}$ is dense in the domain of  $\mathcal{Q}_{\bv}$, this shows that $R_\lambda$ is equal to the resolvent associated to $\mathcal{Q}_{\bv}$ on $\mathcal{C}_\mathrm{exp}$. Since $\mathcal{C}_\mathrm{exp}$ is dense in $L^2(\R \times \Omega_\T)$, this implies that $R_\lambda$ is the resolvent associated to  $\mathcal{Q}_{\bv}$. By uniqueness of the resolvent, we have proved the desired result.
\end{proof}

\subsection{Liouville descendant states}\label{subsec:poisson}

In this section we recall the definition and properties of the Poisson operator introduced in \cite[Section 6]{GKRV}. This will be used to describe the diagonalization of the Liouville Hamiltonian and obtain a representation of the Virasoro algebra.
 
\subsubsection{The spectral Riemann surface}

Let $\mc{D}_0:=\{\alpha_{\pm j}:=Q\pm i \sqrt{2j} \, |\, j\in \N\}$ and consider the non-compact Riemann surface $\Sigma$ so that $r_j(\alpha):=\sqrt{-(\alpha-Q)^2-2j}$ are 
holomorphic functions for all $2j>0$ and so that for each $\alpha\in \Sigma$, ${\rm Im}(\sqrt{-(\alpha-Q)^2-2j})>0$ for $j\in\N$ large enough. Here we use the convention that for $j=0$, $r_0(\alpha)=-i(\alpha-Q)$.
The surface $\Sigma$ is as a ramified covering $\pi:\Sigma\to \C$ with ramifications of order $2$ at $\mc{D}_0$. We identify 
the half-space $\{\alpha \in\C \,|\, {\rm Re}(\alpha)<Q\}$ with a subset of $\Sigma$ using a complex embedding by the 
section $s$ so that ${\rm Im}(\sqrt{-(s(\alpha)-Q)^2-2j})>0$ for all $j\in\N_0$ if ${\rm Re}(\alpha)<Q$. 
The group $G$ of automorphisms of the covering is abelian and generated by elements $(\gamma_{\pm j})_j$, of order $2$, corresponding to simple 
loops around $\alpha_{\pm j}$ and we have $r_j(\gamma_j(\alpha))=-r_j(\alpha)$.

\subsubsection{Poisson operator}\label{sec:Poisson_Operator}
Let us define the family of Poisson operators $\mc{P}_\ell(\alpha)$ as  in \cite[Section 6]{GKRV}. 
First, we use the finite dimensional space for $\ell\in \N_0$
\begin{equation}\label{def_E_ell}
E_\ell:= \bigoplus_{j=0}^{\ell} \ker ({\bf P}-j) \subset L^2(\Omega_\T ,\mathbb{P}_\T).
\end{equation}
and let $\Pi_{E_\ell}$ and $\Pi_{\ker ({\bf P}-j)}$ be the orthogonal projections on respectively $E_\ell$ and $\ker ({\bf P}-j)$. 

We next choose a function $\rho\in C^\infty(\R)$ satisfying $\rho(c)=c$ for $c\leq -1$ and $\rho(c)=0$ for $c\geq 0$. We consider the space $e^{-\beta \rho}\mc{D}(\mc{Q})$ as the space of functions $F$ such that $e^{\beta \rho}F \in \mc{D}(\mc{Q})$ and it is equipped with the norm $\|F\|_{e^{\beta \rho}\mc{D}(\mc{Q})}:=\|e^{\beta \rho}F\|_{\mc{D}(\mc{Q})}$ . Below, we use the determination of the logarithm with a cut at $\R^+$: we set 
$\sqrt{Re^{i\theta}}=\sqrt{R}e^{i\theta/2}$ for $\theta \in [0,2\pi)$ and $R>0$.

\begin{proposition}[Proposition 6.19 in \cite{GKRV}]
\label{poissonprop}
Let $0<\beta<\gamma/2$ and $\ell\in\N$.
Then there is an analytic family of operators $\mc{P}_\ell(\alpha)$  
\[\mc{P}_\ell(\alpha): E_\ell\to e^{-\beta \rho}\mc{D}(\mc{Q})\]
in the region 
\begin{equation}\label{regionvalide}
\Big\{\alpha=Q+ip\in \C\,\Big|\, {\rm Re}(\alpha)< Q, {\rm Im}\sqrt{p^2-2\ell}<\beta\Big\}\cup \Big\{Q+ip\in Q+i\R\, \Big| \, |p|\in \bigcup_{j\geq \ell} (\sqrt{2j},\sqrt{2(j+1)})\Big\},
\end{equation}
continuous at each $Q\pm i\sqrt{2j}$ for $j\geq \ell$, satisfying $(\mathbf{H}-\tfrac{Q^2+p^2}{2})\mc{P}_\ell(\alpha)F=0$ and in the region $\{c\leq 0\}$
\begin{equation}\label{expansionP} 
\mc{P}_\ell(\alpha)F=\sum_{j\leq \ell}\Big(F^-_je^{ic\sqrt{p^2-2j}}+F_j^+(\alpha)
e^{-ic\sqrt{p^2-2j}}\Big)+G_\ell(\alpha,F)
\end{equation}
with $F^-_j=\Pi_{\ker(\mathbf{P}-j)}F$, $F_j^+(\alpha)\in \ker (\mathbf{P}-j)$, and
$G_\ell(\alpha,F)\in \mc{D}(\mc{Q})$ can be decomposed into
\[G_\ell(\alpha,F)=G_\ell^1(\alpha,F)+G_\ell^2(\alpha,F), \quad G_\ell^1(\alpha,F)\in e^{\frac{\beta}{2}\rho}L^2(\R\times\Omega_\T), \,\, \Pi_\ell(G_\ell^2(\alpha,F))=0.\] 
For $F\in E_\ell$, $u=\mc{P}_\ell(\alpha)F$ is the unique solution in $e^{-\beta\rho}\mc{D}(\mc{Q})$ 
to the equation $(\mathbf{H}-\tfrac{Q^2+p^2}{2})u=0$ satisfying that for $\chi\in C^\infty(\R)$equal to $1$ in $c\leq -1$ and with support in $(-\infty,0)$, there are 
$F_j^+(\alpha)\in \ker (\mathbf{P}-j)$ such that
\[\chi(c)\Big(u-\sum_{j\leq \ell}\Big((\Pi_{\ker(\mathbf{P}-j)}F)e^{ic\sqrt{p^2-2j}}+F_j^+(\alpha)
e^{-ic\sqrt{p^2-2j}}\Big)\Big)\in \mc{D}(\mc{Q)}.\]
Finally, $F_j^+(\alpha)$ depends analytically on $\alpha$ in the region \eqref{regionvalide} and the operator $\mc{P}_\ell(\alpha)$ admits a meromorphic extension to the region  
\begin{equation}\label{regiondextension}
\Big\{\alpha=Q+ip \in \Sigma\, |\,  \forall j\leq \ell, {\rm Im}\sqrt{p^2-2j}\in (\beta/2-\gamma,\beta/2)\Big\}\end{equation}
and $\mc{P}_\ell(\alpha)F$ satisfies \eqref{expansionP} in that region. 
\end{proposition}
We notice that, by uniqueness of the expansion \eqref{expansionP}, for $j<\ell$ and $|p|>\sqrt{2\ell}$, we have 
\begin{equation}\label{egalitepoissonell} 
\mc{P}_\ell(Q+ip)|_{\ker ({\bf P}-j)}= \mc{P}_{j}(Q+ip)|_{\ker ({\bf P}-j)}.
\end{equation}

\subsubsection{Verma modules and highest weight vectors for the free field}\label{Vermafreefield}

We shall use a particular eigenbasis of ${\bf P}$ that is parametrized by the Young diagrams  and that is adapted to the Virasoro algebra. 
Recall that a Young diagram $\nu$ is a non-increasing finite sequence of integers $\nu_1\geq \dots\geq \nu_k>0$. We denote by $\mc{T}$ the union of $\{0\}$ and the set of Young diagrams and we denote by $|\nu|=\sum_{j}\nu_j$ its length. We shall write $\mc{T}_n:=\{\nu \in \mc{T}\,|\, |\nu|=n\}$. 
Notice that $\mc{T}$ can be viewed as a subset of $\mc{N}$, where $\mc{N}$ is the set 
of sequences $(k_n)_{n\geq 0}\in \N_0^{\N_0}$ so that there is $n_0$ with  $k_n=0$ for all $n\geq n_0$.
Given two elements $\nu= (\nu_i)_{i \in [1,k]}\in\mc{N}$ and $\tilde{\nu}= (\tilde{\nu}_i)_{i \in [1,j]}\in\mc{N}$ 
we define the operators on $\mc{C}_{\rm exp}$ 
\begin{equation}\label{LnuL-nu}
\begin{gathered}
 \mathbf{L}_{-\nu}^0:=\mathbf{L}_{-\nu_k}^0 \cdots \, \mathbf{L}_{-\nu_1}^0,  \qquad  \tilde{\mathbf{L}}_{-\tilde \nu}^0:=\tilde{\mathbf{L}}_{-\tilde\nu_j}^0 \cdots\, \tilde{\mathbf{L}}_{-\tilde\nu_1}^0\\
 \mathbf{L}_{\nu}^0:=\mathbf{L}_{\nu_k}^0 \cdots \, \mathbf{L}_{\nu_1}^0,   \qquad  \tilde{\mathbf{L}}_{\tilde \nu}^0:=\tilde{\mathbf{L}}_{\tilde\nu_j}^0 \cdots\, \tilde{\mathbf{L}}_{\tilde\nu_1}^0
\end{gathered}\end{equation}
and define the functions
\begin{align}\label{psibasis}
\Psi^0_{\alpha,\nu, \tilde\nu}:=\mathbf{L}_{-\nu}^0\tilde{\mathbf{L}}_{-\tilde\nu}^0  \: \Psi^0_\alpha,
\end{align}
with the convention that $\Psi^0_{\alpha,0,0}:=\Psi^0_{\alpha}=e^{(\alpha-Q)c}$. The following 3 properties are proved in \cite[Proposition 4.9]{GKRV}: first,
\[ {\bf H}^0\Psi^0_{\alpha,\nu, \tilde\nu}=(2\Delta_\alpha+|\nu|+|\tilde{\nu}|)\Psi^0_{\alpha,\nu, \tilde\nu}\quad  \textrm{ where recall that }\Delta_\alpha =\frac{\alpha}{2}(Q-\frac{\alpha}{2}).\]
then, for each $\nu,\tilde{\nu}\in \mc{T}$, there is a polynomial $\mc{Q}_{\alpha,\nu,\tilde{\nu}}$ in the variables $(x_n,y_n)_n$ so that 
\begin{equation}\label{descendantfree}
\Psi^0_{\alpha,\nu, \tilde\nu}=\mc{Q}_{\alpha,\nu,\tilde{\nu}}\Psi_{\alpha}^0, \qquad {\bf P} \mc{Q}_{\alpha,\nu,\tilde{\nu}}=(|\nu|+|\tilde{\nu}|) \mc{Q}_{\alpha,\nu,\tilde{\nu}}.
\end{equation}
and finally ${\rm span}\{  \mc{Q}_{\alpha,\nu,\tilde{\nu}}\, |\, \nu,\tilde{\nu}\in \mc{T},\,  |\nu|+|\tilde{\nu}|=\ell\}=\ker({\bf P}-\ell)$ if $\alpha\notin Q-\frac{\gamma}{2}\N-\frac{2}{\gamma}\N$. More generally, a relation of the form \eqref{descendantfree} holds if $\nu,\tilde{\nu}\in \mc{N}$ with 
$\mc{Q}_{\alpha,\nu,\tilde{\nu}}\in \ker ({\bf P}-|\nu|-|\tilde{\nu}|)$ where $|\nu|=\sum_j \nu_j$ and $|\tilde{\nu}|=\sum_j \tilde{\nu}_j$. In fact if
$n_1,\dots,n_k\in \N$ and $\tilde{n}_1,\dots,\tilde{n}_{k'}\in \N$, then
\[
{\bf L}_{-n_k}^0\dots {\bf L}_{-n_1}^0\tilde{{\bf L}}_{-\tilde{n}_{k'}}^0\dots \tilde{{\bf L}}_{-\tilde{n}_1}^0\Psi^0_\alpha\in {\rm span}\Big\{ \Psi^0_{\alpha,\nu,\tilde{\nu}}\,|\, \nu,\tilde{\nu}\in \mc{T}, |\nu|=\sum_{j=0}^kn_j, \; |\tilde{\nu}|=\sum_{j=0}^{k'}\tilde{n}_j\Big\}
\]
if $\alpha\notin Q-\frac{\gamma}{2}\N-\frac{2}{\gamma}\N$. In particular, for $n\in \Z$,  if $\nu,\tilde{\nu}\in\mc{T}$, there are 
coefficients $\ell_{n}^\alpha (\nu,\nu')\in \C$, that are $0$ when $|\nu|-n\not=|\nu'|$, so that 
\begin{equation}\label{descendant_general}  
{\bf L}_{n}^0 \Psi_{\alpha,\nu,\tilde{\nu}}^0=\sum_{\nu'\in \mc{T}}\ell_{n}^\alpha (\nu,\nu')\Psi_{\alpha,\nu',\tilde{\nu}}^0 , 
\quad \tilde{{\bf L}}_{n}^0 \Psi_{\alpha,\nu,\tilde{\nu}}^0=\sum_{\tilde{\nu}'\in \mc{T}}\ell_{n}^\alpha (\tilde{\nu},\tilde{\nu}')\Psi_{\alpha,\nu,\tilde{\nu}'}^0.
\end{equation}
Moreover, when $0<n\leq \nu_k$ and $\nu=(\nu_1,\dots,\nu_k)$, then $\ell_{-n}^{\alpha}(\nu,\nu')=\delta_{\nu',(\nu,n)}$, while when $n>|\nu|$, 
$\ell_{n}^{\alpha}(\nu,\cdot)=0$. 
The functions $\ell_{n}^{\alpha}$ can be computed by 
commuting recursively ${\bf L}^0_{n}$ with the ${\bf L}_{-\nu_j}^0$ for those $\nu_j\leq n$ when writing $\Psi_{\alpha,\nu,\tilde{\nu}}^0$ under the form \eqref{psibasis}. In the proof of \cite[Proposition 4.9]{GKRV}, it is shown that 
\[ {\bf L}_{n}^0 \Psi_{\alpha,\nu,\tilde{\nu}}^0= e^{(\alpha-Q)c} {\bf L}_{n}^{0,\alpha}\mc{Q}_{\alpha,\nu,\tilde{\nu}},\qquad \tilde{{\bf L}}_{n}^0 \Psi_{\alpha,\nu,\tilde{\nu}}^0= e^{(\alpha-Q)c} \tilde{{\bf L}}_{n}^{0,\alpha}\mc{Q}_{\alpha,\nu,\tilde{\nu}}\]
where ${\bf L}_n^{0,\alpha}$ (and similarly for $\tilde{{\bf L}}_n^{0,\alpha}$) 
is an unbounded operator acting on $L^2(\Omega_\T)$ (i.e only in the $(x_n,y_n)_n$ variables but not on $c$) and well defined on $\mc{S}\subset L^2(\Omega_\T)$ ($\mc{S}$ is defined in Section \ref{freeham}); it also satisfies $({\bf L}_n^{0,\alpha})^*={\bf L}_{-n}^{0,2Q-\bar{\alpha}}$ on $L^2(\Omega_\T)$. By \eqref{descendant_general}
\begin{equation}\label{Ln0alpha}
 \sum_{\nu'\in \mc{T}}\ell_{n}^{\alpha}(\nu,\nu')\mc{Q}_{\alpha,\nu',\tilde{\nu}}={\bf L}_n^{0,\alpha}\mc{Q}_{\alpha,\nu,\tilde{\nu}}, \qquad 
\sum_{\tilde{\nu}'\in \mc{T}}\ell_{n}^{\alpha}(\tilde{\nu},\tilde{\nu}')\mc{Q}_{\alpha,\nu,\tilde{\nu}'}=\tilde{{\bf L}}_n^{0,\alpha}\mc{Q}_{\alpha,\nu,\tilde{\nu}}.
\end{equation}
Following \cite[Section 6.4]{GKRV}, we define the Hilbert space 
\[ \mc{H}_{\mc{T}}=\bigoplus_{n=0}^\infty \C^{d_n} , \quad \cjg v,v'\cjd_{\mc{T}}=\sum_{n=0}^\infty \cjg v_n,v_n'\cjd_{d_n}\] 
where $d_n=\sharp \mc{T}_n$, and $\cjg \cdot,\cdot\cjd_{d_n}$ is the standard Hermitian product on $\C^{d_n}$.  
If $(e_{\nu})_{\nu \in \mc{T}}$ denotes 
the canonical Hilbert basis of $\mc{H}_{\mc{T}}$, we see that for $\alpha\notin Q-\frac{\gamma}{2}\N_0-\frac{2}{\gamma}\N_0$
\[ \mc{I}_{\alpha}: \mc{H}_{\mc{T}}\otimes \mc{H}_{\mc{T}} \to L^2(\Omega_\T), \quad e_{\nu}\otimes e_{\tilde{\nu}} \mapsto \mc{Q}_{\alpha,\nu,\tilde{\nu}}\]
is an isomorphism, but it is not a unitary map since $(\mc{Q}_{\alpha,\nu,\tilde{\nu}})_{\nu,\tilde{\nu}}$ is not orthonormal. 
By \cite[Proposition 4.9]{GKRV} we have for $P>0$
\[ \cjg \mc{Q}_{Q+iP,\nu,\tilde{\nu}}, \mc{Q}_{Q+iP,\nu',\tilde{\nu}'}\cjd_{L^2(\Omega_\T)}=F_{Q+iP}(\nu,\nu')F_{Q+iP}(\tilde{\nu},\tilde{\nu}') \]
where $(F_{Q+iP}(\nu,\nu'))_{\nu,\nu'\in \mc{T}}$ are the so-called \emph{Shapovalov matrices}, which satisfy $F_{Q+iP}(\nu,\nu')=0$ if $|\nu|\not=|\nu'|$ and so that 
$(F_{Q+iP}(\nu,\nu'))|_{\nu,\nu'\in \mc{T}_n}$ are invertible for all $n\in \N$; the inverse of $F_{Q+iP}$ is denoted $F_{Q+iP}^{-1}$. 
Therefore, if we equip $\mc{H}_{\mc{T}}$ with the following scalar product
\[ \cjg u,u'\cjd_{Q+iP}:= \sum_{\nu,\nu'\in\mc{T}} F_{Q+iP}(\nu,\nu') u_\nu \bar{u}_{\nu'} ,\quad \textrm{for } u=\sum_{\nu} u_\nu e_\nu, \,\,  u'=\sum_{\nu} u'_\nu e_\nu ,\]
the map $\mc{I}_{Q+iP}: (\mc{H}_{\mc{T}}\otimes \mc{H}_{\mc{T}}, \cjg\cdot,\cdot\cjd_{Q+iP}\otimes \cjg\cdot,\cdot\cjd_{Q+iP}) \to L^2(\Omega_\T)$ becomes a unitary isomorphism of Hilbert spaces. 
Let us consider the linear maps for $\alpha\notin Q-\frac{\gamma}{2}\N-\frac{2}{\gamma}\N$
\begin{equation}\label{def:ell_nalpha}
\ell_n^{\alpha}:\mc{H}_{\mc{T}}\to \mc{H}_{\mc{T}},\quad 
\ell_n^{\alpha}(e_{\nu})=\sum_{\nu'\in \mc{T}}\ell_n^{\alpha}(\nu,\nu')e_{\nu'} 
\end{equation}
and we also call $\ell_n^\alpha$ (resp. $\tilde{\ell}_n^\alpha$) the action of $\ell_n^\alpha$ on $\mc{H}_{\mc{T}}\otimes \mc{H}_{\mc{T}}$ on the left variable 
(resp.  on $\mc{H}_{\mc{T}}\otimes \mc{H}_{\mc{T}}$ on the right variable).
Notice  from \eqref{Ln0alpha} that 
$\mc{I}_{\alpha}\ell_n^{\alpha}={\bf L}^{0,\alpha}_n\mc{I}_{\alpha}$ and $\mc{I}_{\alpha}\tilde{\ell}_n^{\alpha}=\tilde{{\bf L}}^{0,\alpha}_n\mc{I}_{\alpha}$.
From \eqref{Ln0alpha} and the fact that $({\bf L}_n^{0,Q+iP})^*={\bf L}_{-n}^{0,Q+iP}$ for $P>0$, we conclude that $(\ell_n^{Q+iP})^*=\ell_{-n}^{Q+iP}$ on $(\mc{H}_{\mc{T}}, \cjg\cdot,\cdot\cjd_{Q+iP})$ and 
\begin{equation}\label{ellnadjoint}
(\ell_n^{Q+iP})^*=\ell_{-n}^{Q+iP}, \quad  (\tilde{\ell}_n^{\, Q+iP})^*=\tilde{\ell}_{-n}^{\, Q+iP} \textrm{ on }(\mc{H}_{\mc{T}}\otimes \mc{H}_{\mc{T}}, \cjg\cdot,\cdot\cjd_{Q+iP})\otimes \cjg\cdot,\cdot\cjd_{Q+iP}).
\end{equation}

Let us now define the vector spaces $\mc{V}_\alpha^0, \bbar{\mc{V}}_\alpha^0$  by 
\[ \begin{gathered}
\mc{V}_\alpha^0:= {\rm span}\{ \Psi^0_{\alpha,\nu,0}\,|\, \nu \in \mc{T} \} \subset e^{(|{\rm Re}(\alpha)-Q|+\eps)|\rho(c)|}L^2(\R\times \Omega_\T), \\
\bbar{\mc{V}}_\alpha^0:= {\rm span}\{ \Psi^0_{\alpha,0,\tilde{\nu}}\,|\, \nu \in \mc{T} \} \subset e^{(|{\rm Re}(\alpha)-Q|+\eps)|\rho(c)|}L^2(\R\times \Omega_\T)
\end{gathered}\]
In the langage of representation theory (see \cite[Chapter 6]{Schottenloher}), if $\alpha\notin Q-\frac{\gamma}{2}\N_0-\frac{2}{\gamma}\N_0$ the state $\Psi_{\alpha}^0$ is a \emph{highest weight vector} associated to the highest weight representation of Virasoro algebra ${\rm Vir}(c_L)$ with central charge $c_L=1+6Q^2$ generated by the operators ${\bf L}_n^0$ for $n\in \Z$ and $c_L$, 
and $\mc{V}_\alpha^0$ is an associated \emph{Verma module} (and similarly for $\bbar{\mc{V}}_\alpha^0$):
\[ \left \{ \begin{array}{lll}
 {\bf L}_n^{0}\Psi_\alpha^0=0, \,\, \forall n>0,  \\
{\bf L}_0^0 \Psi_{\alpha}^0= \Delta_{\alpha} \Psi_{\alpha}^0
 \end{array}\right. ,\qquad 
\left \{ \begin{array}{lll}
 \tilde{{\bf L}}_n^{0}\Psi_\alpha^0=0, \,\, \forall n>0,  \\
\tilde{{\bf L}}_0^0 \Psi^0_{\alpha}= \Delta_{\alpha} \Psi^0_{\alpha} 
 \end{array}\right..
 \]
 For $\alpha=Q+iP$ with $P>0$, we equip $\mc{V}_{Q+iP}^{0}$ and $\bbar{\mc{V}}_{Q+iP}^{0}$ with the scalar product 
\begin{equation}\label{scalarproduct}
\cjg F,F'\cjd_{\mc{V}^0_{Q+iP}}:= \cjg e^{-iPc}F, e^{-iPc}F'\cjd_{L^2(\Omega_\T)}.
\end{equation} 
Here, we recall from \eqref{descendantfree} that $e^{-iPc}\Psi_{Q+ip,\nu,0}=\mc{Q}_{Q+ip,\nu,0}$ can be viewed as an element in $L^2(\Omega_{\mathbb{T}})$ since it is independent of $c$ (thus the same holds for any $F\in \mc{V}_{Q+iP}^{0}$), and the discussion above shows that \eqref{scalarproduct} is non-degenerate if $P>0$ since it is given in terms of the Schapovalov matrices,
and the two representations ${\bf L}_n^0$ in  $\mc{V}_{Q+iP}^{0}$  and $\tilde{\bf L}_n^0$ in $\bbar{\mc{V}}_{Q+iP}^{0}$ are then unitary.   
The direct sum ${\rm Vir}(c_L)\oplus {\rm Vir}(c_L)$ also has a representation into 
\[\mc{W}^{0}_{\alpha}:={\rm span}\{ \Psi^0_{\alpha,\nu,\tilde{\nu}}\,|\, \nu,\tilde{\nu} \in \mc{T} \}\] 
given by the action of $({\bf L}_n^0)_n$ and $(\tilde{\bf L}_n^0)_n$ on the basis elements $\Psi^0_{\alpha,\nu,\tilde{\nu}}$; 
notice that the representation of $({\bf L}_n^0)_n$ on $\mc{W}^{0}_{\alpha}$ produces canonically the representation of 
$({\bf L}_n^0)_n$ on $\mc{V}_\alpha^0$ by using the inclusion $\mc{V}_\alpha^0\subset \mc{W}_\alpha^0$. 
We use again the expression \eqref{scalarproduct} for the scalar product on $\mc{W}^{0}_{Q+iP}$.
We see that 
\[ \left\{\begin{array}{rcl} 
(\mc{H}_{\mc{T}}, \cjg\cdot,\cdot\cjd_{Q+iP})& \to  & (\mc{V}_{Q+iP}^{0},\cjg \cdot,\cdot \cjd_{\mc{V}^0_{Q+iP}}) \\
 e_{\nu}& \mapsto & \Psi^0_{Q+iP,\nu,0}
\end{array}\right.,\quad 
\left\{\begin{array}{rcl} 
(\mc{H}_{\mc{T}}, \cjg\cdot,\cdot\cjd_{Q+iP})& \to  & (\mc{V}_{Q+iP}^{0},\cjg \cdot,\cdot \cjd_{\mc{V}^0_{Q+iP}}) \\
 e_{\tilde{\nu}}& \mapsto & \Psi^0_{Q+iP,0,\tilde{\nu}}
\end{array}\right.
\]
are unitary isomorphisms, which conjugate respectively ${\bf L}_{n}^0$ and $\tilde{\bf L}_{n}^0$ with $\ell_{n}^{Q+iP}$. 
One also has a unitary isomorphism 
\[\left\{\begin{array}{rcl} 
(\mc{H}_{\mc{T}}\otimes \mc{H}_{\mc{T}}, \cjg\cdot,\cdot\cjd_{Q+iP}\otimes  \cjg\cdot,\cdot\cjd_{Q+iP})& \to  & (\mc{W}_{Q+iP}^0,\cjg \cdot,\cdot \cjd_{\mc{V}^0_{Q+iP}}) \\
 e_{\nu}\otimes e_{\tilde{\nu}}& \mapsto & \Psi^0_{Q+iP,\nu,\tilde{\nu}}
\end{array}\right.\]
which allows us to identify 
\[\mc{V}_{Q+iP}^{0}\otimes \bbar{\mc{V}}_{Q+iP}^{0} \simeq \mc{W}^{0}_{Q+iP}.\]

Since for $\alpha\notin Q-\frac{\gamma}{2}\N_0-\frac{2}{\gamma}\N_0$, $\ell_{n}^{\alpha}$ maps a finite sum $\sum_{\nu,\tilde{\nu}\in \mc{T}_N}a_{\nu,\tilde{\nu}}e_{\nu}\otimes e_{\tilde{\nu}}$ to a finite sum of the same form with $N$ replaced by $N-n$, we can compose $\ell_n^\alpha \ell_m^\alpha$ on such finite sums, and their commutators are given by the Virasoro commutation laws 
\[ [ \ell^{\alpha}_n,\ell^{\alpha}_m]=(n-m)\ell^{\alpha}_{n+m}+\frac{c_L}{12}(n^3-n)\delta_{n,-m}, \quad [\tilde{\ell}^{\alpha}_n,\tilde{\ell}^{\alpha}_m]=(n-m)\tilde{\ell}^{\alpha}_{n+m}+\frac{c_L}{12}(n^3-n)\delta_{n,-m}.\]
For each $P>0$, the $(\ell_n^{Q+iP})_n$ thus forms a unitary representations of the Virasoro algebra into 
$\mc{H}_{\mc{T}}$, and $(\ell_n^{Q+iP}, \tilde{\ell}_m^{\, Q+iP})_{n,m}$ is a representation of ${\rm Vir}(c_L)\oplus {\rm Vir}(c_L)$ into $(\mc{H}_{\mc{T}}\otimes \mc{H}_{\mc{T}}, \cjg\cdot,\cdot\cjd_{Q+iP}\otimes  \cjg\cdot,\cdot\cjd_{Q+iP})$, with each copy being unitary.

\subsubsection{Liouville descendants}

The Liouville descendant states $\Psi_{\alpha,\nu,\tilde{\nu}}$ are defined using the Poisson operator of Proposition \ref{poissonprop} (unlike in Section \ref{sec:Poisson_Operator}, 
the square root below is now defined so that $\sqrt{Re^{i\theta}}=\sqrt{R}e^{i\theta/2}$ for $R>0$ and $\theta \in (-\pi,\pi)$): 
\begin{proposition}[Proposition 6.26 and Proposition 6.9 of \cite{GKRV}]\label{constPoisson}
For $\nu,\tilde{\nu}\in \mc{T}$, let $\ell:=|\nu|+|\nu'|$, and 
\begin{equation}\label{defWell}
W_\ell:= \Big\{ \alpha=Q+iP \in \C \setminus \mc{D}_0 \, |\,  
{\rm Re}(\alpha)\leq Q, \,\, {\rm Im}(\sqrt{P^2+2\ell})>{\rm Im}(P)-\gamma/2\Big\}
\end{equation}
where $\mc{D}_0=\bigcup_{j\geq 0}\{Q\pm i\sqrt{2j}\}$. Then for $\alpha=Q+iP$ with $\alpha \in W_\ell $
\[ \Psi_{\alpha,\nu,\tilde{\nu}}:=\mc{P}_\ell(Q+i\sqrt{P^2+2\ell})\mc{Q}_{\alpha,\nu,\tilde{\nu}},\]
is well defined, analytic in $\alpha$, as an element in $e^{- \beta \rho}\mc{D}(\mc{Q})$ for all $\beta>Q-{\rm Re}(\alpha)$ and it is continuous in $\alpha$ with at most square root singularities at $\mc{D}_0$. 
Moreover $\Psi_{\alpha,\nu,\tilde{\nu}}$ 
is an eigenfunction of ${\bf H}$ with eigenvalue 
$\frac{Q^2+P^2}{2}+\ell$ if $\alpha=Q+iP$.
The set $W_\ell$ is a connected subset of the half-plane $\{{\rm Re}(\alpha)\leq Q\}$, containing $(Q+i\R)\setminus \mc{D}_0$ and the real half-line $(-\infty,Q-\frac{2\ell}{\gamma}-\frac{\gamma}{4})$. Moreover, there is $C_{|\nu|+|\tilde{\nu}|}<Q$  such that for real valued $\alpha<C_{|\nu|+|\tilde{\nu}|}$,  $\Psi_{\alpha,\nu,\tilde{\nu}}$ is given by 
\begin{equation}\label{defdescendantsbis}
\Psi_{\alpha,\nu,\tilde{\nu}}=\lim_{t \to +\infty}e^{t(2 \Delta_{\alpha}+|\nu|+|\tilde{\nu}|)}e^{-t{\bf H}}\Psi_{\alpha,\nu,\tilde{\nu}}^0
\end{equation}
where the limit holds in a weighted space $e^{-\beta\rho}L^2$ for $\beta>0$ large enough. 
\end{proposition}
Notice that, since ${\bf H}^0\Psi^0_{\alpha,\nu,\tilde{\nu}}=(2\Delta_{Q+ip}+|\nu|+|\nu'|)\Psi^0_{\alpha,\nu,\tilde{\nu}}$, the term in \eqref{defdescendantsbis} corresponds (at least formally) to looking at the large time limit of $e^{-t{\bf H}}e^{t{\bf H}_0}\Psi^0_{\alpha,\nu,\tilde{\nu}}$.
Since $\Psi_{\alpha,\nu,\tilde{\nu}}=e^{(\alpha-Q)c}\mc{Q}_{\alpha,\nu,\tilde{\nu}}$ by \eqref{descendantfree}, the identity \eqref{defdescendantsbis} shows that the Poisson operator can be viewed as a correspondence between the eigenstates $\Psi_{\alpha,\nu,\tilde{\nu}}^0$ of the free field Hamiltonian ${\bf H}^0$ and the eigenstates $\Psi_{\alpha,\nu,\tilde{\nu}}$ of the Liouville Hamiltonian.
We first upgrade the statement \eqref{defdescendantsbis} to a limit in $e^{-\beta \rho}\mc{D}(\mc{Q})$ for $\alpha$ real valued and negative enough and $\beta>Q-\alpha$:
\begin{equation}\label{defdescendantsbisstrong}
\lim_{t\to \infty}\|\Psi_{\alpha,\nu,\tilde{\nu}}-e^{t(\Delta_{\alpha}+|\nu|+|\tilde{\nu}|)}e^{-t{\bf H}}\Psi_{\alpha,\nu,\tilde{\nu}}^0\|_{e^{-\beta \rho}\mc{D}(\mc{Q})}=0
\end{equation} 
 To prove this, we note that for some small $\eps>0$, there is $C_\eps>0$ independent of $t$ such that
 \begin{align*} 
 & \|\Psi_{\alpha,\nu,\tilde{\nu}}-e^{t(\Delta_{\alpha}+|\nu|+|\tilde{\nu}|)}e^{-t{\bf H}}\Psi_{\alpha,\nu,\tilde{\nu}}^0\|_{e^{-\beta\rho}\mc{D}(\mc{Q})} \\
&= \|e^{-\eps{\bf H}}e^{\eps(\Delta_{\alpha}+|\nu|+|\tilde{\nu}|)}(\Psi_{\alpha,\nu,\tilde{\nu}}-e^{(t-\eps)(\Delta_{\alpha}+|\nu|+|\tilde{\nu}|)}e^{-(t-\eps){\bf H}}\Psi_{\alpha,\nu,\tilde{\nu}}^0)\|_{e^{-\beta\rho}\mc{D}(\mc{Q})}\\
 & \leq  C_\eps \|\Psi_{\alpha,\nu,\tilde{\nu}}-e^{(t-\eps)(\Delta_{\alpha}+|\nu|+|\tilde{\nu}|)}e^{-(t-\eps){\bf H}}\Psi_{\alpha,\nu,\tilde{\nu}}^0)\|_{e^{-\beta \rho} L^2}
 \end{align*}
where we used \cite[Lemma 6.5]{GKRV} which states that for all $t>0$ and all $\beta \in \R$
\[
e^{-t{\bf H}}: e^{-\beta \rho}L^2(\R\times \Omega_\T) \to e^{-\beta \rho}\mc{D}(\mc{Q}), \quad e^{-t{\bf H}}: e^{-\beta \rho}\mc{D}'(\mc{Q})\to e^{-\beta \rho}L^2(\R\times \Omega_\T)
\]
are bounded. Note that, writing $e^{-t{\bf H}}=e^{-\frac{t}{2}{\bf H}}e^{-\frac{t}{2}{\bf H}}$, this also implies boundedness of the operator
\begin{equation}\label{boundednesspropag}
e^{-t{\bf H}}: e^{-\beta \rho}\mc{D}'(\mc{Q})\to e^{-\beta \rho}\mc{D}(\mc{Q}).
\end{equation}
\begin{corollary}\label{cor:cts}
Let $\bv$ be a Markovian vector field which satisfies the conditions of Theorem \ref{theoremfreefield}. If $\alpha\ll -1$ real valued and $\beta>Q-\alpha$ then the following limit holds 
\[ \lim_{t\to +\infty}\|\bH_\bv\Psi_{\alpha,\nu,\tilde{\nu}}-e^{t(2\Delta_\alpha+|\nu|+|\tilde{\nu}|)}\bH_\bv e^{-t\bH}\Psi_{\alpha,\nu,\tilde{\nu}}^0\|_{e^{-\beta \rho}\mc{D}'(\mc{Q})}=0.\]
\end{corollary}
\begin{proof}
This is an immediate consequence of \eqref{defdescendantsbisstrong} and the continuity of 
\begin{equation}\label{boundednessHvweight}
\bH_\bv: e^{-\beta \rho}\cD(\cQ)\to e^{-\beta\rho}\cD'(\cQ).
\end{equation}
This last fact can be checked by considering the expression \eqref{defHv}: one has  
${\bf L}_n^0e^{-\beta \rho}= e^{-\beta \rho}({\bf L}_n^0-i\beta\rho'{\bf A}_n)$ (and similarly for $\tilde{{\bf L}}_n^0e^{-\beta \rho}$) and 
 ${\bf L}_n^0,{\bf A}_n$ are bounded as maps $\cD(\cQ)\to \cD'(\cQ)$, while  $e^{\beta \rho}{\bf H}e^{-\beta \rho}=\mathbf{H} -\tfrac{\beta^2}{2}(\rho'(c))^2+\tfrac{\beta}{2} \rho''(c)+\beta \rho'(c)\pl_{c}$ and all terms in the RHS are bounded from $\mc{D}(\mc{Q})$ to $\mc{D}'(\mc{Q})$.
\end{proof}

Next, we show that the action of $\bH_\bv$ on Liouville descendants can be written in terms of the action of ${\bf H}_{\bv}^0$ on descendants of the free field using the limit \eqref{defdescendantsbis}: in other words, the Poisson operator intertwines these actions.
\begin{lemma}\label{prop:intertwine_reps}
Under the assumptions of Corollary \ref{cor:cts}, we have the following limit in $e^{-\beta \rho}\mc{D}'(\mc{Q})$
\[\bH_\bv\Psi_{\alpha,\nu,\tilde{\nu}}=\underset{t\to\infty}{\lim}\, e^{t(2\Delta_\alpha+|\nu|+|\tilde{\nu}|)} e^{-t\bH}\bH^0_{\bv_t}\Psi_{\alpha,\nu,\tilde{\nu}}^0.\]
where $\bv_t=e^tv(e^{-t}z)\del_z$. Similarly, if $\bv$ is a polynomial then 
\[\bH_\bv^*\Psi_{\alpha,\nu,\tilde{\nu}}=\underset{t\to\infty}{\lim}\, e^{t(2\Delta_\alpha+|\nu|+|\tilde{\nu}|)} e^{-t\bH}(\bH^0_{\bv_{-t}})^*\Psi_{\alpha,\nu,\tilde{\nu}}^0.\]
In particular, for all $n\in\Z$,
\begin{align}
&\bL_n\Psi_{\alpha,\nu,\tilde{\nu}}=\underset{t\to\infty}{\lim}\, e^{t(2\Delta_\alpha+|\nu|+|\tilde{\nu}|-n)} e^{-t\bH}\bL_n^0\Psi_{\alpha,\nu,\tilde{\nu}}^0\label{LnPsi}\\
&\tilde{\bL}_n\Psi_{\alpha,\nu,\tilde{\nu}}=\underset{t\to\infty}{\lim}\, e^{t(2\Delta_\alpha+|\nu|+|\tilde{\nu}|-n)} e^{-t\bH}\tilde{\bL}_n^0\Psi_{\alpha,\nu,\tilde{\nu}}^0.\label{tildeLnPsi}
\end{align}
\end{lemma}

\begin{proof}
We begin with the case of a Markovian vector field $\bv$ and we introduce the notation $v_t(z)=e^tv(e^{-t}z)$. First, we remark that both $\bH_\bv e^{-t\bH}$ and $e^{-t\bH}\bH_{\bv_t}^0$ are well-defined as continuous operators $e^{-\beta \rho}\cD(\cQ)\to e^{-\beta \rho}\cD'(\cQ)$ for all $\beta\in \R$ by \eqref{boundednesspropag} and \eqref{boundednessHvweight} since for all $t\geq 0$ the vector field $\bv_t=v_t\pl_z$ satisfies the conditions of Theorem \ref{theoremfreefield}. By definition, we have $\bH_\bv e^{-t\bH}=-\frac{\d}{\d\varepsilon}_{|\varepsilon=0}e^{-\varepsilon\bH_\bv}e^{-t\bH}$. 
By the Markov property of the GFF, we have for $g_{t,\eps}(z)=e^{\varepsilon\bv}e^{t\bv_0}(z)$ and $F\in \mc{C}_{\rm exp}$
\begin{equation}\label{bigexpression}
e^{-\varepsilon\bH_\bv}e^{-t\bH}F(c,\varphi)=|g_{t,\eps}'(0)|^{-\frac{Q^2}{2}}\E_{\varphi}\Big[ F\Big(\Big(c+P\varphi\circ g_{t,\eps}  +X_\D \circ g_{t,\eps}+Q\log \frac{|g_{t,\eps}'|}{|g_{t,\eps}|}\Big)\Big|_{\T}\Big)  e^{- \mu e^{\gamma c} \int_{\D\setminus g_{t,\eps}(\D)}     \frac{e^{\gamma X(x)}} {|x| ^{\gamma Q} } \dd x  }  \Big]. 
\end{equation}
Now, we want to commute $e^{\varepsilon\bv}$ with $e^{t\bv_0}$. Since for each $t\geq 0,z\in \D$, 
both $e^{\varepsilon\bv}e^{t\bv_0}(z)=e^{\varepsilon\bv}(e^{-t}z)$ and 
$e^{t\bv_0}e^{\varepsilon\bv_t}(z)=e^{-t}e^{\varepsilon\bv_t}(z)$ solve the ODE $\pl_\varepsilon u_{\eps,t}(z)=v(u_{\eps,t}(z))$ with initial condition 
$u_{0,t}(z)=e^{-t}z$, we deduce that $e^{\varepsilon\bv}e^{t\bv_0}=e^{t\bv_0}e^{\varepsilon\bv_t} $.  This relation can easily be promoted to propagators
\begin{equation}\label{eq:bch}
e^{-\varepsilon\bH_\bv}e^{-t\bH}= e^{-t\bH}e^{-\varepsilon\bH_{\bv_t}}
\end{equation}
following the proof of the Markov property in the proof of Proposition \ref{continuousprocess}.
By differentiating at $\varepsilon=0$, we get that
\begin{equation} \label{eq:BCH}
 \bH_\bv e^{-t\bH}=e^{-t\bH}\bH_{\bv_t}
 \end{equation} 
 as continuous operators $e^{-\beta \rho}\cD(\cQ)\to e^{-\beta \rho}\cD'(\cQ)$ for all $\beta\in \R$.  
Therefore, by Corollary \ref{cor:cts}, we have the following limit in $e^{-\beta \rho} \cD'(\cQ)$ as $t\to\infty$:
\[
\bH_\bv\Psi_{\alpha,\nu,\tilde{\nu}}
=\underset{t\to\infty}{\lim}\,e^{t(2\Delta_\alpha+|\nu|+|\tilde{\nu}|)}\bH_\bv e^{-t\bH}\Psi_{\alpha,\nu,\tilde{\nu}}^0
=\underset{t\to\infty}{\lim}\,e^{t(2\Delta_\alpha+|\nu|+|\tilde{\nu}|)}e^{-t\bH}\bH_{\bv_t}\Psi_{\alpha,\nu,\tilde{\nu}}^0.
\]
Now, we write $\bH_{\bv_t}=\bH_{\bv_t}^0+\mu e^{\gamma c}V_{\bv_t}$ and we need to show that 
\[\underset{t\to\infty}{\lim}\,e^{t(2\Delta_\alpha+|\nu|+|\tilde{\nu}|)}e^{-t\bH}e^{\gamma c}V_{\bv_t}\Psi_{\alpha,\nu,\tilde{\nu}}^0=0\]
in $e^{-\beta \rho}\mc{D}'(\mc{Q})$.  We can use the probabilistic representation of the semigroup \eqref{FeynmanKac}. Writing $X=X_\D+P\varphi$, $\varphi_t(e^{i\theta})=X(e^{-t+i\theta})-B_t$ and $c_t=c+B_t$, using that $|v_t|\leq C$ for some uniform $C>0$, we have for all $c\in\R$ and $t>0$ 
 \begin{align*}
&\left|\E_\varphi \left[e^{-t\bH}( e^{\gamma c}V_{\bv_t}\Psi_{\alpha,\nu,\tilde{\nu}}^0)\right]\right|\\
&\leq C e^{-\frac{Q^2}{2}t}e^{(\gamma+\alpha-Q)c}\E_\varphi\left[e^{(\alpha-Q)B_t}|\cQ_{\alpha,\nu,\tilde{\nu}}(\varphi_t)|\int_0^{2\pi}e^{\gamma X(e^{-t+i\theta})}\d\theta e^{-\mu e^{\gamma c}\int_{\D_t^c}\frac{e^{\gamma X(x)}}{|x|^{\gamma Q}} \dd x }\right]\\
&=C |1-e^{-t}|^{\frac{\gamma^2}{2}}e^{\frac{t}{2}((\gamma+\alpha-Q)^2-Q^2)}e^{(\gamma+\alpha-Q)c}\int_0^{2\pi}e^{\gamma P\varphi (e^{-t+i \theta})}  \E_\varphi   \left[|\cQ_{\alpha,\nu,\tilde{\nu}}(\varphi_t+c_{t,\theta})|e^{-\mu \int_{\D_t^c}\frac{ e^{\gamma (X(x)+c)}  }{|x|^{\gamma\alpha}|x-e^{-t+i\theta}|^{\gamma^2}}\dd x}\right]\d\theta\\
&\leq C  e^{t(\gamma(\alpha-Q)+\frac{\gamma^2}{2} -2\Delta_\alpha)}e^{(\gamma+\alpha-Q)c} e^{\gamma \sup_{\theta} P\varphi (e^{-t+i \theta})}  \sup_{\theta} \E_\varphi [|\cQ_{\alpha,\nu,\tilde{\nu}}(\varphi_t+c_{t,\theta})|]
\end{align*}
where we used the Girsanov transform in the third line to remove the terms $e^{(\alpha-Q)B_t}$ and $e^{\gamma X_{\D}(e^{-t+i\theta})}$, 
and $c_{t,\theta}(\theta')= \ln \frac{|1-e^{-2t} e^{i (\theta-\theta')} |}{|1- e^{i (\theta-\theta')} |}+t$. Each Fourier coefficient of $c_{t,\theta}$ is uniformly bounded in $t,\theta$ and therefore by expanding in $c_{t,\theta}$ we get that 
\begin{equation*}
 \sup_{\theta} \E_\varphi [|\cQ_{\alpha,\nu,\tilde{\nu}}(\varphi_t+c_{t,\theta})|] \leq \E_\varphi [   G(\varphi_t)   ]
\end{equation*}
where $G$ is a finite sum of the form $\sum_{i} \alpha_i |P_i|$ where $P_i$ are polynomials in the Fourier coefficients. Since $\P_\T$ is stationary for the Ornstein-Uhlenbeck process, we have $\E[G(\varphi_t)^4]=\E[G(\varphi)^4]<\infty$ uniformly in $t$. Hence for $t \geq 1$ and by using Jensen on the conditional expectation $\E_\varphi[\cdot]$
\begin{align*}
\E[   | e^{\gamma \sup_{\theta} P\varphi (e^{-t+i \theta})}  \E_\varphi [G(\varphi_t)]  |^2] & \leq \E[   | e^{4 \gamma \sup_{\theta} P\varphi (e^{-t+i \theta})}  | ]^{1/2} \E[ \E_\varphi [G(\varphi_t)] ^4] ]^{1/2} \\
& \leq \E[   | e^{4 \gamma \sup_{\theta} P\varphi (e^{-t+i \theta})}  | ]^{1/2} \E[ G(\varphi_t)^4 ] ^{1/2}   \\
& \leq C  .
\end{align*}
Moreover, we have $\int_\R e^{2(\gamma+\alpha-Q)c}e^{2\beta\rho(c)}\d c<\infty$ since $\beta>Q-\alpha$ and $\alpha-Q+\gamma<0$ so there is $C>0$ such that
\[e^{t(2\Delta_\alpha+|\nu|+|\tilde{\nu}|)}\|e^{-t\bH}(e^{\gamma c}V_{\bv_t}\Psi_{\alpha,\nu,\tilde{\nu}}^0)\|_{e^{-\beta\rho}L^2}\leq  Ce^{(\gamma(\alpha-Q)+|\nu|+|\tilde{\nu}|)t} \]
which converges to 0 for $\alpha <  Q-\frac{|\nu|+|\tilde{\nu}|}{\gamma} $. This proves that the limit
\[\bH_\bv\Psi_{\alpha,\nu,\tilde{\nu}}=\underset{t\to\infty}{\lim}\, e^{t(2\Delta_\alpha+|\nu|+|\tilde{\nu}|)} e^{-t\bH}\bH^0_{\bv_t}\Psi_{\alpha,\nu,\tilde{\nu}}^0\]
holds in $e^{-\beta \rho}\mc{D}'(\mc{Q})$. 

In fact, the identity \eqref{eq:BCH} holds for all vector field $\bv$ (not just Markovian) which admits a holomorphic extension in a neighorhood of $\D$. Indeed, for such vector field $\bv$, we know that $\omega \bv_0+\bv$ is Markovian for some $\omega>0$ large enough: then \eqref{eq:BCH} applied with $\omega \bv_0+\bv$ gives (by real linearity of $\bv\mapsto \bH_{\bv}$) 
\[ (\omega \bH+\bH_\bv) e^{-t\bH}=e^{-t\bH}(\bH_{\bv_t}+\omega \bH)\]
which shows \eqref{eq:BCH} for $\bv$. We use \eqref{eq:BCH} and the fact that $e^{-t\bH}:\cD'(\cQ)\to\cD(\cQ)$ is continuous to obtain for all $F,G\in\cD(\cQ)$
\[
\langle\bH_{\bv_t}^*e^{-t\bH}F,G\rangle_2
=\langle e^{-t\bH}F,\bH_{\bv_t}G\rangle_2=\langle F,e^{-t\bH}\bH_{\bv_t}G\rangle_2
=\langle F,\bH_\bv e^{-t\bH}G\rangle_2
=\langle\bH_\bv^*F,e^{-t\bH}G\rangle_2
=\langle e^{-t\bH}\bH_\bv^*F,G\rangle_2,
\]
which implies the equality of bounded operators $\cD(\cQ)\to\cD'(\cQ)$
\begin{equation}\label{staridentity}
\bH_{\bv_t}^*e^{-t\bH}=e^{-t\bH}\bH_\bv^*.
\end{equation}
If $\bv$ is a polynomial then we can apply identity \eqref{staridentity} with $\bv_{-t}$ in place of $\bv$ and using similar arguments as in the previous case we deduce that
$\bH_\bv^*\Psi_{\alpha,\nu,\tilde{\nu}}=\underset{t\to\infty}{\lim}\, e^{t(2\Delta_\alpha+|\nu|+|\tilde{\nu}|)} e^{-t\bH}(\bH^0_{\bv_{-t}})^*\Psi_{\alpha,\nu,\tilde{\nu}}^0$  provided that $\alpha <  Q-\frac{N-1+|\nu|+|\tilde{\nu}|}{\gamma} $ where $N$ is the degree of the polynomial (this is due to the fact that $|e^{-t}v(e^{t}z)| \leq e^{(N-1)t}$).

To prove \eqref{LnPsi} and \eqref{tildeLnPsi}, we use 
 \eqref{eq:BCH} with $\bv=-z^{n+1}\pl_z$: since $\bv_t=$we obtain $\bL_{n}e^{-t\bH}=e^{-nt}e^{-t\bH}\bL_{n}$ for all $n\geq0$, and similarly for $\tilde{\bL}_n$. Taking adjoints, we get  $\bL_{-n}e^{-t\bH}=e^{nt}e^{-t\bH}\bL_{-n}$. Repeating the above argument shows that the potential part of $\bL_{-n}$ converges to 0 in the $t\to\infty$ limit, so we obtain
\[\bL_{-n}\Psi_{\alpha,\nu,\tilde{\nu}}=\underset{t\to\infty}{\lim}\,e^{t(2\Delta_\alpha+|\nu|+|\tilde{\nu}|+n)}e^{-t\bH}\bL_{-n}\Psi_{\alpha,\nu,\tilde{\nu}}=\underset{t\to\infty}{\lim}\,e^{t(2\Delta_\alpha+|\nu|+|\tilde{\nu}|+n)}e^{-t\bH}\bL_{-n}^0\Psi_{\alpha,\nu,\tilde{\nu}}\]
and the same holds for \eqref{tildeLnPsi}.
\end{proof}

\begin{proposition}\label{intertwining_for_descendants}
Let $\ell\in \N$, and $\nu,\tilde{\nu}\in\mc{T}$ such that $|\nu|+|\tilde{\nu}|=\ell$. For $n\in \Z$, and $\alpha$ in the set $W_\ell\cap W_{\ell-n}$ defined by \eqref{defWell} and $\alpha\notin Q-\frac{\gamma}{2}\N-\frac{2}{\gamma}\N$, for $\beta>Q-{\rm Re}(\alpha)$, 
we have the following identities in $e^{-\beta\rho}\cD(\cQ)$ 
\begin{equation}\label{eq:intertwine_reps_ctd}
 \bL_{n}\Psi_{\alpha,\nu,\tilde{\nu}}=\sum_{\nu'\in \mc{T}}\ell_{n}^\alpha(\nu,\nu') \Psi_{\alpha,\nu',\tilde{\nu}},\qquad 
 \tilde{\bL}_{n}\Psi_{\alpha,\nu,\tilde{\nu}}=\sum_{\tilde{\nu}\in \mc{T}}\tilde{\ell}_{n}^{\alpha}(\tilde{\nu},\tilde{\nu}')\Psi_{\alpha,\nu,\tilde{\nu}'}
 \end{equation}
 where $\ell_{\alpha,n}(\nu,\nu')$ is defined by \eqref{descendant_general}. 
In particular, if $n>0$, $\nu=(\nu_1,\dots,\nu_k)\in \mc{T}$ and $\tilde{\nu}=(\tilde{\nu}_1,\dots,\tilde{\nu}_{k'})$,
\begin{equation}\label{eq:intertwine_reps_ctd2} 
\forall n\leq \nu_k, \,\,   \bL_{-n}\Psi_{\alpha,\nu,\tilde{\nu}}=\Psi_{\alpha,(\nu,n),\tilde{\nu}}, \qquad 
\forall n\leq \tilde{\nu}_{k'} ,\, \,  \tilde{\bL}_{-n}\Psi_{\alpha,\nu,\tilde{\nu}}=\Psi_{\alpha,\nu,(\tilde{\nu},n)}.
\end{equation}
\end{proposition}
\begin{proof} Consider the vector field ${\bf v}=\omega {\bf v}_0+{\bf v}_n=(-\omega z+z^{n+1})\pl_z=:v(z)\pl_z$ for $n\geq 1$ and $\omega>0$ large enough so that ${\bf v}$ is Markovian. The vector field ${\bf v}_t=e^{t}v(e^{-t})\pl_z$ can be written 
\[ {\bf v}_t= -\omega {\bf v}_0+e^{-nt}{\bf v}_n.\]
By \eqref{defHv} (applied to the free-field case), one has ${\bf H}^0_{{\bf v}_t}=\omega {\bf H}^0+e^{-nt}{\bf L}_n^0$. 
By Lemma \ref{prop:intertwine_reps}, we then get 
\[\begin{split} 
{\bf H}_{\bv} \Psi_{\alpha,\nu,\tilde{\nu}}=& 
\underset{t\to\infty}{\lim}\, e^{t(2\Delta_\alpha+|\nu|+|\tilde{\nu}|)} e^{-t\bH}\bH^0_{\bv_t}\Psi_{\alpha,\nu,\tilde{\nu}}^0\\
=&  \omega(\Delta_{\alpha}+|\nu|+|\tilde{\nu}|) \Psi_{\alpha,\nu,\tilde{\nu}} + 
\underset{t\to\infty}{\lim}\, e^{t(2\Delta_\alpha+|\nu|+|\tilde{\nu}|-n)} \sum_{\nu'\in\mc{T}}\ell_{n}^\alpha(\nu,\nu')\Psi_{\alpha,\nu',\tilde{\nu}}^0\\
=&  \omega(\Delta_{\alpha}+|\nu|+|\tilde{\nu}|) \Psi_{\alpha,\nu,\tilde{\nu}} + \sum_{\nu'\in\mc{T}}\ell_{n}^\alpha(\nu,\nu')\Psi_{\alpha,\nu',\tilde{\nu}}.
\end{split}\]  
Since \eqref{defHv} tells us that $\bH_\bv=\omega \bH+{\bf L}_n$, we see that \eqref{eq:intertwine_reps_ctd} holds for $n\leq -1$ when $\alpha$ is real valued. Using the holomorphic extension of 
$\Psi_{\alpha,\nu,\tilde{\nu}}$ with respect to $\alpha$ in Proposition \ref{constPoisson}, one obtains the result in $W_{\ell}\cap W_{\ell+n}$.

To deal with the ${\bf L}_{-n}$ case with $n>0$, we apply the same type of argument: 
one has $({\bf H}^0_{{\bf v}_{-t}})^*=\omega {\bf H}^0+e^{nt}{\bf L}_{-n}^0$ and 
\[\begin{split} 
{\bf H}_{\bv}^* \Psi_{\alpha,\nu,\tilde{\nu}}=& 
\underset{t\to\infty}{\lim}\, e^{t(2\Delta_\alpha+|\nu|+|\tilde{\nu}|)} e^{-t\bH}(\bH^0_{\bv_{-t}})^*\Psi_{\alpha,\nu,\tilde{\nu}}^0\\
=&  \omega(\Delta_{\alpha}+|\nu|+|\tilde{\nu}|) \Psi_{\alpha,\nu,\tilde{\nu}} + 
\underset{t\to\infty}{\lim}\, e^{t(2\Delta_\alpha+|\nu|+|\tilde{\nu}|+n)} e^{-t\bH}\sum_{\nu'\in\mc{T}}\ell_{n}^\alpha(\nu,\nu')\Psi_{\alpha,\nu'\tilde{\nu}}^0\\
=&  \omega(\Delta_{\alpha}+|\nu|+|\tilde{\nu}|) \Psi_{\alpha,\nu,\tilde{\nu}} + \sum_{\nu'\in\mc{T}}\ell_{n}^\alpha(\nu,\nu')\Psi_{\alpha,\nu',\tilde{\nu}}.
\end{split}\]  
Since ${\bf H}_{\bv}^*=\omega \bH+{\bf L}_{-n}$, we obtain the desired result, since the same argument works exactly the same for  $\tilde{\bf L}_n\Psi_{\alpha,\nu,\nu'}$. 
\end{proof}

\section{Scattering coefficient, analyticity of the eigenstates and Verma modules of Liouville}

In this Section, we compute explicitly the scattering matrix for the Liouville conformal field theory by combining the scattering theory developed in \cite{GKRV}
with the properties and formula of the Virasoro operators for Liouville. This implies analytic extension of the primary and descendant states $\Psi_{\alpha,\nu,\tilde{\nu}}$ to the complex plane $\alpha\in \C$ and completes the analysis performed in \cite{GKRV}.

We can define the scattering matrices as in \cite[Definition 6.23]{GKRV}: 
\begin{definition} 
For each $\ell\in \N_0=\N\cup\{0\}$, the scattering matrix ${\bf S}_\ell(Q+ip):E_\ell\to E_\ell$, defined for $p\in \R\setminus[-\sqrt{2\ell},\sqrt{2\ell}]$ 
is the linear map such that:  $\forall F_j\in \ker ({\bf P}-j)$ with $j\leq \ell$
\begin{equation}\label{scatmatrix}
{\bf S}_\ell(Q+ip)\sum_{j=0}^{\ell}F_j:= \sum_{j=0}^{\ell}F_j^+(Q+ip)
\end{equation}
where $F_j^+(Q+ip)\in \ker ({\bf P}-j)$ are the functions appearing in the asymptotic expansion \eqref{expansionP} of $\mc{P}_\ell(Q+ip)\sum_{j=0}^\ell F_j$. By \cite[Corollary 6.24]{GKRV} the matrices ${\bf S}_0(\alpha)$ admits an meromorphic extension from $(Q+i\R)\setminus \mc{D}_0$ to a neighborhood in $\Sigma$
of the set $\{\alpha \,| \, {\rm Re}(\alpha)\in(Q-\eps,Q], \alpha\notin \mc{D}_0\}$ for some $\eps>0$\footnote{The interested reader can check from the proof of Corollary 6.4 in \cite{GKRV} that in fact ${\bf S}_0(\alpha)$ is holomorphic in 
${\rm Re}(\alpha)\in(Q-\eps,Q]$.}.
\end{definition}
From \eqref{egalitepoissonell}, we also see that for $j<\ell$
\[ {\bf S}_\ell(Q+ip)|_{\ker ({\bf P}-j)}={\bf S}_j(Q+ip)|_{\ker ({\bf P}-j)}.\]
It is convenient to use a change of complex parameter $p=\sqrt{P^2+2\ell}$ and define 
\[ \tilde{\bf S}_\ell(Q+iP):={\bf S}_\ell(Q+i\sqrt{P^2+2\ell}).\]
By \cite[Corollary 6.24]{GKRV}, $\tilde{\bf S}_\ell(Q+iP)$ is meromorphic in 
\[\{P\in \C\,|\, {\rm Im}(P)\in [0,\gamma/4), 
{\rm Im}\sqrt{P^2+2\ell}\geq 0\} \setminus \bigcup_{j\geq 0}\{\pm\sqrt{2j}\}.\]

In particular we see from \eqref{expansionP} that for $P \in \R\setminus \{0\}$, 
$\Psi_{Q+iP,\nu,\tilde{\nu}}$ have asymptotic expansion for $\ell:=|\nu|+|\tilde{\nu}|$
\begin{equation}\label{asymptoticdescendants}
\Psi_{Q+iP,\nu,\tilde{\nu}}(c,\varphi)=e^{iPc}\mc{Q}_{Q+iP,\nu,\tilde{\nu}}(\varphi)+ \sum_{j=0}^\ell \Pi_{\ker ({\bf P}-j)}(\tilde{{\bf S}}_\ell(Q+iP)\mc{Q}_{Q+iP,\nu,\tilde{\nu}})(\varphi)e^{-ic\sqrt{P^2+2(\ell-j)}}+G_{Q+iP,\nu,\tilde{\nu}}(c,\varphi)
\end{equation}
where $G_{Q+iP,\nu,\tilde{\nu}}\in \mc{D}(\mc{Q})$, and they are the unique solution of $({\bf H}-\frac{Q^2+P^2}{2}+\ell)u=0$ with such an asymptotic expansion where the coefficient of $e^{iPc}$ is 
given by $\mc{Q}_{Q+iP,\nu,\tilde{\nu}}(\varphi)$.

\subsection{Computation of the reflection coefficient for primary fields} 
In this section we compute the scattering matrix ${\bf S}_0(\alpha)$, also called reflection coefficient.
 
The primary fields are $\Psi_{\alpha}:=\Psi_{\alpha,0,0}=\mc{P}_0(\alpha)1$, by Proposition \ref{poissonprop} they are analytic in $\{{\rm Re}(\alpha)\leq Q\}\setminus \mc{D}_0$ and satisfy in ${\rm Re}(\alpha)\in (Q-\gamma/4,Q]$ and $c\leq 0$
\begin{equation}\label{asymptpsialpha}
 \Psi_{\alpha}=e^{(\alpha-Q)c}+ e^{(Q-\alpha)c}{\bf S}_0(\alpha)1+G(\alpha),
 \end{equation}
 with $G(\alpha)\in \mc{D}(\mc{Q})$ satisfying for all $0\leq \beta<\gamma/2$
 \[ G(\alpha) =G^1(\alpha)+G^2(\alpha), \quad \textrm{ with }G^1(\alpha)\in e^{\frac{\beta}{2}\rho}L^2(\R\times\Omega_\T), \, \textrm{ and } \Pi_0(G^2(\alpha))=\mathbb{E}[G^2(\alpha)]=0.\]
 Notice that $e^{(Q-\alpha)c}$ is much larger than $G^1(\alpha)$ as $c\to -\infty$ if $\alpha\in (Q-\gamma/4,Q)$, and since $\mathbb{E}[G^2(\alpha)]=0$ 
the term $e^{(Q-\alpha)c}{\bf S}_0(\alpha)1$ can be recovered by considering the asymptotic at $c\to -\infty$ of 
\[ \mathbb{E}[ \Psi_{\alpha}]=e^{(\alpha-Q)c}+ e^{(Q-\alpha)c}{\bf S}_0(\alpha)1+ \mathbb{E}[G^1(\alpha)].\] 
\begin{proposition}\label{coefreflection}
For ${\rm Re}(\alpha)\in (Q-\gamma/2,Q]$, the scattering coefficient $R(\alpha):={\bf S}_0(\alpha)1$ is given by the explicit expression 
\begin{equation}\label{Ralpha} 
R(\alpha)=-\Big(\pi \mu \frac{\Gamma(\frac{\gamma^2}{4})}{\Gamma(1-\frac{\gamma^2}{4})}\Big)^{2\frac{(Q-\alpha)}{\gamma}}\frac{\Gamma(-\frac{\gamma(Q-\alpha)}{2})\Gamma(-\frac{2(Q-\alpha)}{\gamma})}{\Gamma(\frac{\gamma(Q-\alpha)}{2})\Gamma(\frac{2(Q-\alpha)}{\gamma})}.
\end{equation}
\end{proposition}

\begin{proof} 

First we prove the result for $\alpha\in (Q-\gamma/2,Q)$ and then we use the fact that both ${\bf S}_0(\alpha)1$ and the RHS of \eqref{Ralpha} extend analytically in $\{{\rm Re}(\alpha)\in  (Q-\eps,Q]\}\setminus \mc{D}_0$ for some $\eps>0$ small. 
In Section 7.1 of \cite{GKRV}, it is proved that for $ \alpha<Q$ the primary field $\Psi_\alpha$ can be expressed by the probabilistic expression 
\begin{equation}\label{defvalpha}
\Psi_\alpha(c,\varphi) =e^{(\alpha-Q)c} \E_\varphi\Big[ \exp\Big(-\mu e^{\gamma c}\int_{\D} |x|^{-\gamma\alpha }e^{\gamma X(x)}\dd x \Big)\Big];
\end{equation}
this expression is a direct consequence of the Feynman-Kac representation \eqref{FeynmanKac} of $e^{-t{\bf H}}$ with $f_t(z)=e^{-t}z$ and of the representation \eqref{defdescendantsbis} of $\Psi_\alpha$.
Let $\alpha\in (\frac{2}{\gamma},Q)$. To compute the $c\to-\infty$ asymptotic expansion we first note that
\begin{align*}
\int_{\D} |x|^{-\gamma\alpha }  e^{\gamma X(x)}\dd x=\int_\D  |x|^{-\gamma\alpha} e^{\gamma P\varphi(x)}(1-|x|^2)^{\frac{\gamma^2}{2}} M_{\gamma,\D}(\dd x)
\end{align*}
where in this proof $M_{\gamma,\D}$ is the GMC measure with respect to the Dirichlet GFF $X_\D$, i.e. the measure defined by
\begin{equation}\label{GMCdisk}
\begin{gathered}
M_{\gamma,\D} (\dd x):=  \underset{\epsilon \to 0} {\lim}  \; \; e^{\gamma X_{\D,\epsilon}(x)-\frac{\gamma^2}{2} \E[X_{\D,\epsilon}(x)^2]} \dd x
\end{gathered}
\end{equation}
where $X_{\D,\epsilon}= X_\D \ast \theta_\epsilon$ is the mollification of $X_\D$ with an  approximation   $(\theta_\epsilon)_{\epsilon>0}$ of the Dirac mass $\delta_0$. First let us show that the contribution of the integral outside of any disk  
$\D_r=\{|x|\leq r\}$ does not not play a part in the asymptotics.
Using the inequality $1-e^{-\mu x}\leq \mu x$ in the second line below, we get
\begin{align}
|  \E_\varphi [ e^{-\mu e^{\gamma c}\int_{\D} |x|^{-\gamma\alpha }e^{\gamma X(x)} \dd x}]  -\E_\varphi [e^{-\mu e^{\gamma c}\int_{\D_r} |x|^{-\gamma\alpha }e^{\gamma X(x)} \dd x}]| &
\leq    \E_\varphi[1-e^{-\mu e^{\gamma c}\int_{\D\setminus \D_r} |x|^{-\gamma\alpha }e^{\gamma X(x)} \dd x } ] \nonumber
\\
\leq & \mu e^{\gamma c}r^{-\alpha\gamma} \E_\varphi [\int_\D e^{\gamma P\varphi}(1-|x|^2)^{\frac{\gamma^2}{2}}M_{\gamma,\D}(\dd x)] \nonumber \\
\leq & \mu e^{\gamma c}r^{-\alpha\gamma} \int_\D e^{\gamma P\varphi(x)}(1-|x|^2)^{\frac{\gamma^2}{2}}\dd x. \label{approximater}
\end{align}
For fixed $r$, this quantity is $e^{\gamma c}\mc{O}(1)$, where $\mc{O}(1)$ means bounded in $L^2(\Omega_\T)$ as $c\to-\infty$. Therefore we can focus on evaluating what happens inside $\D_r$. We want to make use of Kahane's convexity inequalities (see \cite{review}) to estimate the contribution coming from $\D_r$. Let us set $m_r:=-\inf_{x,x'\in\D_r}\log |1-x\bar x'| $. For $x,x'\in \D_r$ 
\begin{equation}
\ln \frac{1}{|x-x'|}-m_r\leq G_\D(x,x')\leq \log \frac{1}{|x-x'|}.
\end{equation}
Let us consider a centered Gaussian field $\tilde{X}$ with covariance $\E[\tilde{X}(x)\tilde{X}(x')]=\ln \frac{1}{|x-x'|}$ inside $\D_r$, a standard Gaussian random variable  $Z$ and the GFF on the circle $\varphi$, all of them independent of each other under $\P$ (with expectation $\E$). Let  $M_{\gamma,\tilde{X}}$ be the GMC measure of $\tilde{X}$ and set $E_r:= e^{\gamma m_r^{1/2}Z-\frac{\gamma^2m_r}{2}}$.  Using Kahane's inequalities, we get
\begin{equation}\label{eq:def_D}
\begin{aligned}
D_r(c,\varphi)
&:=\E_\varphi [ 1-e^{-\mu e^{\gamma c}\int_{\D_r}|x|^{-\alpha\gamma} e^{\gamma P\varphi(x)}(1-|x|^2)^{\frac{\gamma^2}{2}}M_{\gamma,\tilde{X}}(\dd x)}] \\
&\geq\E_\varphi\Big[1-e^{-\mu e^{\gamma c}E_r\int_{\D_r}|x|^{-\alpha\gamma}e^{\gamma P\varphi(x)}(1-|x|^2)^\frac{\gamma^2}{2}M_{\gamma,\D}(\d x)}\Big]
\end{aligned}
\end{equation}
and 
\begin{align*}
D_r(c,\varphi)
\leq\E_\varphi\Big[1-e^{-\mu e^{\gamma c}\int_{\D_r}|x|^{-\alpha\gamma}e^{\gamma P\varphi(x)}(1-|x|^2)^\frac{\gamma^2}{2}M_{\gamma,\D}(\d x)}\Big].
\end{align*}
Next we show that, in $L^2(\R\times\Omega_\T)$, 
\begin{align}
&\lim_{r \to 0^+} \limsup_{c\to-\infty}e^{2(\alpha-Q)c} D_r(c,\varphi)
\leq-\mu^{\frac{2(Q-\alpha)}{\gamma}} \tfrac{2(Q-\alpha)}{\gamma}\bar{R}(\alpha)\Gamma\big(-\tfrac{2(Q-\alpha)}{\gamma}\big)\label{limsup}\\
&\lim_{r \to 0^+}\liminf_{c\to-\infty}e^{2(\alpha-Q)c} D_r(c,\varphi)
\geq-\mu^{\frac{2(Q-\alpha)}{\gamma}} \tfrac{2(Q-\alpha)}{\gamma}\bar{R}(\alpha)\Gamma\big(-\tfrac{2(Q-\alpha)}{\gamma}\big)\label{liminf}
\end{align}
with $\bar{R}(\alpha)$ the unit volume reflection coefficient (see \cite[Equations (3.10) \& (3.12)]{dozz}). 
We will also show that
\begin{align}
&\underset{r \to 0}{\lim}\underset{c\to-\infty}{\limsup}\, e^{2(\alpha-Q)c}\Big|\E_\varphi\Big[e^{-\mu e^{\gamma c}E_r\int_{\D_r}e^{\gamma P\varphi(x)}\frac{(1-|x|^2)^\frac{\gamma^2}{2}}{|x|^{\alpha\gamma}}M_{\gamma,\D}(\d x)}-e^{-\mu e^{\gamma c}\int_{\D_r}|x|^{-\alpha\gamma}e^{\gamma P\varphi(x)}\frac{(1-|x|^2)^\frac{\gamma^2}{2}}{|x|^{\alpha\gamma}}M_{\gamma,\D}(\d x)}\Big]\Big|=0 \label{laststep}.
\end{align}
With the previous considerations (namely \eqref{approximater}), this yields the desired result.  As the proof of \eqref{limsup} and \eqref{liminf} are quite similar, we will only focus on the first one.

It was proven in \cite[Lemma 3.1]{dozz} that
$$\big|\P\Big(\int_{\D_r}|x|^{-\alpha\gamma} (1-|x|^2)^{\frac{\gamma^2}{2}} M_{\gamma,\tilde{X}}(\dd x)>u\Big)-u^{-\frac{2}{\gamma}(Q-\alpha)}\bar R(\alpha)\big|\leq C_ru^{-\frac{2}{\gamma}(Q-\alpha)-\eta}$$
for some constants $C_r>0$  and $\eta>0$. Therefore, setting $S_r:=\sup_{x\in\D_r}e^{\gamma P\varphi(x)}$ and using for the 
random variable $Y\geq 0$ the relation $\E[1-e^{-aY}]=\int_0^\infty a\P(Y>y)e^{-ay}\,\d y$, we get
\begin{align*}
\E_\varphi\Big[  1-&\exp\big(-\mu e^{\gamma c}\int_{\D_r} |x|^{-\alpha\gamma} e^{\gamma P\varphi(x)}(1-|x|^2)^{\frac{\gamma^2}{2}} M_{\gamma,\tilde{X}}(\dd x))\Big]\\
\leq & \E_\varphi\Big[  1-\exp\big(-\mu e^{\gamma c}S_r\int_{\D_r}|x|^{-\alpha\gamma}  (1-|x|^2)^{\frac{\gamma^2}{2}} M_{\gamma,\tilde{X}}(\dd x)\big)\Big]\\
= &\E_\varphi\Big[\int_0^\infty\P \Big(\int_{\D_r}|x|^{-\alpha\gamma}   (1-|x|^2)^{\frac{\gamma^2}{2}} M_{\gamma,\tilde{X}}(\dd x)>u\Big)\mu e^{\gamma c}S_re^{-u\mu e^{\gamma c}S_r}\,du\Big]\\
\leq &\E_\varphi\Big[\int_0^\infty \Big( u^{-\frac{2}{\gamma}(Q-\alpha)}\bar R(\alpha)+C_ru^{-\frac{2}{\gamma}(Q-\alpha)-\eta}\Big)\mu e^{\gamma c}S_re^{-u \mu e^{\gamma c}S_r }\,du\Big]\\
=&\bar R(\alpha)\Gamma\Big(1-\frac{2}{\gamma}(Q-\alpha)\Big)\mu^{\frac{2}{\gamma}(Q-\alpha)}e^{2(Q-\alpha)c}S_r^{\frac{2}{\gamma}(Q-\alpha)}  +C_r \mu^{\frac{2}{\gamma}(Q-\alpha)+\eta}e^{2(Q-\alpha)c+\gamma\eta c}S_r^{\frac{2}{\gamma}(Q-\alpha)+\eta}.
\end{align*}
As  $\lim_{r\to 0}S_r^{\beta}=1$ in $L^2(\Omega_\T)$ for any fixed $\beta>0$,  we deduce that
\begin{align*}
\limsup_{c\to-\infty}e^{2(\alpha-Q)c}  &\E_\varphi\Big[ 1-\exp\big(-\mu e^{\gamma c}\int_{\D_r}|x|^{-\alpha\gamma} e^{\gamma P\varphi(x)}(1-|x|^2)^{\frac{\gamma^2}{2}}M_{\gamma,\tilde{X}}(\dd x)\big)\Big]\\
\leq &\bar R(\alpha)\Gamma\big(1-\frac{2}{\gamma}(Q-\alpha)\big)\mu^{\frac{2}{\gamma}(Q-\alpha)}.
\end{align*}
To conclude, we need to show \eqref{laststep}.  Along the same lines as previously and using the relation $\E[1-e^{-aY}]=\int_0^\infty a\P(Y>y)e^{-ay}\,\d y$ for a variable $Y \geq 0$ 
\begin{align*}
&\underset{r \to 0}{\lim}\underset{c\to-\infty}{\limsup}\,\Big|\E_\varphi\Big[e^{-\mu e^{\gamma c}E_r\int_{\D_r}|x|^{-\alpha\gamma}e^{\gamma P\varphi(x)}(1-|x|^2)^\frac{\gamma^2}{2}M_{\gamma,\D}(\d x)}-e^{-\mu e^{\gamma c}\int_{\D_r}|x|^{-\alpha\gamma}e^{\gamma P\varphi(x)}(1-|x|^2)^\frac{\gamma^2}{2}M_{\gamma,\D}(\d x)}\Big]\Big|=0.
\end{align*}
Indeed, the previous asymptotic expansions give bounds in terms of $\mu^{\frac{2}{\gamma}(Q-\alpha)}$ so the asymptotic expansions of the two terms in \eqref{laststep} are the same as $r$ goes to $0$ since $E_r$ converges to 1. 
\end{proof}

\subsection{Reflection coefficients for descendants}

Recall from \eqref{defdescendantsbis} and from \cite[Proposition 7.2]{GKRV} the link between $\Psi_{\alpha,\nu,\tilde{\nu}}$ and $\Psi^0_{\alpha,\nu,\tilde{\nu}}$:
\[ \Psi_{\alpha,\nu,\tilde{\nu}}=\lim_{t\to +\infty}e^{t(2\Delta_\alpha+|\nu|+|\tilde{\nu}|)}e^{-t{\bf H}}\Psi^0_{\alpha,\nu,\tilde{\nu}}\]
for all $\alpha< \min(Q-\frac{\gamma}{4}-\frac{2(|\nu|+|\tilde{\nu}|)}{\gamma},Q-\gamma)$, with $\Delta_\alpha=\frac{\alpha}{2}(Q-\frac{\alpha}{2})$ if $|\nu|+|\tilde{\nu}|>0$, while this holds for all $\alpha<Q$ if $\nu=\tilde{\nu}=0$.

Let us define for $P\in \R$ the involution (preserving each $\ker ({\bf P}-\ell)$ for $\ell\geq 0$)
\[ \mc{L}_P: L^2(\Omega)\to L^2(\Omega) , \qquad \mc{L}_P(\mc{Q}_{Q+iP,\nu,\tilde{\nu}}):=\mc{Q}_{Q-iP,\nu,\tilde{\nu}}, \,\, \forall \nu,\tilde{\nu}\in \mc{T};\]
recall that $\mc{Q}_{\alpha,\nu,\tilde{\nu}}$ are polynomials in $x_n,y_n$ variables, defined in \eqref{descendantfree}, and the set 
$(\mc{Q}_{\alpha,\nu,\tilde{\nu}})_{\nu\in \mc{T},\tilde{\nu}\in \mc{T}}$ with $|\nu|+|\tilde{\nu}|=\ell$ forming a 
basis of $\ker ({\bf P}-\ell)$ if $\alpha\notin Q-\frac{\gamma}{2}\N_0-\frac{2}{\gamma}\N_0$.
\begin{theorem}\label{scatteringtheorem}
The scattering matrix for the Liouville conformal field theory is of the form 
\[ \forall \ell\in \N_0,\quad  \tilde{{\bf S}}_\ell (Q+iP)=R(Q+iP)\mc{L}_P\]
where $R(\alpha)$ is the reflection coefficient \eqref{coefreflection}. 
\end{theorem}
\begin{proof}
We shall proceed by induction. For $\ell\in \N_0$,  we call $\mc{I}(\ell)$ the following property: 
for all $\nu,\tilde{\nu}\in \mc{T}$ with $|\nu|+|\nu'|\leq \ell$ and $P\in \R_+\setminus \cup_{j=0}^\infty \{\sqrt{2j}\}$,
\begin{equation}\label{assumptionpsi}
\Psi_{Q+iP,\nu,\tilde{\nu}}(c,\varphi)=e^{iPc}\mc{Q}_{Q+iP,\nu,\tilde{\nu}}(\varphi)+R(Q+iP)e^{-iPc}\mc{Q}_{Q-iP,\nu,\tilde{\nu}}(\varphi)+G_{Q+iP,\nu,\tilde{\nu}}(c,\varphi)
\end{equation}
for some $G_{Q+iP,\nu,\tilde{\nu}}\in \mc{D}(\mc{Q})$.

First, $\mc{I}(0)$ is satisfied by Proposition \ref{coefreflection}.

Next, we assume $\mc{I}_{\ell-1}$ for some $\ell\geq1$ and we want to show $\mc{I}_{\ell}$. For $\nu,\tilde{\nu}\in \mc{T}$ with $|\nu|+|\tilde{\nu}|=\ell$,
$\Psi_{Q+iP,\nu,\tilde{\nu}}$ is of the form 
\[\begin{gathered} 
\Psi_{Q+iP,\nu,\tilde{\nu}}=\Psi_{Q+iP,(\nu',n),\tilde{\nu}}, \textrm{ with } |\nu'|+n=|\nu|, \,\,  \nu'\in \mc{T}, \, \textrm{ or }\\
\Psi_{Q+iP,\nu,\tilde{\nu}}=\Psi_{Q+iP,\nu,(\tilde{\nu}',n)}, \textrm{ with } |\tilde{\nu}'|+n=|\tilde{\nu}|, \,\, \tilde{\nu}'\in \mc{T}
\end{gathered}\]
Since the proof is the same in both cases, let us assume we are in the first situation. By Proposition \ref{intertwining_for_descendants}, we have 
for all  $F\in \mc{C}_{\rm exp}$ 
\begin{equation}\label{egalitedescendant}
 \cjg \Psi_{Q+iP,\nu',\tilde{\nu}},{\bf L}_n  F\cjd_2 =\cjg \Psi_{Q+iP,(\nu',n),\tilde{\nu}}, F\cjd_2.  
\end{equation}
Let us take $P>0$ fixed. Then by \eqref{assumptionpsi} and \eqref{formulaLn}, we have if $F$ is supported in $\{c\leq 0\}$
\begin{equation}\label{splitting}
 \begin{split}
\cjg \Psi_{Q+iP,\nu',\tilde{\nu}},{\bf L}_n F\cjd_{2} =  &\cjg  e^{iPc}\mc{Q}_{Q+iP,\nu',\tilde{\nu}}+R(Q+iP)e^{-iPc}\mc{Q}_{Q-iP,\nu',\tilde{\nu}}, {\bf L}^0_nF\cjd_{2} \\
& +\cjg G_{Q+iP,\nu',\tilde{\nu}},{\bf L}_nF\cjd_{2}  \\
& +  \frac{\mu}{2}\cjg  e^{iPc}\mc{Q}_{Q+ip,\nu',\tilde{\nu}}+R(Q+iP)e^{-iPc}\mc{Q}_{Q-iP,\nu',\tilde{\nu}}, e^{\gamma c} V_n F\cjd_{2}
\end{split}\end{equation}
where $V_n$ is the potential term in \eqref{formulaLn}. 
We consider some particular choice of $F$ depending on a parameter $T>0$, namely
\[ F_T(c,\varphi)=\chi_T(c) e^{iP'c}h(\varphi) , \quad \textrm{ with }h\in E_{\ell}, \quad \chi_T(c)=\chi(c+T)\]
for $P'\in \R$, and  $\chi \in C_0^\infty(\R;[0,1])$ satisfying $\chi=1$ near $0$. We also choose $\tilde{\chi}$ with the same property but with $\tilde{\chi}\chi=\chi$ and we let $\tilde{\chi}_T(c)=\tilde{\chi}(c+T)$.
Note that 
${\bf L}_n^0(e^{iP'c}h(\varphi))= e^{iP'c}h'(\varphi)$ for some $h'\in E_{\ell-n}$.

We will show that, as $T\to \infty$, the first line in the RHS of \eqref{splitting} has an explicit asymptotic behaviour while the second and third line goes to $0$. This will allow us to isolate the scattering coefficient in the expression $\cjg \Psi_{Q+iP,\nu',\tilde{\nu}},{\bf L}_n(\chi_T e^{iP'c}h)\cjd_2$.

We start with the term  $\cjg G_{Q+iP,\nu',\tilde{\nu}},{\bf L}_nF_T\cjd_{L^2(\R^-\times \Omega_\T)}$ and decompose ${\bf L}_n={\bf L}_n^0+\frac{\mu}{2}e^{\gamma c}V_n$.
First, we will show that $\cjg G_{Q+iP,\nu',\tilde{\nu}},e^{\gamma c}V_nF_T\cjd \to 0$.
 Since $\| \chi_Te^{iP'c}h\|_{\mc{D}(\mc{Q})}\leq C$ for some $C>0$ uniform with respect to $T>1$,
then  $\chi_T e^{\gamma c}V_ne^{iP'c}h\in \mc{D}'(\mc{Q})$ with uniform norm with respect to $T>1$.
Using that $\psi G_{Q+iP,\nu',\tilde{\nu}}\in \mc{D}(\mc{Q})$ for any $\psi \in C^\infty(\R;[0,1])$ with support 
in $\R_-$ and bounded derivative, there is $C>0$ such that as $T\to \infty$
\begin{equation}\label{firstboundG}
 |\cjg G_{Q+iP,\nu',\tilde{\nu}},e^{\gamma c}\chi_T V_ne^{iP'c}h\cjd_{2}|\leq C\|\tilde{\chi}_TG_{Q+iP,\nu',\tilde{\nu}}\|_{\mc{D}(\mc{Q})} \|\chi_Te^{iP'c}h\|_{\mc{D}(\mc{Q})}
\end{equation}
where we used the boundedness $e^{\gamma c}V_n:\mc{D}(\mc{Q})\to \mc{D}'(\mc{Q})$.
One has 
\[ \|\chi_Te^{iP'c}h\|_{\mc{D}(\mc{Q})}^2\leq \frac{1}{2}\int_{\R} (|-iP'\chi_T+\chi'_T|^2+(Q^2+2\ell)\chi^2_T) \|h\|_{L^2(\Omega_\T)}^2+\chi_T(c)^2\mc{Q}_{e^{\gamma c}V}(h) \dd c \]
which is uniformly bounded as $T\to \infty$ and 
\[ \begin{split}
\|\tilde{\chi}_TG_{Q+iP,\nu',\tilde{\nu}}\|_{\mc{D}(\mc{Q})}^2\leq &  \frac{1}{2}\int_{\R} (|\tilde{\chi}'_T|^2+Q^2) \|G_{Q+iP,\nu',\tilde{\nu}}\|^2
+ \tilde{\chi}_T^2 \|\pl_c G_{Q+iP,\nu',\tilde{\nu}}\|^2 \dd c\\
& + \int_{\R} \tilde{\chi}_T^2 \|{\bf P}^{\frac{1}{2}}G_{Q+iP,\nu',\tilde{\nu}}\|^2+\tilde{\chi}_T^2 \mc{Q}_{e^{\gamma c}V}(G_{Q+iP,\nu',\tilde{\nu}})\dd c
\end{split}
\]
converges to $0$ as $T\to \infty$ since $G_{Q+iP,\nu',\tilde{\nu}}\in \mc{D}(\mc{Q})$ and  $\supp\, \tilde{\chi}_T\subset [-T-1,-T+1]$. Consequently \eqref{firstboundG} goes to $0$ as $T\to \infty$.

Next we show that $\cjg G_{Q+iP,\nu',\tilde{\nu}},{\bf L}_n^0F_T\cjd_2 \to 0$. First, we compute 
\begin{equation}\label{commutationLnchi} 
[{\bf L}_n,\chi_T]=[{\bf L}^0_n,\chi_T]=  i\chi_T'{\bf A}_n,
\end{equation}
\begin{equation}\label{Ln0Psi_T}
\begin{split} 
 {\bf L}_n^0 (\chi_Te^{iP'c}h) = ie^{iP'c}\chi_T'{\bf A}_nh+ 
 \chi_T {\bf L}_{n}^0(e^{iP'c}h).
 \end{split}
 \end{equation}
 Notice that ${\bf A}_nh\in E_{\ell-n}$ since $n>0$, and that $e^{-iP'c}{\bf L}_{n}^0(e^{iP'c}h)\in E_{\ell-n}$. In particular one can bound 
$\|{\bf L}_n^0 (\chi_Te^{iP'c}h)\|_{L^2(\R\times \Omega_\T)}\leq C$ for some $C>0$ uniform in $T$.  We then have 
 \[  |\cjg G_{Q+iP,\nu',\tilde{\nu}},{\bf L}_n^0F_T\cjd_2|\leq \|\tilde{\chi}_TG_{Q+iP,\nu',\tilde{\nu}}\|_{L^2}\|{\bf L}_n^0F_T\|_{L^2}\leq C\|\tilde{\chi}_TG_{Q+iP,\nu',\tilde{\nu}}\|_{L^2} \to 0\]
 as $T\to 0$, thus the second line of \eqref{splitting} goes to $0$ as $T\to \infty$.
 
The third line of \eqref{splitting} also goes to $0$ as $T\to \infty$: 
\begin{equation}\label{3rdline}
 |\cjg e^{\pm iPc}\mc{Q}_{Q\pm iP,\nu',\tilde{\nu}}, e^{\gamma c}V_n F_T\cjd_{2}|\leq C \Big| \int_{-T-1}^{-T+1}e^{\gamma c}\dd c\Big|\times 
|\cjg V_n h,\mc{Q}_{Q\pm iP,\nu',\tilde{\nu}}\cjd_{L^2(\Omega_\T)}| \to 0
\end{equation}
where $\cjg V_n h,\mc{Q}_{Q+iP,\nu',\tilde{\nu}}\cjd_{L^2(\Omega_\T)}$ makes sense 
since $\mc{Q}_{Q+iP,\nu',\tilde{\nu}},h\in \mc{S}$.

Finally, we deal with the first line of \eqref{splitting}. Since $({\bf L}_n^0)^*={\bf L}_{-n}^0$,
\begin{align*}
 & \cjg  e^{iPc}\mc{Q}_{Q+iP,\nu',\tilde{\nu}}+R(Q+iP)e^{-iPc}\mc{Q}_{Q-iP,\nu',\tilde{\nu}}, {\bf L}^0_n F_{T}\cjd_2=\\
& \cjg \chi _T(c){\bf L}^0_{-n}(e^{iPc}\mc{Q}_{Q+iP,\nu',\tilde{\nu}}+R(Q+iP)e^{-iPc}\mc{Q}_{Q-iP,\nu',\tilde{\nu}}),  e^{iP'c}h\cjd_2.
\end{align*}
We compute for all $P\in \R$
\[
{\bf L}^0_{-n}(e^{\pm iPc}\mc{Q}_{Q\pm iP,\nu',\tilde{\nu}})= e^{\pm iPc}\mc{Q}_{Q\pm iP,(\nu',n),\tilde{\nu}}
\]
so that 
\[\begin{split} 
\cjg  e^{\pm iPc}\mc{Q}_{Q\pm iP,\nu',\tilde{\nu}}, {\bf L}^0_n (\chi_Te^{ iP'c}h)\cjd_2=& \int \chi(c+T)e^{ic(\pm P-P')}\dd c 
\cjg\mc{Q}_{Q\pm iP,(\nu',n),\tilde{\nu}} , h\cjd_{L^2(\Omega_\T)}\\
=&e^{iT(P'\mp P)} \hat{\chi}(P'\mp P)\cjg\mc{Q}_{Q\pm iP,(\nu',n),\tilde{\nu}} , h \cjd_{L^2(\Omega_\T)}
\end{split}\]
where $\hat{\chi}$ denotes the Fourier transform. We therefore have, choosing $P'=-P$ and $\chi$ so that $\int \chi=1$, that as $T\to +\infty$
\begin{equation}\label{asymptoticPsiLnF}
\begin{split}
\cjg \Psi_{Q+iP,\nu',\tilde{\nu}},{\bf L}_n F_T\cjd_{2} = & e^{-2iTP}\hat{\chi}(-2P)\cjg\mc{Q}_{Q+iP,(\nu',n),\tilde{\nu}} , h\cjd_{L^2(\Omega_\T)}\\
& +R(Q+ip)\cjg\mc{Q}_{Q- iP,(\nu',n),\tilde{\nu}} , h\cjd_{L^2(\Omega_\T)}+o(1).
\end{split}
\end{equation}
Next, we use \eqref{asymptoticdescendants} and the same arguments as above to obtain, as $T\to \infty$,
\begin{align*} 
& \cjg \Psi_{Q+iP,(\nu',n),\tilde{\nu}},  F_{T}\cjd_2= \\ 
&\qquad e^{-2iTP}\hat{\chi}(-2P)\cjg\mc{Q}_{Q+iP,(\nu',n),\tilde{\nu}} , h\cjd_{L^2(\Omega_\T)}+ \cjg 
\Pi_{\ker ({\bf P}-\ell)}(\tilde{{\bf S}}_\ell(Q+iP)\mc{Q}_{Q+iP,\nu,\tilde{\nu}}), h\cjd_{L^2(\Omega_\T)}
\\
&\qquad+\sum_{j<\ell} e^{iT(P-P_{j\ell})}\hat{\chi}(P-P_{j\ell}) 
\cjg \Pi_{\ker ({\bf P}-j)}(\tilde{{\bf S}}_\ell(Q+iP)\mc{Q}_{Q+iP,\nu,\tilde{\nu}}), h\cjd_{L^2(\Omega_\T)}+o(1)
\end{align*}
with $P_{j\ell}:=\sqrt{P^2+2(\ell-j)}$.
Combining this asympotic expansion with \eqref{asymptoticPsiLnF} and the fact that we can choose $h$ arbitrary in $E_\ell$,
we deduce that 
\[ \Pi_{\ker ({\bf P}-\ell)}(\tilde{{\bf S}}_\ell(Q+iP)\mc{Q}_{Q+iP,\nu,\tilde{\nu}}) =R(Q+iP) \mc{Q}_{Q-iP,\nu,\tilde{\nu}}.\]
Applying the same arguments but choosing $P'=-P_{j\ell}$ for $j<\ell$, we also get 
\[ \Pi_{\ker ({\bf P}-j)}(\tilde{{\bf S}}_\ell(Q+iP)\mc{Q}_{Q+iP,\nu,\tilde{\nu}})=0.\]
We have thus proved \eqref{assumptionpsi} and therefore $\mc{I}(\ell)$ holds. This ends the proof.
\end{proof}

\subsection{Analyticity and functional equation for the Liouville eigenstates $\Psi_{\alpha,\nu,\tilde{\nu}}$.}

First, let us prove the following:
\begin{lemma}\label{analyticity}
The function $\Psi_{\alpha}$ extends analytically from $\{{\rm Re}(\alpha)<Q\}$ to a neighborhood of 
$\{{\rm Re}(\alpha)\leq Q\}$ in $\C$ and it satisfies the functional equation on ${\rm Re}(\alpha)=Q$
\[ \Psi_{Q-iP}=R(Q-iP)\Psi_{Q+iP}\]
where $R(\alpha)$ is the reflection coefficient \eqref{Ralpha}.
\end{lemma}
\begin{proof} We know from Proposition \ref{poissonprop} that $\Psi_{\alpha}$ admits a meromorphic extension from $\{{\rm Re}(\la)<Q\}\subset \Sigma$
to a open neighborhood $U\subset \Sigma$ of $\{{\rm Re}(\la)<Q\}$, where we recall that $\Sigma$ is the Riemann surface that is a ramified 
 covering of $\C$ of order $2$ associated to the functions $r_j(\alpha)=\sqrt{-(\alpha-Q)^2-2j}$. We will prove that this extension descends to $\pi(U)$ if $\pi:\Sigma\to \C$ as a meromorphic function. We need to check that if $\gamma_j\in G$ is an automorphism of the covering (for $j>0$), then 
 $\gamma_j^*\Psi_\alpha=\Psi_{\alpha}$. Since $\gamma_j^*r_k=-r_k$ for $k\geq j$ and $\gamma_j^*r_k=r_k$ if $k<j$, we see from \eqref{asymptpsialpha} 
 that for $P>0$ and $\alpha=Q+iP$
 \[ \gamma_j^*\Psi_{\alpha}(c,\varphi)=e^{iPc}+e^{-iPc}R(\alpha)+\gamma_j^*G(\alpha)\]
but then 
\[ (\gamma_j^*\Psi_{\alpha}-\Psi_{\alpha})\in \mc{D}(\mc{Q}).\]
Since ${\bf H}$ does not have $L^2$-eigenvalues by \cite[Lemma 6.2]{GKRV}, we deduce that $\gamma_j^*\Psi_{\alpha}=\Psi_{\alpha}$ and, by analytic continuation, this shows that the meromorphic extension of $\Psi_{\alpha}$ in $U$ descends to $\pi(U)$. The functional equation also comes from the uniqueness of $\Psi_{Q+iP}$ having prescribed asymptotic term $e^{iPc}$ and analyticity come from the fact that $R(\alpha)$ has only poles in ${\rm Re}(\alpha)<Q$ by \eqref{Ralpha}.
\end{proof}
This result implies the following  statement on all descendants: 
\begin{theorem}\label{extensionPsialpha}
For each $\nu,\tilde{\nu}\in \mc{T}$, the functions $\Psi_{\alpha,\nu,\tilde{\nu}}$ have an analytic extension to $\C$ as elements in $e^{-\beta \rho}\mc{D}(\mc{Q})$ for $\beta>|{\rm Re}(\alpha)-Q|$. They satisfy the functional equation 
\[ \Psi_{2Q-\alpha,\nu,\tilde{\nu}}=R(2Q-\alpha)\Psi_{\alpha,\nu,\tilde{\nu}}\]
where  $R(\alpha)$ is the reflection coefficient \eqref{Ralpha}. The zeros of $\alpha\mapsto \Psi_{\alpha,\nu,\tilde{\nu}}$ are located exactly at $(Q+\frac{\gamma}{2}\N_0) \cup(Q+\frac{2}{\gamma}\N_0)$ and for ${\rm Re}(\alpha)<Q$ we have 
\[(\Psi_{\alpha,\nu,\tilde{\nu}}-\Psi^0_{\alpha,\nu,\tilde{\nu}})|_{\{c\leq 0\}}\in e^{\min({\rm Re}(\alpha)-Q+\gamma/2,Q-{\rm Re}(\alpha)) c}\mc{D}(\mc{Q}).\]
\end{theorem}
\begin{proof} First, for $\Psi_\alpha$, the analytic extension is given by the functional equation $\Psi_{\alpha}=R(\alpha)\Psi_{2Q-\alpha}$ 
and the fact that $\Psi_\alpha$ is analytic in $\{{\rm Re}(\alpha)\leq Q\}$. Moreover we see that $\Psi_\alpha \not=0$ if $\{{\rm Re}(\alpha)\leq Q\}
\setminus \{Q\}$ 
by considering its $c\to -\infty$ first asymptotic term which follows from the proof of \cite[Proposition 6.21]{GKRV}: 
if $m_\alpha:=\min({\rm Re}(\alpha)-Q+\gamma/2,Q-{\rm Re}(\alpha))$,
\[ \Psi_{\alpha}(c,\varphi)-e^{(\alpha-Q)c} \in  e^{m_\alpha c}\mc{D}(\mc{Q}).\]
The functional equation implies that the only $\alpha$'s where $\Psi_{\alpha}=0$ are at $(Q+\frac{\gamma}{2}\N) \cup(Q+\frac{2}{\gamma}\N)$.
 To deal with the descendants, we proceed by induction. Consider a descendant 
$\Psi_{\alpha,\nu,\tilde{\nu}}$ and assume that the result has been proved for all the descendants $\Psi_{\alpha,\nu',\tilde{\nu}'}$ with 
$|\nu'|<|\nu|$ and $\tilde{\nu}'\leq |\tilde{\nu}|$. We can 
write $\Psi_{\alpha,\nu,\tilde{\nu}}={\bf L}_{-n}\Psi_{\alpha,\nu',\tilde{\nu}}$ if $\nu=(\nu',n)$. For all $\beta>0$, this provides an analytic continuation of 
$\Psi_{\alpha,\nu,\tilde{\nu}}$ as an element in $e^{-\beta\rho}\mc{D}'(\mc{Q})$ in the region $|{\rm Re}(\alpha)-Q|<\beta$ using the induction assumption, 
and they belong to $e^{-\beta \rho}\mc{D}(\mc{Q})$ in the set $W_{|\nu|+|\tilde{\nu}|}\cup (Q-W_{|\nu|+|\tilde{\nu}|})\cup (Q+i\R)$ where 
$W_{|\nu|+|\tilde{\nu}|}$ is the set defined by \eqref{defWell}. Now, by \eqref{boundednesspropag}, we can define 
$e^{-t{\bf H}}\Psi_{\alpha,\nu,\tilde{\nu}}$ for $t\geq 0$ and this belongs to $e^{-\beta \rho}\mc{D}(\mc{Q})$ for $t>0$, but this also 
equal to $e^{-t(2\Delta_{\alpha}+|\nu|+|\tilde{\nu}|)}\Psi_{\alpha,\nu,\tilde{\nu}}$ (for example by differentiating in $t$). This shows that 
$\Psi_{\alpha,\nu,\tilde{\nu}}\in e^{-\beta \rho}\mc{D}(\mc{Q})$ for $|{\rm Re}(\alpha)-Q|<\beta$.
Moreover by the induction assumption, 
\[\Psi_{2Q-\alpha,\nu,\tilde{\nu}}={\bf L}_{-n}\Psi_{2Q-\alpha,\nu',\tilde{\nu}}=R(2Q-\alpha){\bf L}_{-n}\Psi_{\alpha,\nu',\tilde{\nu}}=R(2Q-\alpha)\Psi_{\alpha,\nu,\tilde{\nu}}.\]
Finally, to show that $\Psi_{\alpha,\nu,\tilde{\nu}}\not=0$ when $\alpha\notin (Q+\frac{\gamma}{2}\N)\cup (Q+\frac{2}{\gamma}\N)$, we can do this again by induction: it suffices to consider the half-plane ${\rm Re}(\alpha)\leq Q$ and to prove that
\[ (\Psi_{\alpha,\nu,\tilde{\nu}}(c,\varphi)-e^{(\alpha-Q)c}\mc{Q}_{\alpha,\nu,\tilde{\nu}})|_{\{c\leq 0\}}\in e^{m_\alpha c}\mc{D}(\mc{Q}).\]
We prove this by induction. The first step has been proved already. Assume this is true for $\Psi_{\alpha,\nu',\tilde{\nu}}$. We have, by \eqref{formulaLn}
\[\begin{split} 
\Psi_{\alpha,\nu,\tilde{\nu}}(c,\varphi)=&{\bf L}_{-n}\Psi_{\alpha,\nu',\tilde{\nu}}(c,\varphi)= {\bf L}^0_{-n}e^{(\alpha-Q)c}\mc{Q}_{\alpha,\nu',\tilde{\nu}}+  e^{(\alpha+\gamma-Q)c}F_\alpha (\varphi)+G_{\alpha}(c,\varphi)\\
=& e^{(\alpha-Q)c}\mc{Q}_{\alpha,\nu,\tilde{\nu}}+  e^{(\alpha+\gamma-Q)c}F_\alpha (\varphi)+G_{\alpha}(c,\varphi)
\end{split}\]
where $e^{(\alpha+\gamma-Q)c}F_\alpha \in e^{({\rm Re}(\alpha)+\gamma-Q-\eps)\rho}\mc{D}'(\mc{Q})$ and 
$G_\alpha\in e^{m_\alpha c}\mc{D}'(\mc{Q})$, using as in \eqref{boundednessHvweight} that ${\bf L}_{-n}:e^{\beta \rho}
\mc{D}(\mc{Q})\to e^{\beta \rho}\mc{D}'(\mc{Q})$ for $\beta\in \R$. We finally can apply the propagator $e^{-t{\bf H}}$ to this equation for $t>0$ and by \eqref{boundednesspropag} 
\[ \Psi_{\alpha,\nu,\tilde{\nu}}(c,\varphi)-e^{t(2\Delta_\alpha+|\nu|+|\tilde{\nu}|)}e^{-t{\bf H}}\Psi^0_{\alpha,\nu,\tilde{\nu}}\in 
e^{m_\alpha \rho}\mc{D}(\mc{Q}).\]
In the proof of \cite[Proposition 6.9]{GKRV}, it is shown that 
\[ (e^{-t{\bf H}}\Psi^0_{\alpha,\nu,\tilde{\nu}}-e^{-t(2\Delta_\alpha+|\nu|+|\tilde{\nu}|)}\Psi^0_{\alpha,\nu,\tilde{\nu}})|_{\{c\leq 0\}}\in e^{m_\alpha \rho}\mc{D}(\mc{Q})\]
and this concludes the argument. The same applies in the case $\Psi_{\alpha,\nu,\tilde{\nu}}=\tilde{{\bf L}}_{-n}\Psi_{\alpha,\nu,\tilde{\nu}'}$ with $\tilde{\nu}=(\tilde{\nu}',n)$.
\end{proof}

As a consequence of this result, we can extend the validity of Proposition \ref{intertwining_for_descendants} to all $\alpha\in \C\setminus Q\pm (\frac{\gamma}{2}\N_0+\frac{2}{\gamma}\N_0)$. 

\subsection{Verma modules for the Liouville CFT and representation of Virasoro algebra}\label{VermaLiouville}

Define for $\alpha\notin Q\pm (\frac{2}{\gamma}\N_0+\frac{\gamma}{2}\N_0)$ the vector space (for $\eps>0$)
\[\begin{gathered}
\mc{V}_\alpha:= {\rm span}\{\Psi_{\alpha,\nu,0} \, | \, \nu,\in \mc{T} \}\subset e^{(|{\rm Re}(\alpha)-Q|+\eps)|\rho(c)|}L^2(\R\otimes \Omega_\T),\\
\bbar{\mc{V}}_\alpha:= {\rm span}\{\Psi_{\alpha,0,\tilde{\nu}} \, | \, \tilde{\nu}\in \mc{T} \}\subset e^{(|{\rm Re}(\alpha)-Q|+\eps)|\rho(c)|}L^2(\R\otimes \Omega_\T)
\end{gathered}\]
where, for now, elements are simply finite linear combinations of elements $\Psi_{\alpha,\nu,\tilde{\nu}}$. We emphasize that 
Theorem \ref{extensionPsialpha} and \eqref{assumptionpsi} imply that the 
$\Psi_{\alpha,\nu,\tilde{\nu}'}$ are linearly independent if $\alpha\notin Q-\frac{\gamma}{2}\N_0-\frac{2}{\gamma}\N_0$, since the $\Psi^0_{\alpha,\nu,\tilde{\nu}'}$ are.
We also define 
\[ \mc{W}_\alpha:={\rm span}\{\Psi_{\alpha,\nu,\tilde{\nu}} \, | \, \nu,\tilde{\nu}\in \mc{T} \}\simeq \mc{V}_\alpha\otimes \bbar{\mc{V}}_\alpha \]
where the right identification is done by the map $\Psi_{\alpha,\nu,0}\otimes \Psi_{\alpha,0,\tilde{\nu}}\mapsto \Psi_{\alpha,\nu,\tilde{\nu}}$.
Proposition \ref{intertwining_for_descendants} then implies that ${\bf L}_n$ preserve $\mc{V}_\alpha$, $\tilde{\bf L}_n$ preserves $\bbar{\mc{V}}_\alpha$ and  both preserve $\mc{W}_\alpha$. This space $\mc{W}_\alpha$ can be identified with $\mc{H}_{\mc{T}}\otimes \mc{H}_{\mc{T}}$ by the linear isomorphism  $e_{\nu}\otimes e_{\tilde{\nu}}\mapsto \Psi_{\alpha,\nu,\tilde{\nu}}$
and the maps 
${\bf L}_n$ and $\tilde{\bf L}_n$ are then conjugated to $\ell_n^{\alpha}$ and $\tilde{\ell}_n^{\alpha}$, which implies that the commutations relations of 
the ${\bf L}_n,\tilde{{\bf L}}_m$ are the same as those of the free field Virasoro operators ${\bf L}^0_n,\tilde{{\bf L}}_m^0$:
\[ [\mathbf{L}_n,\mathbf{L}_m]=(n-m)\mathbf{L}_{n+m}+\frac{c_L}{12}(n^3-n)\delta_{n,-m}, \quad [\widetilde{\mathbf{L}}_n,\widetilde{\mathbf{L}}_m]=(n-m)\widetilde{\mathbf{L}}_{n+m}+\frac{c_L}{12}(n^3-n)\delta_{n,-m}\] 
on respectively $\mc{V}_\alpha$ and $\bbar{\mc{V}}_\alpha$, and both also hold on $\mc{W}_\alpha$ where in addition $[{\bf L}_n,\tilde{\bf L}_m]=0$.
 
When $\alpha=Q+iP$ with $P>0$, we define the scalar product on $\mc{W}_{Q+iP}$
\[ \cjg \Psi_{Q+iP,\nu,\tilde{\nu}}, \Psi_{Q+iP,\nu',\tilde{\nu}'} \cjd_{Q+iP}:= F_{Q+iP}(\nu,\nu')F_{Q+iP}(\tilde{\nu},\tilde{\nu}').
\]
By combining Proposition \ref{intertwining_for_descendants},  Theorem \ref{extensionPsialpha} and \cite[Corollary 6.5]{GKRV1}, we obtain the
\begin{theorem}\label{vermamodule}
For $\alpha\notin Q\pm (\frac{2}{\gamma}\N_0+\frac{\gamma}{2}\N_0)$, the operators $({\bf L}_n)_{n\in \Z}$ acting on the vector space $\mc{V}_\alpha$ give a representation of the Virasoro algebra ${\rm Vir}(c_L)$ with central charge $c_L=1+6Q^2$, $\mc{V}_\alpha$ is  
a Verma module associated to the highest weight state $\Psi_{\alpha}$, and this representation is equivalent to the representation 
of $(\ell_n^\alpha)_n$ on $\mc{H}_{\mc{T}}$ given in Section \ref{Vermafreefield}. 
The same holds for $(\tilde{\bf L}_n)_{n\in \Z}$ acting on $\bbar{\mc{V}}_\alpha$. 
Moreover, for $\alpha=Q+iP$ with $P>0$, these representations are unitary representations of the Virasoro algebra and are unitarily equivalent to the representation of $(\ell_n^{Q+iP})_n$ and $(\tilde{\ell}_n^{Q+iP})_n$ on $\mc{H}_T$. The family $({\bf L}_n, \tilde{\bf L}_m)_{n,m\in \Z}$ gives
 a representation of ${\rm Vir}(c_L)\oplus {\rm Vir}(c_L)$ into $\mc{W}_{\alpha}$, whose restriction to both ${\rm Vir}(c_L)\oplus \{0\}$ and $ \{0\}\oplus {\rm Vir}(c_L)$ is unitary when $\alpha=Q+iP$ for $P>0$.
Finally, one has the direct integral decomposition 
\[ L^2(\R\times \Omega_\T)= \frac{1}{2\pi}\int_{0}^\infty \mc{V}_{Q+iP}\otimes \bbar{\mc{V}}_{Q+iP} \, \dd P=\frac{1}{2\pi}\int_{0}^\infty \mc{W}_{Q+iP} \, \dd P\]
 in the sense that for all $u,u'\in L^2(\R\times \Omega_\T)$
 \[  \cjg u,u'\cjd_2=\frac{1}{2\pi} \sum_{\nu,\nu',\tilde{\nu},\tilde{\nu}'\in\mc{T}}\int_0^\infty \cjg u,\Psi_{Q+iP,\nu,\tilde{\nu}}\cjd_2 \cjg \Psi_{Q+iP,\nu',\tilde{\nu}'},u'\cjd_2 F_{Q+iP}^{-1}(\nu,\nu')F_{Q+iP}^{-1}(\tilde{\nu},\tilde{\nu}')\, \dd P.\]
\end{theorem}

\end{document}